\documentclass[11pt,english]{amsart}
\usepackage{amsfonts, amssymb, amsmath, amsthm,eucal,latexsym,nicefrac,mathrsfs}
\usepackage{dsfont, mathrsfs,enumerate,eucal,fontenc}
\usepackage[varg]{txfonts}
\usepackage{xspace}

\usepackage[english]{babel}

\usepackage{pstricks}
\usepackage{color}
\usepackage{graphics} 
\usepackage{psfrag}  % allows to scan text in eps-files, to convert them to latex
\usepackage[all]{xy}

\usepackage{xcolor}
\definecolor{rouge}{rgb}{0.85,0.1,.4}
\definecolor{bleu}{rgb}{0.1,0.2,0.9}
\definecolor{violet}{rgb}{0.7,0,0.8}
\usepackage[colorlinks=true,linkcolor=bleu,urlcolor=violet,citecolor=rouge]{hyperref}

\DeclareMathAlphabet{\mathpzc}{OT1}{pzc}{m}{it}

\theoremstyle{plain}
\newtheorem{theorem}{Theorem}[section]
\newtheorem{lemma}[theorem]{Lemma}
\newtheorem{theo}[theorem]{Theorem}

\newtheorem{coro}[theorem]{Corollary}
\newtheorem{prop}[theorem]{Proposition}

\theoremstyle{definition}
\newtheorem{defi}[theorem]{Definition}

\theoremstyle{remark}
\newtheorem{rema}[theorem]{Remark}
\newtheorem{claim}[theorem]{Claim}

\def\g{{\mathfrak{g}}}  % ground alg de Lie
\def\k{{\Bbbk}}

\def\rg{\ell}               % r = for the rank
\def\r{{\rm reg}}

\def\rs{{\rm reg,ss}}              % rs for the set of regular semisimple elements
\def\poie#1#2#3#4#5#6#7#8#9{\def\un{#5#6#7#8#9}\def\deux{#6#7#8#9}\def\trois{#2#4#8#9}
\def\quatre{#8#9}\def\cinq{#5#6#7}\def\six{#6#7}\def\sept{#2#4}
\ifx\un\empty {#1}_{#2}{#3 \hskip 0.15em}{#1}_{#4} \else \ifx\deux\empty 
{#5}(#1_{#2}){#3 \hskip 0.15em}{#5}(#1_{#4})
\else \ifx\trois\empty {#5}_{#6}(#1){#3 \hskip 0.15em}{#5}_{#7}(#1) 
\else \ifx\quatre\empty {#5}_{#6}(#1_#2){#3 \hskip 0.15em}{#5}_{#7}(#1_#4) 
\else \ifx\cinq\empty {#1}_{#2}^{#8}{#3 \hskip 0.15em}#1_#4^{#9} 
\else \ifx\six\empty {#5}(#1_{#2}^{#8}){#3 \hskip 0.15em}{#5}(#1_{#4}^{#9}) 
\else \ifx\sept\empty {#5}_{#6}(#1)^{#8}{#3 \hskip 0.15em}{#5}_{#7}(#1)^{#9} \else
{#5}_{#6}(#1_{#2}^{#8})^{#9}{#3 \hskip 0.15em}{#5}_{#7}(#1_{#4}^{#8})^{#9} 
\fi \fi \fi \fi \fi \fi \fi}
\def\poi#1#2#3#4#5#6#7{\def\un{#5#6#7}\def\deux{#6#7}
\def\trois{#2#4} \def\cinq{#3#4#5}
\ifx\un\empty {#1}_{#2}{#3 \hskip 0.15em}{#1}_{#4} \else
\ifx\deux\empty {#5}(#1_{#2}){#3 \hskip 0.15em}{#5}(#1_{#4}) \else
\ifx\trois\empty {#5}_{#6}(#1){#3 \hskip 0.15em}{#5}_{#7}(#1) \else
{#5_{#6}}(#1_{#2}){#3 \hskip 0.15em}{#5_{#7}}(#1_{#4}) \fi \fi \fi}
\def\rond{\raisebox{.3mm}{\scriptsize$\circ$}}

\def\tens{\raisebox{.3mm}{\scriptsize$\otimes$}}

\def\dv#1#2{\langle {#1},{#2}\rangle}

\def\tk#1#2{{#2}\otimes _{#1}}

\def\no{n$^{\circ}$}

\def\ec#1#2#3#4#5{\def\un{#3#4#5}\def\deux{#3#5}\def\trois{#3}
\def\four{#2#4#5}\def\five{#2#5}\def\six{#2}\def\seven{#3#4}
\def\eight{#2#4} \def\nine{#2#3#4}
\ifx\nine\empty {\rm #1}_{#5} \else
\ifx\un\empty {\rm #1}({\goth #2}) \else
\ifx\deux\empty {\rm #1}({\goth #2}_{#4}) \else
\ifx\trois\empty {\rm #1}_{#5}({\goth #2}_{#4}) \else
\ifx\four\empty {\rm #1}(#3) \else
\ifx\five\empty {\rm #1}(#3_{#4}) \else
\ifx\six\empty {\rm #1}_{#5}(#3_{#4}) \else
\ifx\seven\empty {\rm #1}_{#5} ({\goth#2})\else
\ifx\eight\empty {\rm #1}_{#5}({#3})
\fi \fi \fi \fi \fi \fi \fi \fi \fi}
\def\hec#1#2#3#4#5{\def\un{#3#4#5}\def\deux{#3#5}\def\trois{#3}
\def\four{#2#4#5}\def\five{#2#5}\def\six{#2}\def\seven{#3#4}
\def\eight{#2#4} \def\nine{#2#3#4}
\ifx\nine\empty \hat{{\rm #1}}_{#5} \else
\ifx\un\empty \hat{{\rm #1}}({\goth #2}) \else
\ifx\deux\empty \hat{{\rm #1}}({\goth #2}_{#4}) \else
\ifx\trois\empty \hat{{\rm #1}}_{#5}({\goth #2}_{#4}) \else
\ifx\four\empty \hat{{\rm #1}}(#3) \else
\ifx\five\empty \hat{{\rm #1}}(#3_{#4}) \else
\ifx\six\empty \hat{{\rm #1}}_{#5}(#3_{#4}) \else
\ifx\seven\empty \hat{{\rm #1}}_{#5} ({\goth#2})  \else
\ifx\eight\empty \hat{{\rm #1}}_{#5}({#3})
\fi \fi \fi \fi \fi \fi \fi \fi \fi}
\def\e#1#2{\ec {#1}#2{}{}{}}
\def\es#1#2{\ec {#1}{}{#2}{}{}}

\def\ai#1#2#3{\def\deux{#2#3} \def\trois{#3} \def\quatre{#2} % invariants
\ifx\deux\empty \es S{{\goth #1}}^{{\goth #1}} \else
\ifx\trois\empty \es S{{\goth #1}^{#2}}^{{\goth #1}^{#2}} \else
\ifx\quatre\empty \es S{{\goth #1}_{#3}}^{{\goth #1}_{#3}} \else
\es S{{\goth #1}_{#3}^{#2}}^{{\goth #1}_{#3}^{#2}} \fi \fi \fi}

\def\Bbb{\mathbb}
\def\goth{\mathfrak}
\def\cal{\mathcal}

\def\gi#1#2#3#4{\def\trois{#3#4} \def\quatre{#4}\def\cinq{#3}\ifx\trois\empty {\rm i}_{#1,{\goth #2}}
\else \ifx\quatre\empty {\rm i}_{#1_{#3},{\goth #2}} \else\ifx\cinq\empty {\rm i}_{#1,{\goth #2}_{#4}} \else {\rm i}_{#1_{#3},{\goth #2}_{#4}} \fi \fi \fi}
\def\j#1#2{\def\deux{#2} \ifx\deux\empty {\rm rk}\hskip .125em{{\goth #1}} \else {\rm rk}\hskip .125em{{\goth #1}_{#2}} \fi}
\def\aj#1#2{\def\deux{#2} \ifx\deux\empty {\rm j}_{{\goth #1}} \else {\rm j}_{{\goth #1}_{#2}} \fi}
\def\an#1#2{\def\deux{#2} \ifx\deux\empty {\cal O}_{#1} \else {\cal O}_{#1,#2} \fi }
\def\han#1#2{\def\deux{#2} \ifx\deux\empty {\hat{{\cal O}}}_{#1} \else {\hat{{\cal O}}}_{#1,#2} \fi }

\def\gg#1#2{{\goth #1}_{#2}\times {\goth #1}_{#2}}
\def\lg#1#2{{\goth {#1}}_{#2}\times {\goth g}}

\def\lp{{\goth l}_{*}\times {\goth p}}
\def\lpq#1#2{{\goth #1}_{#2}\times {\goth #1}}

\def\ll{{\goth {l}}_{*}\times {\goth l}}

\def\sgg#1#2{\es S{\gg {#1}{#2}}}

\def\dim{{\rm dim}\hskip .125em}
\def\dd{{\rm d}}

\def\ad{{\rm ad}\hskip .1em}
   
\def\tr{{\rm tr}\hskip .125em}
\def\det{{\rm det}\hskip .125em}
\def\j#1#2{\ell _{{\goth #1}_{#2}}}

\def\gr{{\rm gr}\hskip .125em}
\def\n{{\rm n}}

\def\s{{\rm s}}
\def\u{{\rm u}}

\def\sx#1#2{{\rm S}^{#1}({\goth #2})}
\def\sy#1#2{{\rm S}^{#1}(#2)}

\def\ex #1#2{\mbox{$\bigwedge^{#1}(#2)$}} 

\def\b#1#2{{\mathrm {b}}_{{\mathfrak{#1}}_{#2}}}

\def\bi#1#2{{\mathrm B}_{{\goth {#1}_{#2}}}}

\def\bk#1#2{{\mathrm C}_{{\goth #1}_{#2}}}
\def\bii{{\mathrm B}}
\def\bjj{{\mathrm E}}
\def\bkk{{\mathrm C}}
\def\oi{\overline{{\mathrm B}}}
\def\ok{\overline{{\mathrm C}}}

\title
[Commuting variety]
{Projective dimension and commuting variety of a reductive Lie algebra}

\author
[J-Y Charbonnel]{Jean-Yves Charbonnel}

\address{Jean-Yves Charbonnel, Universit\'e de Paris - CNRS \\
Institut de Math\'ematiques de Jussieu - Paris Rive Gauche\\
UMR 7586 \\ Groupes, repr\'esentations et g\'eom\'etrie \\
B\^atiment Sophie Germain \\ Case 7012 \\ 
75205 Paris Cedex 13, France}
\email{jean-yves.charbonnel@imj-prg.fr}



\date\today

\begin{document}

\large

\begin{abstract}
The commuting variety of a reductive Lie algebra ${\goth g}$ is the underlying 
variety of a well defined subscheme of $\gg g{}$. In this note, it is proved that 
this scheme is normal and Cohen-Macaulay. In particular, its ideal of definition is a 
prime ideal. As a matter of fact, this theorem results from a so called Property 
$({\bf P})$ for a simple Lie algebra. This property says that some cohomology complexes 
are exact.
\end{abstract}

\maketitle

\setcounter{tocdepth}{1}
\tableofcontents

\section{Introduction} \label{int}
In this note, the base field $\k$ is algebraically closed of characteristic $0$, 
${\goth g}$ is a reductive Lie algebra of finite dimension, $\rg$ is its rank
and $G$ is its adjoint group.    

\subsection{} \label{int1}
The dual of ${\goth g}$ identifies with ${\goth g}$ by a non degenerate symmetric 
bilinear form on ${\goth g}$ extending the Killing form of the derived  algebra of 
${\goth g}$. Denote by $(v,w)\mapsto \dv vw$ this bilinear form and $I_{{\goth g}}$ the 
ideal of $\k[\gg g{}]$ generated by the functions 
$(x,y)\mapsto \dv v{[x,y]}, \, v \in {\goth g}$. The commuting variety 
${\cal C}({\goth g})$ of ${\goth g}$ is the subvariety of elements $(x,y)$ of $\gg g{}$ 
such that $[x,y]=0$. It is the underlying variety of the subscheme ${\cal S}({\goth g})$ 
of $\gg g{}$ defined by $I_{{\goth g}}$. 
It is a well known and long standing open question whether or not this scheme is reduced,
that is ${\cal C}({\goth g})={\cal S}({\goth g})$. According to Richardson~\cite{Ric}, 
${\cal C}({\goth g})$ is irreducible and according to Popov~\cite[Theorem 1]{Po}, the 
singular locus of ${\cal S}({\goth g})$ has codimension at least $2$ in 
${\cal C}({\goth g})$. Then, according to Serre's normality criterion, arises the 
question to know whether or not ${\cal C}({\goth g})$ is normal. There are 
many results about the commuting variety. A result of Dixmier~\cite{Di} proves that 
$I_{{\goth g}}$ contains all the elements of the radical of $I_{{\goth g}}$, of degree 
$1$ in the second variable. In \cite{Ga}, Gan and Ginzburg prove that for 
${\goth g}$ simple of type ${\mathrm {A}}$, the invariant elements under $G$ of 
$I_{{\goth g}}$ is a radical ideal of the algebra $\k[\gg g{}]^{G}$ of invariant elements
of $\k[\gg g{}]$ under $G$. In \cite{Gi}, Ginzburg proves that the normalization of 
${\cal C}({\goth g})$ is Cohen-Macaulay.  

\subsection{Main results} \label{int2}    
According to the identification of ${\goth g}$ and its dual, $\k[\gg g{}]$ is equal to 
the symmetric algebra $\sgg g{}$ of $\gg g{}$. The main result of this note is the 
following theorem:

\begin{theo}\label{tint}
The subscheme of $\gg g{}$ defined by $I_{{\goth g}}$ is Cohen-Macaulay and normal. 
Furthermore, $I_{{\goth g}}$ is a prime ideal of $\sgg g{}$.
\end{theo}
 
According to Richardson's result and Popov's result, it suffices to prove that the scheme 
${\cal S}({\goth g})$ is Cohen-Macaulay. The main idea of the proof in the theorem
uses the main argument of the Dixmier's proof: for a finitely generated module $M$ over 
$\sgg g{}$, $M=0$ if the codimension of its support is at least $l+2$ with $l$ the 
projective dimension of $M$ (see Appendix~\ref{pdc}). 

All the complexes considered in this note are localizations of $\sgg g{}$-submodules of
the $\sgg g{}$-algebra $\tk {\k}{\sgg g{}}\tk {\k}{\e Sg}\ex {}{{\goth g}}$. We
introduce the characteristic submodule of ${\goth g}$, denoted by $\bi g{}$. By
definition, $\bi g{}$ is a $\sgg g{}$-submodule of $\tk {\k}{\sgg g{}}{\goth g}$ and an
element $\varphi $ of $\tk {\k}{\sgg g{}}{\goth g}$ is in $\bi g{}$ if and only if for
all $(x,y)$ in a dense subset of $\gg g{}$, $\varphi (x,y)$ is in the sum of the
subspaces ${\goth g}^{ax+by}$ with $(a,b)$ in $\k^{2}\setminus \{0\}$ and ${\goth g}^{ax+by}$ the centralizer of $ax+by$ in ${\goth g}$. According to a Bolsinov's result,
$\bi g{}$ is a free $\sgg g{}$-module of rank $\b g{}$, the dimension of the Borel
subalgebras of ${\goth g}$. Moreover, the orthogonal complement to $\bi g{}$ in 
$\tk {\k}{\sgg g{}}{\goth g}$ with respect to the $\sgg g{}$-bilinear extension of the
symmetric bilinear form on ${\goth g}$ is a free $\sgg g{}$-module of rank $\b g{}-\rg$
and for $\varphi $ in $\bi g{}$, $\dv {\varphi (x,y)}{[x,y]}=0$ for all $(x,y)$ in
$\gg g{}$. Let $\dd$ be the $\sgg g{}$-derivation of the algebra
$\tk {\k}{\sgg g{}}\ex {}{{\goth g}}$ such that for $v$ in ${\goth g}$, $\dd v$ is the
function on $\gg g{}$: $(x,y)\mapsto \dv v{[x,y]}$. Then
$\tk {\k}{\sgg g{}}\ex {}{{\goth g}}$ inherits a structure of differential graded algebra
with respect to the usual gradation of $\ex {}{{\goth g}}$ and the ideal of 
$\tk {\k}{\sgg g{}}\ex {}{{\goth g}}$ generated by the exterior product
$\ex {\b g{}}{\bi g{}}$ over $\sgg g{}$ is a subcomplex. Denote by
$C_{\bullet}({\goth g})$ the graded subcomplex of $\tk {\k}{\sgg g{}}\ex {}{{\goth g}}$
such that $C_{i+\b g{}}({\goth g}) := \ex {i}{{\goth g}}\wedge \ex {\b g{}}{\bi g{}}$. 
Then Theorem~\ref{tint} is a consequence of the following theorem:

\begin{theo}\label{t2int}
The complex $C_{\bullet}({\goth g})$ has no homology of degree bigger than $\b g{}$ and 
$I_{{\goth g}}$ is isomorphic to the space of boundaries of degree $\b g{}$.
\end{theo}

By standard results of homological algebra (see Appendix~\ref{pdc}), Theorem~\ref{t2int}
is a consequence of the key theorem of this note:

\begin{theo}\label{t3int}
For $i=1,\ldots,\b g{}-\rg$, the $\sgg g{}$-module $C_{i+\b g{}}({\goth g})$ has
projective dimension at most $i$.
\end{theo}

As a matter of fact, it is easy to see that the support of the homology of 
$C_{\bullet}({\goth g})$ is contained in ${\cal C}({\goth g})$. Then, by 
Theorem~\ref{t3int}, $C_{\bullet}({\goth g})$ has no homology of degree bigger than 
$\b g{}$ and $I_{{\goth g}}$ has projective dimension at most $2(\b g{}-\rg)-1$.
So, by Auslander-Buchsbaum's theorem, ${\cal S}({\goth g})$ is Cohen-Macaulay.

For the proof of Theorem~\ref{t3int}, we consider the algebra 
$\tk {\k}{\e Sg}\ex {}{{\goth g}}$ and the $\ex {}{{\goth g}}$-derivation $\dd $ such that
$$ \dd v\tens a = v\wedge a \quad  \text{with} \quad
v \in {\goth g}, \quad a \in \ex {}{{\goth g}}.$$
Then $(\tk {\k}{\e Sg}\ex {}{{\goth g}},\dd)$ is a differential graded algebra with
respect to the usual gradation of $\ex {}{{\goth g}}$. Denote by $D^{\bullet}({\goth g})$
the graded complex so defined. For $k$ nonnegative integer, 
let $D_{k}^{\bullet}({\goth g})$ be the graded subcomplex of $D^{\bullet}({\goth g})$ 
such that
$$ D_{k}^{i}({\goth g}) := \tk {\k}{\sx {k-i}g}\ex {i}{{\goth g}} $$
and $D_{k}^{\bullet}({\goth g},\bi g{})$ the graded subcomplex of 
$\tk {\k}{\sgg g{}}D^{\bullet}({\goth g})$ such that
$$ D_{k}^{i+\b g{}}({\goth g},\bi g{}) := 
\tk {\k}{\sx {k-i}g}\ex {i}{{\goth g}}\wedge \ex {\b g{}}{\bi g{}} .$$

\begin{defi}\label{dint}
Let $n := \b g{}-\rg$. We say that ${\goth g}$ has Property $({\bf P})$ if
$D_{k}^{\bullet}({\goth g},\bi g{})$ has no cohomology of degree different from $\b g{}$
for $k=1,\ldots,n$.
\end{defi}

By an induction argument, Theorem~\ref{t3int} is a consequence of the following Theorem:

\begin{theo}\label{t4int}
Each simple Lie algebra has Property $({\bf P})$.
\end{theo}

As a matter of fact, the proof of this theorem is the main part of this note.

\subsection{Sketch of proofs} \label{int3}
We suppose ${\goth g}$ simple and we prove Theorem~\ref{t4int} by induction on the 
rank of ${\goth g}$. For $k=1,\ldots,n$, we denote by $S_{k}$ the support in $\gg g{}$ of
the cohomology of $D_{k}^{\bullet}({\goth g},\bi g{})$ of degree different from $\b g{}$.
This subset of $\gg g{}$ is invariant under the diagonal action of $G$ and the 
canonical action of ${\mathrm {GL}}_{2}(\k)$ since $\bi g{}$ is an invariant 
module under these actions. As a result, the image of $S_{k}$ by the first projection 
$$ \xymatrix{ \gg g{} \ar[rr]^-{\varpi _{1}} && {\goth g}}$$ 
is a $G$-invariant closed subset of ${\goth g}$. In particular, if $\varpi _{1}(S_{k})$
does not contain semisimple elements different from $0$, $\varpi _{1}(S_{k})$ is 
contained in the nilpotent cone ${\goth N}_{{\goth g}}$ of ${\goth g}$ and $S_{k}$
is contained in the so-called nilpotent bicone ${\cal N}_{{\goth g}}$ of ${\goth g}$. By 
definition, ${\cal N}_{{\goth g}}$ is the subset of elements $(x,y)$ of $\gg g{}$ such 
that the subspace of ${\goth g}$, generated by $x$ and $y$, is contained in 
${\goth N}_{{\goth g}}$. By \cite[Theorem 1.2]{CMo}, ${\cal N}_{{\goth g}}$ has 
codimension $\b g{}+\rg$ in $\gg g{}$. So ${\goth g}$ has Property $({\bf P})$ if
$\varpi _{1}(S_{k})$ is contained in ${\goth N}_{{\goth g}}$ for $k=1,\ldots,n$.  

Fix a Borel subalgebra ${\goth b}$ of ${\goth g}$ and ${\goth h}$ a Cartan subalgebra of
${\goth g}$, contained in ${\goth b}$. Then the main step of the proof of
Theorem~\ref{t4int} is the equality $\varpi _{1}(S_{k})\cap {\goth h}=\{0\}$ for
$k=1,\ldots,n$. Let $z$ be in ${\goth h}$. Denote by ${\goth g}^{z}$ the centralizer of
$z$ in ${\goth g}$. The orbit of $z$ under the Weyl group contains an element $z'$ such
that ${\goth g}^{z'}+{\goth b}$ is a parabolic subalgbera. So, we can suppose that
${\goth p}:={\goth g}^{z}+{\goth b}$ is an algebra. Then ${\goth l}:={\goth g}^{z}$ is
the reductive factor of ${\goth p}$ containing ${\goth h}$. Denote by ${\goth d}$ the
derived algebra of ${\goth l}$ and ${\goth l}_{*}$ the subset of elements $x$ of
${\goth l}$ such that ${\goth g}^{x}$ is contained in ${\goth l}$. Then ${\goth l}_{*}$
is a principal open subset of ${\goth l}$. Let $\poi {{\goth d}}1{,\ldots,}{\n}{}{}{}$ be
the simple factors of ${\goth d}$, ${\goth z}$ the center of ${\goth l}$, 
${\goth p}_{\pm,\u}$ the sum of root spaces with respect to ${\goth h}$, not contained in
${\goth l}$ and $d$ the half dimension of ${\goth p}_{\pm,\u}$. When $z$ is regular, 
$\n=0$. When $\n$ is positive, for $i=1,\ldots,\n$, denote by $\j di$ the rank of 
${\goth d}_{i}$, $\b di$ the dimension of the Borel subalgebras of ${\goth d}_{i}$ and 
set $n_{i} := \b di - \j di$. For $j$ nonnegative integer, set:
$$ {\Bbb I}_{j} := \{(\poi i{-1}{,\ldots,}{\n}{}{}{}) \in {\Bbb N}^{\n+2} \; \vert \;
i_{1} \leq n_{1},\ldots,i_{\n} \leq n_{\n}, \; \poi i{-1}{+\cdots +}{\n}{}{}{} = j\} ,$$
and for $k=1,\ldots,n$ and $\iota =(\poi i{-1}{,\ldots,}{\n}{}{}{})$ in 
${\Bbb I}_{k}$, denote by $D_{k,\iota ,{\goth p}}^{\bullet}({\goth g})$ the total complex 
deduced from the multicomplex
$$ \tk {\k}{D_{i_{-1}}^{\bullet}({\goth p}_{\pm,\u})}\tk {\k}
{D_{i_{0}}^{\bullet}({\goth z})}\tk {\k}{D_{i_{1}}^{\bullet}({\goth d}_{1})}
\tk {\k}{\cdots }D_{i_{\n}}^{\bullet}({\goth d}_{\n}) $$
and set:
$$ D_{k,{\goth p}}^{\bullet}({\goth g}) := \bigoplus _{\iota \in {\Bbb I}_{k}} 
D_{k,\iota ,{\goth p}}^{\bullet}({\goth g}) .$$ 
Denoting by $\an {}{}$ the local ring of $\k[G]$ at the identity, let $\bii$ be the
restriction of $\bi g{}$ to $\lg l*$ and $\widetilde{\bii}$ the 
$\tk {\k}{\an {}{}}\k[\lg l*]$-submodule of
$\tk {\k}{\an {}{}}\tk {\k}{\k[\lg l*]}{\goth g}$ generated by the maps 
$$ (g,x,y) \longmapsto g.\varphi (x,y) \quad  \text{with} \quad \varphi \in \bii.$$
Then $\bii$ and $\widetilde{\bii}$ are free modules of rank $\b g{}$. As usual, let
${\goth m}$ be the maximal ideal of $\an {}{}$ and $\han {}{}$ the completion of
$\an {}{}$ for the ${\goth m}$-adic topology. Set
$\widehat{\bii} := \tk {\an {}{}}{\han {}{}}\widetilde{\bii}$ and denote by
$D_{k,{\goth p}}^{\bullet}({\goth g},\widehat{\bii})$ the graded subcomplex of
$\tk {\k}{\han {}{}}\tk {\k}{\k[\lg l*]}D^{\bullet}({\goth g})$,
$$ D_{k,{\goth p}}^{\bullet}({\goth g},\widehat{\bii}) := 
D_{k,{\goth p}}^{\bullet}({\goth g})[-\b g{}]\wedge \ex {\b g{}}{\widehat{\bii}} .$$
An important step of the proof of Theorem~\ref{t4int} is the following proposition:

\begin{prop}\label{pint}
Suppose that the simple factors of ${\goth l}$ have Property $({\bf P})$. Then, for 
$k=1,\ldots,n$, $D_{k,{\goth p}}^{\bullet}({\goth g},\widehat{\bii})$ has no cohomology
of degree different from $\b g{}$.
\end{prop}

Denote by $D_{k}^{\bullet}({\goth g},\widetilde{\bii})$ the graded complex
$$ D_{k}^{\bullet}({\goth g},\widetilde{\bii}) := 
D_{k}^{\bullet}({\goth g})[-\b g{}]\wedge \ex {\b g{}}{\widetilde{\bii}} .$$
As $\han {}{}$ is a faithfully flat extension of $\an {}{}$, from Proposition~\ref{pint},
Proposition~\ref{prep} and \cite[Theorem 1.1]{Ch}, we deduce that
$D_{k}^{\bullet}({\goth g},\widetilde{\bii})$ has no cohomology of degree
different from $\b g{}$ for $k=1,\ldots,n$. As a matter of fact, \cite[Theorem 1.1]{Ch} is only true for 
simple Lie algebras. The complex $D_{k}^{\bullet}({\goth g},\widetilde{\bii})$ is a 
subcomplex of 
$\tk {\k}{\an {}{}}\tk {\k}{\k[\lg l*]}D_{k}^{\bullet}({\goth g})$ and the morphism 
$$ \xymatrix{ G\times \lg l* \ar[rr] && \gg g{}}, \qquad 
(g,x,y) \longmapsto (g(x),g(y))$$ 
is a flat morphism whose image is the cartesian product of an open neighborhood of $z$ 
and ${\goth g}$. As a consequence, $z$ is not in $\varpi _{1}(S_{k})$ for $k=1,\ldots,n$. 

\bigskip
As a result, one of the main step to prove Theorem~\ref{t4int} is the proof of 
Proposition~\ref{pint}. For that purpose, denote by $\oi$ the restriction of 
$\bii$ to $\lp$ and $\oi_{{\goth l}}$ the restriction of $\bi l{}$ to $\ll$. Then $\oi$ 
and $\oi_{{\goth l}}$ are free modules of rank $\b g{}$ and $\b l{}$ respectively. 
Denoting by ${\goth p}_{\u}$ the nilpotent radical of ${\goth p}$ and ${\goth p}_{-,\u}$ 
the complement to ${\goth p}_{\u}$ in ${\goth p}_{\pm,\u}$, invariant under the adjoint 
action of ${\goth h}$, 
$$ {\goth g} = {\goth p}_{-,\u} \oplus {\goth p} \quad  \text{and} \quad 
{\goth p} = {\goth l} \oplus {\goth p}_{\u}$$
so that $\k[\ll]$ and $\k[\lp]$ are subalgebras of $\k[\lg l*]$. Let $\oi_{+}$ be the 
submodule of the $\k[\lp]$-module $\tk {\k}{\k[\lp]}{\goth g}$ generated by
$\oi_{{\goth l}}$ and ${\goth p}_{\u}$. Then $\oi _{+}$ is a free module of rank $\b g{}$
and $\oi$ is a submodule of $\oi _{+}$ of the same rank. For $M$ free submodule of rank
$\b g{}$ of $\tk {\k}{\k[\lp]}{\goth g}$, denote by
$D_{k,{\goth p}}^{\bullet}({\goth g},M)$ the graded subcomplex of
$\tk {\k}{\k[\lp]}D^{\bullet}({\goth g})$,
$$ D_{k,{\goth p}}^{\bullet}({\goth g},M) := 
D_{k,{\goth p}}^{\bullet}({\goth g})[-\b g{}]\wedge \ex {\b g{}}{M} .$$
By Property $({\bf P})$ for the simple factors of ${\goth l}$, for $k=1,\ldots,n$,
the complexes $D_{k,{\goth p}}^{\bullet}({\goth g},\oi_{+})$ and 
$D_{k,{\goth p}}^{\bullet}({\goth g},\oi)$ have no cohomology of degree different from
$\b g{}$.

Denoting by $J$ the ideal of definition of $\lp$ in $\k[\lg l*]$, let 
$\hat{J}$ be the ideal of $\tk {\k}{\han {}{}}\k[\lg l*]$ generated by
${\goth m}\tens 1$ and  $1\tens J$. For a well defined free $\k[\lp]$-module $\bjj$ of
rank $\b g{}-\rg$ such that
$$ \tk {\k}{\k[\lp]}{\goth p}_{\u} \subset \bjj \subset
\tk {\k[\gg l{}]}{\k[\lp]}\bi l{} + \tk {\k}{\k[\lp]}{\goth p}_{\u} ,$$
we defined a filtration of $D_{k,{\goth p}}({\goth g},\widehat{\bii})$. The main step of
the proof of Proposition~\ref{pint} is the proof of the following result: the associate
graded space to this filtration is isomorphic to
$$ \tk {\k[\lp]}A D_{k,{\goth p}}({\goth g},\oi) \quad  \text{with} \quad 
A = \tk {\k[\lp]}{(\bigoplus _{l\in {\Bbb N}} \hat{J}^{l}/\hat{J}^{l+1})}\ex {}{\bjj} .$$
Then, by the above result on the cohomology of $D_{k,{\goth p}}({\goth g},\oi)$,
$D_{k,{\goth p}}({\goth g},\widehat{\bii})$ has no cohomology of degree different from
$\b g{}$. 

\subsection{Organization of the note.} \label{int5}
In Section~\ref{sc}, the characteristic submodule $\bi g{}$ is introduced and 
some of its properties are given. In particular, its restrictions to parabolic 
subalgebras are considered. In Section~\ref{pm}, we prove that the main theorem
and Theorem~\ref{t2int} result from Theorem~\ref{t3int}. In Section~\ref{co},
we introduce some complexes in a general case and in the particular case of the matter
of the note. Moreover we prove some results in the general case. In Section~\ref{pp},
we define Property $({\bf P})$ and give some consequequences of this property. In
Section~\ref{stp}, we prove that Property $({\bf P})$ for ${\goth g}$ simple results
from Proposition~\ref{pint} for each parabolic subalgebra ${\goth p}$ of ${\goth g}$,
containing ${\goth b}$. In Section~\ref{rcm}, we give some results on the characteristic
module. In Section~\ref{de}, we give some new notations. In particular, we introduce the
module $\bjj$ and we give some results on the modules
$\ex j{\bjj}\wedge \ex {\b g{}}\bii, \, j=1,\ldots,n$. In Section~\ref{sp},
we introduce some spaces and give some important results for the proof
of Proposition~\ref{pint}. In Section~\ref{fc}, we define a filtration of the
complexes $D_{k,{\goth p}}({\goth g},\widehat{\bii}), \, k=1,\ldots,n$ and consider the
associated graded spaces to these filtrations. Then we prove Proposition~\ref{pint} and
Theorem~\ref{t4int}.

The appendix has two sections. In Section~\ref{pdc}, we recall some well known results
of cohomology. In Section~\ref{rep}, we give a property of a ${\goth g}$-submodule of a
rational ${\goth g}$-module under certain hypothesis. This result is used to prove
Theorem~\ref{tstp4}.

\section{General remarks and notations}\label{grem}
The main notations are in the Tables of Notations at the end of the note. For
$k$ positive integer, ${\Bbb N}^{k}$ is ordered by the lexicographic order, induced by
the usual order of ${\Bbb N}$, and its subsets as well. This order is denoted by $\prec$.
As usual $\k^{*} := \k\setminus \{0\}$. In this note, ${\goth g}$ is a reductive Lie
algebra over $\k$ but in many Sections it is supposed to be simple.

$\bullet$ For $V$ a module over a $\k$-algebra, its dual is denoted by $V^{*}$ and
its symmetric and exterior algebras are denoted by $\ec S{}V{}{}$ and $\ex {}V$ 
respectively. For all integer $i$, $\sy iV$ and $\ex iV$ are the spaces of degree $i$
of $\es SV$ and $\ex {}V$ with respect to the usual gradation. In particular, for $i$
negative, $\sy iV$ and $\ex iV$ are equal to $\{0\}$. If $E$ is a subset of $V$, the 
submodule of $V$ generated by $E$ is denoted by span($E$). For $d$ positive integer and
$V$ a vector space over $\k$ of finite dimension, Gr$_d(V)$ is the grassmannian of
subspaces of dimension $d$ of $V$. 

$\bullet$
All topological terms refer to the Zariski topology. If $Y$ is a subset of a topological
space $X$, denote by $\overline{Y}$ the closure of $Y$ in $X$. For $Y$ an open subset
of the algebraic variety $X$, $Y$ is called {\it a big open subset} if the codimension
of $X\setminus Y$ in $X$ is at least $2$. For $Y$ a closed subset of an algebraic 
variety $X$, its dimension is the biggest dimension of its irreducible components and its
codimension in $X$ is the smallest codimension in $X$ of its irreducible components. For 
$X$ an algebraic variety, $\k[X]$ is the algebra of regular functions on $X$, 
$\an X{}$ is its structural sheaf and for $x$ in $X$, $\an Xx$ is the local ring of $X$
at $x$. When $X$ is irreducible, $\k(X)$ is the field of rational functions on $X$. As
usual, for ${\cal F}$ an $\an X{}{}$-module and $Y$ an open subset of $X$,
$\Gamma (Y,{\cal F})$ is the space of sections of ${\cal F}$ over $Y$.

$\bullet$
All the complexes considered in this note are graded complexes over ${\Bbb Z}$
of vector spaces and their differentials are homogeneous of degree $\pm 1$ and they are 
denoted by $\dd $. As usual, the gradation of the complex is denoted by $C_{\bullet}$ 
if the degree of $\dd $ is $-1$ and $C^{\bullet}$ otherwise. 

For $E$ a graded space over ${\Bbb Z}$ and for $i$ integer, $E[i]$ is the graded space 
over ${\Bbb Z}$ whose subspace of degree $n$ is the subspace of degree $n+i$ of $E$. 

$\bullet$
The dual ${\goth g}^{*}$ of ${\goth g}$ identifies with ${\goth g}$ by a given non 
degenerate, invariant, symmetric bilinear form $\dv ..$ on $\gg g{}$
extending the Killing form of $[{\goth g},{\goth g}]$.

$\bullet$ Let ${\goth h}$ be a Cartan subalgebra of ${\goth g}$, ${\goth b}$ a Borel
subalgebra of ${\goth g}$ containing ${\goth h}$, ${\cal R}$ the root system of
${\goth h}$ in ${\goth g}$ and ${\cal R}_{+}$ the positive root system corresponding
to ${\goth b}$. For $\alpha $ in ${\cal R}$, ${\goth g}_{\alpha }$ is the subspace of
weight $\alpha $ of ${\goth g}$ and $x_{\alpha }$ is a generator of the space
${\goth g}_{\alpha }$.

$\bullet$
Let ${\cal P}({\cal R})$ be the set of weights of the root system ${\cal R}$ and 
${\cal P}_{+}({\cal R})$ the subset of dominant weights with respect to ${\cal R}_{+}$.
For $\lambda $ in ${\cal P}_{+}({\cal R})$, denote by $V_{\lambda }$ a simple 
${\goth g}$-module of highest weight $\lambda $.  

$\bullet$
Denote by $\Pi $ the set $\{\poi {\beta }1{,\ldots,}{\rg}{}{}{}\}$ of simple roots of
${\cal R}_{+}$. Let $e$ be the sum of the $x_{\beta }$'s, $\beta $ in $\Pi $, and $h$ the 
element of ${\goth h}\cap [{\goth g},{\goth g}]$ such that $\beta (h)=2$ for all $\beta $
in $\Pi $. The one parameter subgroup of $G$ generated by $\ad h$ is denoted by
$t\mapsto \rho (t)$. Let ${\goth u}$ be the nilpotent radical of ${\goth b}$,
${\goth u}_{-}$ the nilpotent radical of the Borel subalgebra of ${\goth g}$, opposite
to ${\goth b}$, and $B$ the normalizer of ${\goth b}$ in $G$.

\begin{lemma}\label{lint}
Let $O$ be a principal open subset of ${\goth h}$ and $\Sigma $ an irreducible 
hypersurface  of $O+{\goth u}$. Suppose that $\Sigma $ is invariant under
the one-parameter subgroup $t\mapsto \rho (t)$ of $G$. Then $O$ is 
contained in $\Sigma $ or $\Sigma = (\Sigma \cap {\goth h}) + {\goth u}$.
\end{lemma}

\begin{proof}
As $O$ is a principal open subset of ${\goth h}$, $\k[O+{\goth u}]$ is a factorial ring.
Hence $\Sigma $ is the nullvariety in $O+{\goth u}$ of an element of $\k[O+{\goth u}]$.
As a result, $O$ is contained in $\Sigma $ or $\Sigma \cap {\goth h}$ is an
hypersurface of $O$. Suppose that $O$ is not contained in $\Sigma $. For $(x,y)$ in
${\goth h}\times {\goth u}$,
$$ \lim _{t\rightarrow 0} \rho (t).(x+y) = x .$$
Hence $\Sigma \cap {\goth h}$ is the image of $\Sigma $ by the canonical projection 
$\xymatrix{O +{\goth u}\ar[r] & O}$ and $\Sigma $ is contained in 
$\Sigma \cap {\goth h} + {\goth u}$. Moreover, $\Sigma \cap {\goth h}$ is a nonempty
irreducible hypersurface of $O$ as the image of an irreducible subset. Then 
$(\Sigma \cap {\goth h})+{\goth u}$ is an irreducible hypersurface of $O+{\goth u}$
containing $\Sigma $, whence the lemma.
\end{proof}

$\bullet$ 
For $x \in \g$, denote by $x_{\s}$ its semisimple component, $x_{\n}$ its nilpotent 
component and ${\goth g}^{x}$ its centralizer in ${\goth g}$. The set of regular elements 
of $\g$ is 
$$\g_{\r} \ := \ \{ x\in \g \ \vert \ \dim \g^x=\rg \} .$$
We denote by $\g_{\rs}$ the set of regular semisimple elements of $\g$. Then $\g_{\r}$
and $\g_{\rs}$ are $G$-invariant dense open subsets of ${\goth g}$.
According to \cite{Ve}, ${\goth g}\setminus {\goth g}_{\r}$ is equidimensional of 
codimension $3$.  

$\bullet$
Denote by $\poi m1{,\ldots,}{\rg}{}{}{}$ the exponents of the root system ${\cal R}$.
Let $\poi p1{,\ldots,}{\rg}{}{}{}$ be a generating family of the $\k$-algebra $\e Sg^{G}$
such that $p_{i}$ is homogeneous of degree $d_{i}=m_{i}+1$ for $i=1,\ldots,\rg$.
For $(x,y)\in\g \times \g$, consider a shift of $p_i$ in direction $y$: $p_i(x+ty)$ with
$t\in\k$. Expanding $p_i(x+ty)$ as a polynomial in $t$, we obtain
\begin{eqnarray}\label{eq:pi}
p_i(x+ty)=\sum\limits_{m=0}^{d_i} p_{i}^{(m)} (x,y) t^m,  && \forall
(t,x,y)\in\k\times\g\times\g
\end{eqnarray}
where $y \mapsto (m!)p_{i}^{(m)}(x,y)$ is the derivative at $x$ of $p_i$ at the order
$m$ in the direction $y$. The elements $p_{i}^{(m)}$ defined by~(\ref{eq:pi}) are
invariant elements of $\sgg g{}$ under the diagonal action of $G$ in
$\gg g{}$. Remark that $p_i^{(0)}(x,y)=p_i(x)$ while $p_i^{(d_i)}(x,y)=p_i(y)$  for 
all $(x,y)$ in $\g\times \g$.

\begin{rema}\label{rint}
The family
$\mathcal{P}_x  :=
\{p_{i}^{(m)}(x,.); \ 1 \leq i \leq \rg, 0 \leq m \leq d_i  \}$ for $x\in\g$,
is a Poisson-commutative family of $\e Sg$ by Mishchenko-Fomenko~\cite{MF}.
One says that the family $\mathcal{P}_x$ is constructed by
the \emph{argument shift method}.
\end{rema}

$\bullet$
Let $i \in\{1,\ldots,\rg\}$. For $x$ in $\g$, denote by $\varepsilon _i(x)$ the 
element of $\g$ given by
$$ \dv {\varepsilon _{i}(x)}y := \frac{\dd }{\dd t} p_{i}(x+ty) \left \vert _{t=0} \right.
= p_{i}^{(1)}(x,y) .$$
for all $y$ in ${\goth g}$. Thereby, $\varepsilon _{i}$ is an invariant element of 
$\tk {\k}{\e Sg}\g$ under the canonical action of $G$. According to 
\cite[Theorem 9]{Ko}, for $x$ in ${\goth g}$, $x$ is in ${\goth g}_{\r}$ if and only if 
$\poi x{}{,\ldots,}{}{\varepsilon }{1}{\rg}$ are linearly independent. In this case, 
$\poi x{}{,\ldots,}{}{\varepsilon }{1}{\rg}$ is a basis of ${\goth g}^{x}$. 

Denote by $\varepsilon _{i}^{(m)}$, for $0\leq m\leq d_i-1$, the elements of 
$\tk{\k}{\sgg g{}}{\goth g}$ defined by the equality:
\begin{eqnarray}\label{eq:phi}
\varepsilon _i (x+ty) =\sum\limits_{ m=0}^{d_i-1} t^m \varepsilon _i^{(m)}(x,y)  , &&
\forall (t,x,y)\in\k\times\g\times\g
\end{eqnarray}
and set:
$$ V_{x,y} := 
{\mathrm {span}}(\{\varepsilon _{i}^{(m)}(x,y), \ 1 \leq i \leq \rg,
0 \leq m \leq d_i - 1\}) $$ 
for $(x,y)$ in $\gg g{}$. 

$\bullet$ Let ${\goth N}_{{\goth g}}$ be the nilpotent cone of ${\goth g}$. For $(x,y)$
in $\gg g{}$, denote by $P_{x,y}$ the subspace of ${\goth g}$ generated by $x$ and $y$. 
Let ${\cal N}_{{\goth g}}$ be the nilpotent bicone of ${\goth g}$. By definition, 
${\cal N}_{{\goth g}}$ is the subset of elements $(x,y)$ of $\gg g{}$ such that
$P_{x,y}$ is contained in ${\goth N}_{{\goth g}}$. In particular, ${\cal N}_{{\goth g}}$
is invariant under the diagonal action of $G$ in $\gg g{}$ and the canonical action of 
${\mathrm {GL}}_{2}(\k)$ in $\gg g{}$.

\section{Characteristic module} \label{sc}
Set:
$$ I_{0} := \{(i,m)\in \{1,\ldots,\ell\}\times {\Bbb N} \, \vert \,
0\leq m \leq m_{i} \} \quad  \text{and} \quad
V'_{x,y} = \sum_{(a,b) \in \k^{2}\setminus \{0\}} {\goth g}^{ax+by} $$
for $(x,y)$ in $\gg g{}$. By definition, the characteristic module $\bi g{}$ of
${\goth g}$ is the submodule of elements $\varphi $ of $\tk {\k}{\sgg g{}}{\goth g}$ such
that $\varphi (x,y)$ is in $V'_{x,y}$ for all $(x,y)$ in a dense subset of $\gg g{}$. In
this section, some properties of $\bi g{}$ are given.

\subsection{First properties of \texorpdfstring{$\bi g{}$}{}.} \label{sc1}
Denote by $\Omega _{{\goth g}}$ the subset of elements $(x,y)$ of $\gg g{}$ such that
$P_{x,y}$ has dimension $2$ and $P_{x,y}\setminus \{0\}$ is contained in 
${\goth g}_{\r}$. According to~\cite[Corollary 10]{CMo}, $\Omega _{{\goth g}}$ is a big 
open subset of $\gg g{}$. We recall that for $(x,y)$ in $\gg g{}$, $V_{x,y}$ is the
subspace of ${\goth g}$ generated by
$\varepsilon _{i}^{(m)}(x,y), \ (i,m) \in I_{0} $. 

\begin{prop}\label{psc1}
Let $(x,y)$ be in $\gg g{}$ such that $P_{x,y}\cap {\goth g}_{\r}$ is not empty.  

{\rm (i)} Let $O$ be a nonempty open subset of $\k^{2}$ such that $ax+by$ is in
${\goth g}_{\r}$ for all $(a,b)$ in $O$. Then $V_{x,y}$ is the sum of the
${\goth g}^{ax+by}, \, (a,b) \in O$.

{\rm (ii)} The spaces $[x,V_{x,y}]$ and $[y,V_{x,y}]$ are equal.

{\rm (iii)} The space $V_{x,y}$ has dimension at most $\b g{}$ and the equality holds if 
and only if $(x,y)$ is in $\Omega _{{\goth g}}$.

{\rm (iv)} The space $[x,V_{x,y}]$ is orthogonal to $V_{x,y}$. Furthermore, $(x,y)$ is 
in $\Omega _{{\goth g}}$ if and only if $[x,V_{x,y}]$ is the orthogonal complement to 
$V_{x,y}$ in ${\goth g}$.

{\rm (v)} The space $V_{x,y}$ is contained in $V'_{x,y}$. Moreover, $V_{x,y}=V'_{x,y}$ 
if $(x,y)$ is in $\Omega _{{\goth g}}$.

{\rm (vi)} For $(i,m)$ in $I_{0}$, $\varepsilon _{i}^{(m)}$ is a $G$-equivariant map.
\end{prop}

\begin{proof}
(i) For $z$ in ${\goth g}_{\r}$, $\poi z{}{,\ldots,}{}{\varepsilon }{1}{\rg}$ is a basis 
of ${\goth g}^{z}$ by~\cite[Theorem 9]{Ko}. Hence ${\goth g}^{ax+by}$ is contained in 
$V_{x,y}$ for all $(a,b)$ in $O$ since the maps 
$\poi {\varepsilon }1{,\ldots,}{\rg}{}{}{}$ are homogeneous. For pairwise different 
elements $\poi t{i,0}{,\ldots,}{i,d_{i}-1}{}{}{}$, $i=1,\ldots,\rg$ of $\k$, 
$\varepsilon _{i}^{(m)}(x,y)$ is a linear combination of 
$\varepsilon _{i}(x+t_{i,j}y), \; j=0,\ldots,d_{i}-1$ for $m=0,\ldots,d_{i}-1$. We can 
choose $\poi t{i,0}{,\ldots,}{i,d_{i}-1}{}{}{}$ so that
$(a_{i},a_{i}t_{i,0}),\ldots,(a_{i},a_{i}t_{i,d_{i}-1}))$ are in $O$ for some $a_{i}$ in 
$\k^{*}$, whence the assertion since the maps $\poi {\varepsilon }1{,\ldots,}{\rg}{}{}{}$
are homogeneous.

(ii) Let $O$ be an open subset of ${\k^{*}}^{2}$ such that 
$ax+by$ is in ${\goth g}_{\r}$ for all $(a,b)$ in $O$. For all $(a,b)$ in $O$, 
$[x,{\goth g}^{ax+by}]=[y,{\goth g}^{ax+by}]$ since $[ax+by,{\goth g}^{ax+by}]=0$ and 
$ab\neq 0$, whence the assertion by (i).  

(iii) According to~\cite[Ch. V, \S 5, Proposition 3]{Bou}, 
$$ \poi d1{+\cdots +}{\rg}{}{}{} = \b g{} .$$
So $V_{x,y}$ has dimension at most $\b g{}$. By \cite[Theorem 2.1]{Bol}, $V_{x,y}$
has dimension $\b g{}$ if and only if $(x,y)$ is in $\Omega _{{\goth g}}$.

(iv) According to \cite[Theorem 2.1]{Bol}, $V_{x,y}$ is a totally isotropic subspace 
with respect to the skew bilinear form on ${\goth g}$
$$ (v,w) \longmapsto \dv {ax+by}{[v,w]}$$
for all $(a,b)$ in $\k^{2}$. As a result, by invariance of $\dv ..$, $V_{x,y}$ is 
orthogonal to $[x,V_{x,y}]$. If $(x,y)$ is in $\Omega _{{\goth g}}$, ${\goth g}^{x}$ has 
dimension $\rg$ and it is contained in $V_{x,y}$. Hence, by (iii),
$$ \dim [x,V_{x,y}] = \b g{}-\rg = \dim {\goth g} - \dim V_{x,y}$$
so that $[x,V_{x,y}]$ is the orthogonal complement to $V_{x,y}$ in ${\goth g}$. 
Conversely, if $[x,V_{x,y}]$ is the orhogonal complement to $V_{x,y}$ in ${\goth g}$, 
then
$$ \dim V_{x,y} + \dim [x,V_{x,y}] = \dim {\goth g} .$$
Since $P_{x,y}\cap {\goth g}_{\r}$ is not empty, ${\goth g}^{ax+by}\cap V_{x,y}$ has 
dimension $\rg$ for all $(a,b)$ in a dense open subset of $\k^{2}$. By continuity,
${\goth g}^{x}\cap V_{x,y}$ has dimension at least $\rg$ so that 
$$ 2\dim V_{x,y} - \rg \geq \dim {\goth g} .$$
Hence, by (iii), $(x,y)$ is in $\Omega _{{\goth g}}$.

(v) By (i), $V_{x,y}\subset V'_{x,y}$. Suppose that $(x,y)$ is in $\Omega _{{\goth g}}$. 
According to \cite[Theorem 9]{Ko}, for all $(a,b)$ in $\k ^{2}\setminus \{0\}$, 
$\poi {ax+by}{}{,\ldots,}{}{\varepsilon }{1}{\rg}$ is a basis of ${\goth g}^{ax+by}$. 
Hence ${\goth g}^{ax+by}$ is contained in $V_{x,y}$, whence the assertion. 

(vi) Let $i$ be in $\{1,\ldots,\rg\}$. Since $p_{i}$ is $G$-invariant, 
$\varepsilon _{i}$ is a $G$-equivariant map. As a result, its $2$-polarizations 
$\poie {\varepsilon }i{,\ldots,}{i}{}{}{}{(0)}{(d_{i}-1)}$ are $G$-equivariant under the
diagonal action of $G$ in $\gg g{}$. 
\end{proof}

A part of the following theorem is ~\cite[Theorem 11]{CMo}.

\begin{theo}\label{tsc1}
{\rm (i)} The module $\bi g{}$ is a free module of rank $\b g{}$ whose a basis is the 
sequence $\varepsilon _{i}^{(m)}, \, (i,m) \in I_{0}$.

{\rm (ii)} For $\varphi $ in $\tk {\k}{\sgg g{}}{\goth g}$, $\varphi $ is in 
$\bi g{}$ if and only if $p\varphi \in \bi g{}$ for some $p$ in $\sgg g{}\setminus \{0\}$.

{\rm (iii)} For all $\varphi $ in $\bi g{}$ and for all $(x,y)$ in $\gg g{}$, 
$\varphi (x,y)$ is orthogonal to $[x,y]$.
\end{theo}

\begin{proof}
(i) and (ii) According to Proposition~\ref{psc1}(v), $\varepsilon _{i}^{(m)}$ is in 
$\bi g{}$ for all $(i,m)$. Moreover, according to Proposition~\ref{psc1}(iii), these 
elements are linearly independent over $\sgg g{}$. Let $\varphi $ be an element of 
$\tk {\k}{\sgg g{}}{\goth g}$ such that $p\varphi $ is in $\bi g{}$ for some $p$
in $\sgg g{}\setminus \{0\}$. Then $\varphi (x,y)$ is in $V_{x,y}$ for all $(x,y)$ in a 
dense open subset of $\Omega _{{\goth g}}$ by Proposition~\ref{psc1}(v). According to 
Proposition~\ref{psc1}(iii), the map
$$\xymatrix{\Omega _{{\goth g}} \ar[rr] && \ec {Gr}g{}{}{\b g{}}}, \qquad
(x,y) \longmapsto V_{x,y} $$
is regular. So $\varphi (x,y)$ is in $V_{x,y}$ for all $(x,y)$ in $\Omega _{{\goth g}}$.

\begin{claim}\label{clsc1}
Let $X$ be an irreducible variety, $V$ a vector space of dimension $m$ and
$\poi {\lambda }1{,\ldots,}{k}{}{}{}$ in $\tk {\k}{\k[X]}V$ such that
$\poi x{}{,\ldots,}{}{\lambda }{1}{k}$ are linearly independent for all $x$ in $X$.
Denote by $M$ the $\k[X]$-submodule of $\tk {\k}{\k[X]}V$ generated by
$\poi {\lambda }1{,\ldots,}{k}{}{}{}$. For $\varphi $ in $\tk {\k}{\k[X]}V$, $\varphi $
is in $M$ if and only if $\varphi (x)$ is in
${\mathrm {span}}_{\k}(\{\poi x{}{,\ldots,}{}{\lambda }{1}{k}\})$ for all $x$ in $X$.
\end{claim}

\begin{proof}{[Proof of Claim~\ref{clsc1}]}
The condition is clearly necessary. Suppose
$$ \varphi (x) \in {\mathrm {span}}_{\k}(\{\poi x{}{,\ldots,}{}{\lambda }{1}{k}\}) \qquad
\forall x \in X.$$
Then
$$\varphi = a_{1}\lambda _{1} + \cdots + a_{k}\lambda _{k}$$
for some $\poi a1{,\ldots,}{k}{}{}{}$ in $\k(X)$. The variety $X$ has a cover by affine
open subsets $Y$ such that for some $\poi v1{,\ldots,}{n-k}{}{}{}$ in $V$,
$\poi v1{,\ldots,}{n-k}{}{}{},\poi x{}{,\ldots,}{}{\lambda }{1}{k}$ is a basis of $V$
for all $x$ in $Y$. Then
$\poi v1{,\ldots,}{n-k}{}{}{},\poi {\lambda }1{,\ldots,}{k}{}{}{}$ is a basis of the
$\k[Y]$-module $\tk {\k}{\k[Y]}V$. As a result, $\poi a1{,\ldots,}{k}{}{}{}$ are in
$\k[Y]$, whence the claim since $X$ is covered by such open subsets $Y$.
\end{proof}

By Claim~\ref{clsc1}, for some regular functions $a_{i,m}, \, (i,m) \in I_{0}$ on
$\Omega _{{\goth g}}$,
$$ \varphi (x,y) = \sum_{(i,m) \in I_{0}} a_{i,m}(x,y)\varepsilon _{i}^{(m)}(x,y) $$
for all $(x,y)$ in $\Omega _{{\goth g}}$. Since $\Omega _{{\goth g}}$ is a big open subset
of $\gg g{}$ and $\gg g{}$ is normal, the $a_{i,m}$'s have a regular extension to 
$\gg g{}$. Hence $\varphi $ is a linear combination of the 
$\varepsilon _{i}^{(m)}$'s with coefficients in $\sgg g{}$. As a result, the sequence 
$\varepsilon _{i}^{(m)}, \, (i,m) \in I_{0}$ is a basis of the module
$\bi g{}$ and $\bi g{}$ is the subset of elements $\varphi $ of 
$\tk {\k}{\sgg g{}}{\goth g}$ such that $p\varphi \in \bi g{}$ for some $p$ in 
$\sgg g{}\setminus \{0\}$. 

(iii) Let $\varphi $ be in $\bi g{}$. According to (i) and Proposition~\ref{psc1}(iv),
for all $(x,y)$ in $\Omega _{{\goth g}}$, $[x,\varphi (x,y)]$ is orthogonal 
to $V_{x,y}$. Then, since $y$ is in $V_{x,y}$, $[x,\varphi (x,y)]$ is orthogonal to 
$y$ and $\dv {\varphi (x,y)}{[x,y]}=0$, whence the assertion.
\end{proof}

\subsection{Orthogonal complement to $\bi g{}$.} \label{sc2}
Denote again by $\dv ..$ the canonical $\sgg g{}$-bilinear extension of $\dv ..$ to the
module $\tk {\k}{\sgg g{}}{\goth g}$. Set:
$$ I_{*,0} := I_{0} \setminus \{1,\ldots,\rg\}\times \{0\}.$$

\begin{prop}\label{psc2}
Let $\bk g{}$ be the orthogonal complement to $\bi g{}$ in $\tk {\k}{\sgg g{}}{\goth g}$.

{\rm (i)} For $\varphi $ in $\tk {\k}{\sgg g{}}{\goth g}$, $\varphi $ is in
$\bk g{}$ if and only if $\varphi (x,y)$ is in $[x,V_{x,y}]$ for all $(x,y)$ in 
a nonempty open subset of $\gg g{}$.

{\rm (ii)} The module $\bk g{}$ is free of rank $\b g{}-\rg$. Furthermore, the sequence
of maps
$$ (x,y) \mapsto [x,\varepsilon _{i}^{(m)}(x,y)], \, (i,m) \in I_{*,0}$$
is a basis of $\bk g{}$.

{\rm (iii)} The orthogonal complement to $\bk g{}$ in $\tk {\k}{\sgg g{}}{\goth g}$ is
equal to $\bi g{}$.
\end{prop}

\begin{proof}
(i) Let $\varphi $ be in $\tk {\k}{\sgg g{}}{\goth g}$. If $\varphi $ is in
$\bk g{}$, then $\varphi (x,y)$ is orthogonal to $V_{x,y}$ for all $(x,y)$ in 
$\Omega _{{\goth g}}$. Then, according to Proposition~\ref{psc1}(iv), $\varphi (x,y)$
is in $[x,V_{x,y}]$ for all $(x,y)$ in $\Omega _{{\goth g}}$. Conversely,  
suppose that $\varphi (x,y)$ is in $[x,V_{x,y}]$ for all $(x,y)$ in a nonempty open 
subset $O$ of $\gg g{}$. By Proposition~\ref{psc1}(iv) again, for all $(x,y)$ in 
$O\cap \Omega _{{\goth g}}$, $\varphi (x,y)$ is orthogonal to  
$\varepsilon _{i}^{(m)}(x,y)$ for all $(i,m)$ in $I_{0}$, whence the assertion by 
Theorem~\ref{tsc1}(i).

(ii) Let $\bkk $ be the submodule of $\tk {\k}{\sgg g{}}{\goth g}$ generated by the maps
$$ (x,y) \mapsto [x,\varepsilon _{i}^{(m)}(x,y)], \, (i,m) \in I_{*,0} .$$
By Proposition~\ref{psc1}, (iii) and (iv), for all $(x,y)$ in $\Omega _{{\goth g}}$,
the space $[x,V_{x,y}]$ has dimension $\b g{}-\rg$. Hence the maps
$$ (x,y) \mapsto [x,\varepsilon _{i}^{(m)}(x,y)], \, (i,m) \in I_{*,0}$$ are linearly
independent over $\sgg g{}$ since $\vert I_{*,0} \vert=\b g{}-\rg$. As a result,
$\bkk$ is a free module of rank $\b g{}-\rg$. According to (i), $\bkk $ is a submodule of
$\bk g{}$. Let $\varphi $ be in $\bk g{}$. For all $(x,y)$ in $\Omega _{{\goth g}}$,
$\varphi (x,y)$ is a linear combination of
$[x,\varepsilon _{i}^{(m)}(x,y)], \, (i,m)\in I_{*,0}$ by Proposition~\ref{psc1}(iv).
Then, by Claim~\ref{clsc1}, for some regular functions $a_{i,m}, \, (i,m)\in I_{*,0}$ on
$\Omega _{{\goth g}}$,
$$ \varphi (x,y) = \sum_{(i,m)\in I_{*,0}} a_{i,m}(x,y)[x,\varepsilon _{i}^{(m)}(x,y)]$$
for all $(x,y)$ in $\Omega _{{\goth g}}$. Since $\Omega _{{\goth g}}$ is a big open
subset of $\gg g{}$ and $\gg g{}$ is normal, the $a_{i,m}$'s have a regular extension to
$\gg g{}$. As a result, $\varphi $ is in $\bkk$, whence the assertion.

(iii) Let $\varphi $ be in the orthogonal complement to $\bk g{}$ in 
$\tk {\k}{\sgg g{}}{\goth g}$. According to (ii), for all $(x,y)$ in 
$\Omega _{{\goth g}}$, $\varphi (x,y)$ is orthogonal to $[x,V_{x,y}]$. Hence by 
Proposition~\ref{psc1}(iv), $\varphi (x,y)$ is in $V_{x,y}$ for all $(x,y)$ in 
$\Omega _{{\goth g}}$. So, by Proposition~\ref{psc1}(v), $\varphi $ is in $\bi g{}$,
whence the assertion.
\end{proof}

\begin{rema}\label{rsc2}
For ${\goth g}$ simple, $[\varepsilon _{1},\varepsilon _{i}^{(m)}], \, (i,m) \in I_{*,0}$
is a basis of $\bk g{}$ since $\varepsilon _{1}(x)\wedge x = 0$ for all $x$ in
${\goth g}$.
\end{rema}

\subsection{Restriction to a parabolic subalgebra.} \label{sc3}
For ${\goth a}$ subalgebra of ${\goth g}$, set 
${\goth a}_{\r}:= {\goth a}\cap {\goth g}_{\r}$.

\begin{lemma}\label{lsc3}
Let ${\goth a}$ be an algebraic subalgebra of ${\goth g}$.

{\rm (i)} Suppose that ${\goth a}$ contains ${\goth g}^{x}$ for all $x$ in a dense 
open subset of ${\goth a}$ and suppose that ${\goth a}_{\r}$ is not empty. 
Then $V_{x,y}$ is contained in ${\goth a}$ for all $(x,y)$ in $\gg a{}$.

{\rm (ii)} Suppose that ${\goth a}$ contains a Cartan subalgebra of ${\goth g}$. Then 
$V_{x,y}$ is contained in ${\goth a}$ for all $(x,y)$ in $\gg a{}$.
\end{lemma}

\begin{proof}
(i)  By hypothesis, for all $x$ in a dense open subset of ${\goth a}$, $x$ is a 
regular element and ${\goth g}^{x}$ is contained in ${\goth a}$. So by 
\cite[Theorem 9]{Ko}, $\poi x{}{,\ldots,}{}{\varepsilon }{1}{\rg}$ 
are in ${\goth a}$ for all $x$ in a dense open subset of ${\goth a}$. Then, so is it
for all $x$ in ${\goth a}$ by continuity. As a result, for all $(x,y)$ in $\gg a{}$, 
$\varepsilon _{i}^{(m)}(x,y), \, (i,m) \in I_{0}$ is in ${\goth a}$, whence the assertion.

(ii) Let ${\goth c}$ be a Cartan subalgebra of ${\goth g}$ contained in ${\goth a}$. 
Since ${\goth a}$ is an algebraic subalgebra of ${\goth g}$, all semisimple element of 
${\goth a}$ is conjugate under the adjoint group of ${\goth a}$ to an element of 
${\goth c}$. Hence for all $x$ in ${\goth g}_{\rs}\cap {\goth a}$, ${\goth g}^{x}$ is 
contained in ${\goth a}$, whence the assertion by (i) since 
${\goth g}_{\rs}\cap {\goth a}$ is a dense open subset of ${\goth a}$.
\end{proof}

Let ${\goth p}$ be a parabolic subalgebra of ${\goth g}$ containing ${\goth b}$.  
Denote by ${\goth l}$ its reductive factor containing ${\goth h}$, ${\goth p}_{\u}$ 
its nilpotent radical and $\varpi $ the canonical projection 
$\xymatrix{{\goth p} \ar[r] & {\goth l}}$.

\begin{coro}\label{csc3}
For all $(x,y)$ in $\gg p{}$, $V_{x,y}$ is contained in ${\goth p}$. In particular, 
for all $(x,y)$ in a dense open subset of $\gg b{}$, $V_{x,y}={\goth b}$.
\end{coro}

\begin{proof}
Since ${\goth h}$ is contained in ${\goth p}$, for all $(x,y)$ in $\gg p{}$, $V_{x,y}$
is contained in ${\goth p}$ by Lemma~\ref{lsc3}(ii). Since $(h,e)$ is in  
$\Omega _{{\goth g}}$, $\Omega _{{\goth g}}\cap \gg b{}$ is a dense open subset of 
$\gg b{}$, whence the corollary by Proposition~\ref{psc1}(iii). 
\end{proof}

Let $\bi l{}$ be the characteristic module of ${\goth l}$, ${\goth l}_{{\goth l,}\r}$ the
subset of regular elements of ${\goth l}$ and $\Omega _{{\goth l}}$ the subset of
elements $(x,y)$ of $\gg l{}$ such that 
$P_{x,y}\setminus \{0\}$ is contained in ${\goth l}_{{\goth l},\r}$. For $(x,y)$ in
$\gg l{}$, the image of $\bi l{}$ by the evaluation map at $(x,y)$ is denoted by
$V^{{\goth l}}_{x,y}$. Set:
$$ R_{{\goth p}} := \varpi ^{-1}({\goth l}_{{\goth l},\r}) \cap {\goth g}_{\r} .$$

\begin{lemma}\label{l2sc3}
Let $R'_{{\goth p}}$ be the subset of elements $x$ of $R_{{\goth p}}$ such 
that ${\goth g}^{x}\cap {\goth p}_{\u}=\{0\}$.

{\rm (i)} The sets ${\goth g}_{\r}\cap {\goth p}$ and $R_{{\goth p}}$ are big open subsets
of ${\goth p}$.

{\rm (ii)} For all $x$ in $R_{{\goth p}}$, $\varpi ({\goth g}^{x})={\goth l}^{\varpi (x)}$
if and only if ${\goth g}^{x}\cap {\goth p}_{\u}=\{0\}$.

{\rm (iii)} The set $R'_{{\goth p}}$ is a dense open subset of ${\goth p}$.

{\rm (iv)} For all $(x,y)$ in ${\goth p}\times {\goth p}$, $V_{x,y}$ is contained in 
$V^{{\goth l}}_{\varpi (x),\varpi (y)}+{\goth p}_{\u}$. 

{\rm (v)} For all $(x,y)$ in $R'_{{\goth p}}\times {\goth p}$, 
$\varpi (V_{x,y})=V^{{\goth l}}_{\varpi (x),\varpi (y)}$.
\end{lemma}

\begin{proof}
(i) According to \cite{Ve}, ${\goth l}_{{\goth l},\r}$ is a big open subset of
${\goth l}$. Hence $\varpi ^{-1}({\goth l}_{{\goth l},\r})$ is a big open subset of
${\goth p}$. As a result, it remains to prove that ${\goth g}_{\r}\cap {\goth p}$ is a
big open subset of ${\goth p}$. Suppose that ${\goth p}\setminus {\goth g}_{\r}$ has an
irreducible component $\Sigma $ of codimension $1$ in ${\goth p}$. A contradiction is
expected. As ${\goth g}_{\r}\cap {\goth p}$ is a cone invariant under $B$, $\Sigma $ is a
closed cone invariant under $B$. Since $\k[{\goth p}]$ is a factorial ring, for some $p$
in $\k[{\goth p}]$, homogeneous and relatively invariant under $B$, the
nullvariety of $p$ in ${\goth p}$ is equal to $\Sigma $. As a result, 
$\Sigma \cap {\goth b}$ is an equidimensional closed cone of codimension $1$ of 
${\goth b}$ since ${\goth b}\cap {\goth g}_{\r}$ is not empty. So, by Lemma~\ref{lint}, 
$\Sigma = \Sigma \cap {\goth h}+{\goth u}$ and ${\goth u}$ is contained in $\Sigma $
since $0$ is in $\Sigma \cap {\goth h}$ and ${\goth h}$ is not contained in $\Sigma $,
whence a contradiction since ${\goth g}_{\r}\cap {\goth u}$ is not empty. 

(ii) Let $x$ be in $R_{{\goth p}}$. By Lemma~\ref{lsc3}(ii), ${\goth g}^{x}$ is contained
in ${\goth p}$. As $\varpi $ is a surjective morphism of Lie algebra, 
$\varpi ({\goth g}^{x})$ is contained in ${\goth l}^{\varpi (x)}$. Furthermore, 
$\dim \varpi ({\goth g}^{x})=\rg$ if and only if 
${\goth g}^{x}\cap {\goth p}_{\u}=\{0\}$, whence the assertion since ${\goth l}$ has rank
$\rg$.

(iii) For $x$ regular semisimple in a Cartan subalgebra, contained in ${\goth l}$, 
$x$ is in $R'_{{\goth p}}$ since the elements of ${\goth g}^{x}$ are semisimple. So
$R'_{{\goth p}}$ is not empty. The map $x\mapsto {\goth g}^{x}$ from $R_{{\goth p}}$ to 
$\ec {Gr}g{}{}{\rg}$ is regular. So $R'_{{\goth p}}$ is an open subset of $R_{{\goth p}}$
and ${\goth p}$ by (i).

(iv) Let $L_{{\goth l}}$ be the submodule of elements $\varphi $ of 
$\tk {\k}{\e Sl}{{\goth l}}$ such that $[\varphi (x),x]=0$ for all $x$ in ${\goth l}$.
Then $L_{{\goth l}}$ is a free module of rank $\rg$ according to \cite{Di}. Denote by
$\poi {\varphi }1{,\ldots,}{\rg}{}{}{}$ a basis of $L_{{\goth l}}$. For $x$ in 
$R_{{\goth p}}$ and for $i=1,\ldots,\rg$, $\varpi \rond \varepsilon _{i}(x)$ is in 
${\goth l}^{\varpi (x)}$. So there exists a unique element 
$(\poi x{}{,\ldots,}{}{a}{i,1}{i,\rg})$ of $\k ^{\rg}$ such that
$$ \varpi \rond \varepsilon _{i}(x) = 
a_{i,1}(x)\varphi _{1}\rond \varpi (x)+\cdots + 
a_{i,\rg}(x)\varphi _{\rg}\rond \varpi (x) .$$ 
By Claim~\ref{clsc1}, the functions $\poi a{i,1}{,\ldots,}{i,\rg}{}{}{}$ so defined on
$R_{{\goth p}}$ are regular. Hence they have a regular extension to ${\goth p}$ since
${\goth p}$ is normal and $R_{{\goth p}}$ is a big open subset of ${\goth p}$ by (i). As
a result, for all $(x,y)$ in $\gg p{}$ and for all $(a,b)$ in $\k ^{2}$, 
$\varpi \rond \varepsilon _{i}(ax+by)$ is a linear combination of the elements
$\poi {\varpi (ax+by)}{}{,\ldots,}{}{\varphi }{1}{\rg}$. Hence $\varpi (V_{x,y})$ 
is contained in $V_{\varpi (x),\varpi (y)}^{{\goth l}}$ for all $(x,y)$ in $\gg p{}$, 
whence the assertion. 

(v) Let $(x,y)$ be in $R'_{{\goth p}}\times {\goth p}$. By (iii), for all $z$ in a 
dense open subset of $P_{x,y}$, $z$ is in $R'_{{\goth p}}$. So by (ii), 
${\goth l}^{\varpi (z)}$ is contained in $\varpi (V_{x,y})$ for all $z$ in a dense open 
subset of $P_{x,y}$. As a result, according to Proposition~\ref{psc1}(i), 
$V^{{\goth l}}_{\varpi (x),\varpi (y)}$ is contained in $\varpi (V_{x,y})$, whence the 
assertion by (iv).
\end{proof}

\begin{coro}\label{c2sc3}
For all $(x,y)$ in $\Omega _{{\goth g}}\cap \gg p{}$,  
$V_{x,y} = V^{{\goth l}}_{\varpi (x),\varpi (y)}+{\goth p}_{\u}$.
\end{coro}

\begin{proof}
As $(h,e)$ is in $\gg p{}$, $\Omega _{{\goth g}}\cap \gg p{}$ is
a dense open subset of $\gg p{}$. Let $(x,y)$ be in  
$\Omega _{{\goth g}}\cap R'_{{\goth p}}\times {\goth p}$. By Lemma~\ref{l2sc3}(v), 
$\varpi (V_{x,y})=V^{{\goth l}}_{\varpi (x),\varpi (y)}$. Furthermore, 
$\dim V_{x,y}=\b g{}$ since $(x,y)$ is in $\Omega _{{\goth g}}$. Hence 
${\goth p}_{\u}$ is contained in $V_{x,y}$ and 
$\dim V^{{\goth l}}_{\varpi (x),\varpi (y)}=\b l{}$ since   
$\b g{}=\b l{}+\dim {\goth p}_{\u}$. According to Lemma~\ref{lsc3}(ii), the map 
$(x,y)\mapsto V_{x,y}$ is a regular map from $\Omega _{{\goth g}}\cap \gg p{}$ to 
$\ec {Gr}p{}{}{\b g{}}$. So, for all $(x,y)$ in $\Omega _{{\goth g}}\cap \gg p{}$, 
${\goth p}_{\u}$ is contained in $V_{x,y}$ and $\dim \varpi (V_{x,y})=\b l{}$, whence
the assertion by Lemma~\ref{l2sc3}(iv) since $V^{{\goth l}}_{\varpi (x),\varpi (y)}$ has 
dimension at most $\b l{}$.
\end{proof}

\section{Proof of the main theorem} \label{pm}
In this section, we prove that Theorem~\ref{tint} and Theorem~\ref{t2int} 
results from Theorem~\ref{t3int}. So we suppose that Theorem~\ref{t3int} is true 
for ${\goth g}$. Let $\dd $ be the $\sgg g{}$-derivation of the algebra
$\tk {\k}{\sgg g{}}\ex {}{{\goth g}}$ such that for $v$ in ${\goth g}$, $\dd v$ is the
function on $\gg g{}$, $(x,y)\mapsto \dv v{[x,y]}$. Then
$\tk {\k}{\sgg g{}}\ex {}{{\goth g}}$ inherits a structure of differential graded algebra
with respect to the usual gradation of $\ex {}{{\goth g}}$. By Definition,
$C_{\bullet}({\goth g})$ is the graded $\sgg g{}$-submodule of 
$\tk {\k}{\sgg g{}}\ex {}{{\goth g}}$ such that 
$C_{i+\b g{}}({\goth g}):= \ex i{{\goth g}}\wedge \ex {\b g{}}{\bi g{}}$ for 
$i=0,\ldots,n$.

\begin{lemma}\label{lpm}
{\rm (i)} The graded module $C_{\bullet}({\goth g})$ is a graded subcomplex of 
$\tk {\k}{\sgg g{}}\ex {}{{\goth g}}$.

{\rm (ii)} The ideal $I_{{\goth g}}$ is isomorphic to the space of boundaries of
degree $\b g{}$ of $C_{\bullet}({\goth g})$. 

{\rm (iii)} The support of the homology of $C_{\bullet}({\goth g})$ is contained in 
${\cal C}({\goth g})$.
\end{lemma}

\begin{proof}
(i) Set:
\begin{eqnarray}\label{eqpm}
\varepsilon := \wedge _{(i,m)\in I_{0}} \varepsilon _{i}^{(m)}, 
\end{eqnarray}
where the order of the product is induced by the order of $I_{0}$. 
Then $C_{\bullet}({\goth g})$ is the ideal of $\tk {\k}{\sgg g{}}\ex {}{{\goth g}}$ 
generated by $\varepsilon $ since $\varepsilon _{i}^{(m)}, \, (i,m) \in I_{0}$ is a
basis of $\bi g{}$ by Theorem~\ref{tsc1}(i). According to Theorem~\ref{tsc1}(iii), 
for $(i,m)$ in $I_{0}$, $\varepsilon _{i}^{(m)}$ is a cycle of the complex 
$\tk {\k}{\sgg g{}}\ex {}{{\goth g}}$. Hence so is $\varepsilon $ and 
$C_{\bullet}({\goth g})$ is a subcomplex of $\tk {\k}{\sgg g{}}\ex {}{{\goth g}}$ as
an ideal generated by a cycle.

(ii) By definition, $I_{{\goth g}}$ is the ideal of $\sgg g{}$ generated by
$\dd {\goth g}$. Hence $I_{{\goth g}}\varepsilon $ is the space of boundaries of degree
$\b g{}$ of $C_{\bullet}({\goth g})$.

(iii) Let $(x_{0},y_{0})$ be in $\gg g{}\setminus {\cal C}({\goth g})$ and  $v$ in 
${\goth g}$ such that $\dv v{[x_{0},y_{0}]}\neq 0$. For some affine 
open subset $O$ of $\gg g{}$, containing $(x_{0},y_{0})$, $\dv v{[x,y]}\neq 0$ for all
$(x,y)$ in $O$. Then $\dd v$ is an invertible element of $\k[O]$. For 
$c$ a cycle of $\tk {\sgg g{}}{\k[O]}C_{\bullet}({\goth g})$, 
$$\dd (v\wedge c) = (\dd v)c $$
so that $c$ is a boundary of $\tk {\sgg g{}}{\k[O]}C_{\bullet}({\goth g})$, whence the
assertion. 
\end{proof}

\begin{theo}\label{tpm}
{\rm (i)} The complex $C_{\bullet}({\goth g})$ has no homology of degree bigger than 
$\b g{}$.

{\rm (ii)} The ideal $I_{{\goth g}}$ has projective dimension $2n-1$.

{\rm (iii)} The algebra $\sgg g{}/I_{{\goth g}}$ is Cohen-Macaulay.

{\rm (iv)} The projective dimension of the module 
$\ex {n}{{\goth g}}\wedge \ex {\b g{}}{\bi g{}}$ is equal to $n$. 
\end{theo}

\begin{proof}
(i) Let $Z$ be the space of cycles of degree $\b g{}+1$ of $C_{\bullet}({\goth g})$, 
whence a graded subcomplex of $C_{\bullet}({\goth g})$,
$$ \xymatrix{0 \ar[r] & C_{2n+\rg}({\goth g)} \ar[r] & \cdots \ar[r] & 
C_{n+\rg+2}({\goth g}) \ar[r] & Z \ar[r] & 0 }.$$
According to Lemma~\ref{lpm}(iii), the support of its homology is contained in 
${\cal C}_{{\goth g}}$. In particular, its codimension in $\gg g{}$ is at least
$$ 4n + 2\rg - (2n+2\rg) = 2n =  n + n-1 +1$$
According to Theorem~\ref{t3int}, for $i=n+\rg+2,\ldots,2n+\rg$, 
$C_{i}({\goth g})$ has projective dimension at most $n$. Hence, by 
Corollary~\ref{cpdc}, this complex is acyclic and $Z$ has projective dimension at most 
$2n-2$, whence the assertion.

(ii) and (iii) Since $\bi g{}$ is a free module of rank $\b g{}{}$, 
$\ex {\b g{}{}}{\bi g{}}$ is a free module of rank $1$. By definition, the short sequence
$$ \xymatrix{0 \ar[r] & Z \ar[r] & {\goth g}\wedge \ex {\b g{}}{\bi g{}}
\ar[r] & I_{{\goth g}}\ex {\b g{}{}}{\bi g{}} \ar[r] & 0}$$ 
is exact, whence the short exact sequence
$$ \xymatrix{0 \ar[r] & Z \ar[r] & {\goth g}\wedge \ex {\b g{}}{\bi g{}}
\ar[r] & I_{{\goth g}} \ar[r] & 0 }.$$ 
Moreover, by Theorem~\ref{t3int}, ${\goth g}\wedge \ex {\b g{}}{\bi g{}}$ 
has projective dimension at most $1$. Then, by (i) and Lemma~\ref{l2pdc}, 
$I_{{\goth g}}$ has projective dimension at most $2n-1$. As a result the 
$\sgg g{}$-module $\sgg g{}/I_{{\goth g}}$ has projective dimension at most $2n$. Then by
Auslander-Buchsbaum's theorem \cite[\S 3, \no 3, Th\'eor\`eme 1]{Bou1}, the depth of the 
graded $\sgg g{}$-module $\sgg g{}/I_{{\goth g}}$ is at least
$$ 4\b g{}-2\rg - 2n = 2\b g{} $$
so that, according to \cite[\S 1, \no 3, Proposition 4]{Bou1}, the depth of
the graded algebra $\sgg g{}/I_{{\goth g}}$ is at least $2\b g{}$. In other words, 
$\sgg g{}/I_{{\goth g}}$ is Cohen-Macaulay since it has dimension $2\b g{}$. Moreover, 
since the graded algebra $\sgg g{}/I_{{\goth g}}$ has depth $2\b g{}$, the 
graded $\sgg g{}$-module $\sgg g{}/I_{{\goth g}}$ has projective dimension $2n$. Hence 
$I_{{\goth g}}$ has projective dimension $2n-1$.

(iv) As $I_{{\goth g}}$ has projective dimension $2n-1$, 
$\ex {n}{{\goth g}}\wedge \ex {\b g{}}{\bi g{}}$ has projective dimension $n$ by 
(i), Lemma~\ref{l2pdc} and Theorem~\ref{t3int}. 
\end{proof}

Theorem~\ref{t2int} is given by Theorem~\ref{tpm}(i) and Lemma~\ref{lpm}(ii) and 
Theorem~\ref{tint} is a corollary of Theorem~\ref{tpm}.

\begin{coro}\label{cpm}
The subscheme of $\gg g{}$ defined by $I_{{\goth g}}$ is Cohen-Macaulay and normal. 
Furthermore, $I_{{\goth g}}$ is a prime ideal.
\end{coro}

\begin{proof}
According to Theorem~\ref{tpm}(iii), the subscheme of $\gg g{}$ defined by
$I_{{\goth g}}$ is Cohen-Macaulay. According to \cite[Theorem 1]{Po}, it
is smooth in codimension $1$. So by Serre's normality criterion \cite[\S 1, \no 10,
Th\'eor\`eme 4]{Bou1}, it is normal. In particular, it is reduced and $I_{{\goth g}}$
is radical. According to ~\cite{Ric}, ${\cal C}({\goth g})$ is irreducible. Hence 
$I_{{\goth g}}$ is a prime ideal.
\end{proof}

\section{Some complexes}\label{co} 
Let $X$ be an affine irreducible variety and $V$ a vector space over $\k$ of finite
dimension. Let $D(V)$ be the algebra $\tk {\k}{\es SV}\ex {}V$ and $\dd$ the
$\ex {}V$-derivation of $D(V)$ such that $\dd v\tens 1=1\tens v$ for $v$ in $V$. Hence
$D(V)$ inherits a structure of differential graded algebra with respect to the gradation
of $\ex {}V$. We denote by $D^{\bullet}(V)$ the graded complex so defined. The structure
of graded differential algebra on $D(V)$ induces a structure of graded differential
algebra on $\tk {\k}{\k[X]}D(V)$ whence a graded complex
$\tk {\k}{\k[X]}D^{\bullet}(V)$. 

\subsection{General facts} \label{co1}
For $k$ nonnegative integer, set:
$$ D_{k}(V) := \bigoplus _{i=0}^{k} \tk {\k}{\sy {k-i}V}\ex iV ,$$
so that $D_{k}(V)$ is a graded subcomplex of $D^{\bullet}(V)$.

\begin{defi}\label{dco1}
Let $L$ be a free submodule of positive rank $r$ of $\tk {{\k}}{{\k}[X]}V$. For $k$ 
nonnegative integer, denote by $D_{k}^{\bullet}(V,L)$ the graded subcomplex of 
$\tk {\k}{\k[X]}D^{\bullet}(V)$:
$$ D_{k}^{\bullet}(V,L) := D_{k}^{\bullet}(V)[-r]\wedge \ex rL .$$ 
The restriction to $D_{k}^{\bullet}(V,L)$ of the derivation of 
$\tk {\k}{\k[X]}D^{\bullet}(V)$ is also denoted by $\dd$.

For $W$ subspace of $V$, let $D_{k}^{\bullet}(W,L)$ be the graded subspace of 
$D_{k}^{\bullet}(V,L)$ such that
$$ D_{k}^{i+r}(W,L) := \tk {\k}{\sy {k-i}W}\ex iW\wedge \ex rL .$$
Then $D_{k}^{\bullet}(W,L)$ is a graded subcomplex of $D_{k}^{\bullet}(V,L)$.
\end{defi}

The embedding of $\tk {\k[X]}{\sy kL}\ex rL$ into $\tk {\k}{\sy kV}\ex rL$ is an 
augmentation of $D_{k}^{\bullet}(V,L)$. Denote by $\overline{D}_{k}^{\bullet}(V,L)$ this 
augmented complex. In particular, $\overline{D}_{0}^{\bullet}(V,L)$ is acyclic. 

\begin{lemma} \label{lco1}
Let $k$ be a positive integer.

{\rm (i)} The cohomology of $D^{\bullet}(V)$ is equal to ${\k}$. 

{\rm (ii)} The complex $D_{k}^{\bullet}(V)$ is acyclic.

{\rm (iii)} For $\k[X]=\k$, $\overline{D}^{\bullet}_{k}(V,L)$ is an acyclic complex.
\end{lemma}

\begin{proof}
(i) We prove the statement by induction on the dimension of $V$. For $V$ equal to zero,  
$D^{\bullet}(V)$ is equal to ${\k}$ and its differential is equal to $0$. We suppose
the statement true for any vector space whose dimension is strictly smaller than
$\dim V$. Let $W$ be an hyperplane of $V$ and let $v$ be in $V\backslash W$. Let $a$ be 
a homogeneous cocycle of degree $d$ of $D^{\bullet}(V)$. Then $a$ has a unique expansion
$$ a = v^{m}a_{m} + \cdots + a_{0} \mbox{ ,}$$
with $\poi a0{,\ldots,}{m}{}{}{}$ in $\tk {{\k}}{\es SW}\ex {}V$. If $d=0$, then 
$$ mv^{m-1}a_{m}\tens v + \cdots + a_{1}\tens v =0 $$
so that $a=a_{0}$ is in $\k$ by induction hypothesis. Suppose $d>0$. Then, for
$i=0,\ldots,m$, 
$$a_{i} = a'_{i} + a''_{i}\wedge v ,$$
with $a'_{i}$ and  $a''_{i}$ in $\tk {{\k}}{\es SW}\ex dW$ and 
$\tk {{\k}}{\es SW}\ex {d-1}W$ respectively. From the equality
$$ 0 = \sum_{i=0}^{m} v^{i}\dd a'_{i} + \sum_{i=1}^{m} (-1)^{d}iv^{i-1}
a'_{i}\wedge v + \sum_{i=0}^{m} v^{i}\dd a''_{i}\wedge v ,$$
we deduce that $a'_{0},\ldots,a'_{m}$ are cocycles. So, by the induction hypothesis, 
for $i=0,\ldots,m$, $a'_{i} = \dd b_{i}$ for some $b_{i}$ in 
$\tk {{\k}}{\es SW}\ex {d-1}W$. Then 
$$ a - \dd(\sum_{i=0}^{m} v^{i}b_{i}) = 
v^{m}a''_{m}\wedge v + 
\sum_{i=0}^{m-1} v^{i}((-1)^{d-1}(i+1)b_{i+1}+a''_{i})\wedge v \mbox{ .}$$
Hence $a''_{m}$ and $(-1)^{d-1}(i+1)b_{i+1}+a''_{i}$ are cocycles of 
degree $d-1$ for $i=0,\ldots,m-1$. If $d=1$,
$$ a = \dd(\sum_{i=0}^{m} v^{i}b_{i} + 
\frac{1}{m+1}v^{m+1}a''_{m} + \sum_{i=0}^{m-1} 
\frac{1}{i+1}v^{i+1}((-1)^{d-1}(i+1)b_{i+1}+a''_{i})) \mbox{ .}$$ 
For $d$ bigger than $1$, by induction hypothesis,
$a''_{m}$ is the coboundary of an element $c_{m}$ in 
$\tk {{\k}}{\es SW}\ex {d-2}W$ and for $i=0,\ldots,m-1$,
$(-1)^{d-1}(i+1)b_{i+1}+a''_{i}$ is the coboundary of an element $c_{i}$ in 
$\tk {{\k}}{\es SW}\ex {d-2}W$ so that
$$ a =  \dd(\sum_{i=0}^{m} v^{i}b_{i} + \sum_{i=0}^{m} v^{i}c_{i}\wedge v) . $$

(ii) As $D^{\bullet}(V)$ is the direct sum of $D_{i}^{\bullet}(V),i\in {\Bbb N}$, 
the assertion results from (i).

(iii) Let $F$ be a complement to $L$ in $V$. For $i=0,\ldots,k$,
$$ D_{k}^{i+r}(V,L) = \bigoplus _{j=0}^{k-i}
\tk {\k}{\sy {k-i-j}F}\tk {\k}{\sy jL}\ex {i}F\wedge \ex rL ,$$
whence  
$$D_{k}^{\bullet}(V,L) =   
\bigoplus _{j=0}^{k} \tk {\k}{\sy {j}L}D_{k-j}^{\bullet}(F)[-r]\wedge \ex rL  .$$
By (ii), for $j<k$, $D_{k-j}^{\bullet}(F)$ is acyclic. As a result, 
$D_{k}^{\bullet}(V,L)$ has no cohomology of degree different from $r$, whence the 
assertion since for $v$ in $V$, $v\wedge \ex rL = \{0\}$ if and only if $v$ is in $L$.
\end{proof}
 
For $\pi $ automorphism of $X$, denote by $\pi ^{\#}$ the automorphism of the 
algebra $\tk {{\k}}{{\k}[X]}\tk {{\k}}{\es SV}\ex {}V$ induced by the comorphism of
$\pi $. Let $L$ be a free submodule of rank $r$ of $\tk {\k}{\k[X]}V$.

\begin{lemma}\label{l2co1}
Let $k$ be a positive integer and $\pi $ an automorphism of $X$. 

{\rm (i)} The restriction of $\pi ^{\#}$ to $D^{\bullet}_{k}(V,L)$ is an isomorphism from 
$D^{\bullet}_{k}(V,L)$ onto $D^{\bullet}_{k}(V,\pi ^{\#}(L))$.

{\rm (ii)} For any positive integer $j$, the image by $\pi ^{-1}$ of the support in $X$ of
the cohomology of degree $j$ of $D^{\bullet}_{k}(V,L)$ is the support in $X$ of the 
cohomology of degree $j$ of $D^{\bullet}_{k}(V,\pi ^{\#}(L))$. 
\end{lemma}

\begin{proof}
(i) For $i$ positive integer, $\pi ^{\#}(D_{k}^{i}(V,L))$ is equal to 
$D_{k}^{i}(V,\pi ^{\#}(L))$. Hence $\pi ^{\#}(D_{k}^{\bullet}(V,L))$
is equal to $D^{\bullet}_{k}(V,\pi ^{\#}(L))$. As $\pi ^{\#}$ is an automorphism of the 
complex $\tk {\k}{{\k}[X]}D^{\bullet}(V)$, the restriction of
$\pi ^{\#}$ to $D^{\bullet}_{k}(V,L)$ is an isomorphism of the complex
$D^{\bullet}_{k}(V,L)$ onto the complex $D^{\bullet}_{k}(V,\pi ^{\#}(L))$. 
 
(ii) Let $j$ be a positive integer, $J_{j}$ and $J_{j,\pi }$ the ideals of 
definition in ${\k}[X]$ of the supports of the cohomology of degree $j$ of
$D^{\bullet}_{k}(V,L)$ and $D^{\bullet}_{k}(V,\pi ^{\#}(L))$ respectively. If $a$ is a
cocycle of degree $j$ of $D^{\bullet}_{k}(V,L)$ and  $p$ is in $J_{j}$, for $m$
sufficiently big positive integer, $p^{m}a$ is a coboundary of $D^{\bullet}_{k}(V,L)$.
Hence by (i), $\pi ^{\#}(p)^{m}\pi ^{\#}(a)$ is a coboundary of
$D^{\bullet}_{k}(V,\pi ^{\#}(L))$. So $J_{j,\pi }$ contains $\pi ^{\#}(J_{j})$. By the
same argument, $J_{j}$ contains $(\pi ^{-1})^{\#}(J_{j,\pi })$. Hence $J_{j,\pi }$ is
equal to $\pi ^{\#}(J_{j})$, whence the assertion.
\end{proof}

For any $x$ in $X$, denote by $L(x)$ the image of $L$ by the map 
$\varphi \mapsto \varphi (x)$.

\begin{lemma}\label{l3co1}
Let $X'$ be the subset of elements $x$ of $X$ such that $L(x)$ has dimension $r$ and  
${\cal L} := \tk {\k[X]}{\an X{}}L$.

{\rm (i)} The subset $X'$ of $X$ is open and nonempty. Moreover, $X'$ has a finite cover 
by affine open subsets $Y$ which have the following property: 
\begin{itemize}
\item there exists a subspace $E$ of $V$ which is a complement to $L(x)$ in 
$V$ for all $x$ in $Y$.
\end{itemize}

{\rm (ii)} For all positive integer $k$, the support in $X$ of the cohomology of 
$\overline{D}_{k}^{\bullet}(V,L)$ has an empty intersection with $X'$.

{\rm (iii)} Suppose that $X$ is normal and $X'$ is a big open subset of $X$. Then 
$\overline{D}_{k}^{\bullet}(V,L)$ has no cohomology of degree $r$.
\end{lemma}

\begin{proof}
(i) Let $\poi {\eta }1{,\ldots,}{r}{}{}{}$ be a basis of $L$. For all $x$ in $X$,
$L(x)$ is the subspace of $V$ generated by $\poi x{}{,\ldots,}{}{\eta }{1}{r}$. 
Then $X'$ is a nonempty open subset of $X$. Let $x$ be in $X'$. Let $E$ be a complement
to $L(x)$ in $V$. Then, for all $y$ in an open neighborhood $Y_{x}$ of $x$ in $X$,
$L(y)$ has dimension $r$ and $E$ is a complement to $L(y)$ in $V$. In particular,
$Y_{x}$ is contained in $X'$. 

(ii) Let $k$ be a positive integer and $Y$ an affine open subset of $X'$ which 
satisfies the condition of (i). Denoting by $L_{Y}$ the space of sections of 
${\cal L}$ above $Y$, we have to prove that $\overline{D}_{k}^{\bullet}(V,L_{Y})$ is 
acyclic.

Let $x_{0}$ be in $Y$. For $x$ in $Y$, denote by $\tau (x)$ the linear automorphism of 
$V$ such that $\tau (x)(v)=v$ for all $v$ in $E$ and for $w$ in $L(x)$, $\tau (x)(w)$
is the element of $L(x_{0})$ such that $w-\tau (x)(w)$ is in $E$. Let $\overline{\tau }$ 
be the automorphism of the algebra $\tk {{\k}}{{\k}[Y]}\tk {{\k}}{\es SV}\ex {}V$ 
such that $\overline{\tau }(\varphi )$ is the map 
$$ \xymatrix{ Y \ar[rr] &&  \tk {{\k}}{\es SV}\ex {}V}, \qquad 
x \longmapsto \tau (x)(\varphi (x)) .$$
The images of $L_{Y}$ and $\overline{D}_{k}^{\bullet}(V,L_{Y})$ by $\overline{\tau }$ are
equal to $\tk {{\k}}{{\k}[Y]}L_{Y}(x_{0})$ and
$\tk {{\k}}{{\k}[Y]}\overline{D}_{k}^{\bullet}(V,L_{Y}(x_{0}))$ respectively. Moreover, 
the map
$$ \xymatrix{ \overline{D}_{k}^{\bullet}(V,L_{Y}) \ar[rr]^-{\overline{\tau }} &&
\tk {{\k}}{{\k}[Y]}\overline{D}_{k}^{\bullet}(V,L_{Y}(x_{0}))}$$
is an isomorphism of graded complexes. Hence by lemma \ref{lco1}(iii), 
$\overline{D}_{k}^{\bullet}(V,L_{Y})$ is acyclic.

(iii) Let $a$ be a cocycle of degree $r$ of $\overline{D}_{k}^{\bullet}(V,L)$. By (ii), 
the restriction of $a$ to $Y$ is the image by $\dd $ of a unique element $\varphi $ 
of $\tk {{\k}[Y]}{\sy k{L_{Y}}}\ex r{L_{Y}}$. So, by (i), the restriction of $a$ to $X'$ 
is the image by $\dd $ of a section above $X'$ of
$\tk {\an X{}}{\sy k{\cal L}}\ex r{{\cal L}}$. As $L$ is a free module, 
$\tk {{\k}[X]}{\sy kL}\ex rL$ is a free module and any section above $X'$ of
$\tk {\an X{}}{\sy k{\cal L}}\ex r{{\cal L}}$ is the restriction to $X'$ of an element of
$\tk {{\k}}{\sy kL}\ex rL$ since $X'$ is a big open subset of the normal variety 
$X$. Hence $a$ is a coboundary.
\end{proof}

\subsection{Some equivalence} \label{co2}
Let $L$ be a free submodule of rank $r$ of $\tk {\k}{\k[X]}V$ such that $r < \dim V$.
For $k$ nonnegative integer, denote by $K_{k}^{\bullet}(V,L)$ the graded subcomplex of 
$\tk {\k}{\k[X]}D^{\bullet}(V)$ whose subspace of degree $i$ is
$$ K_{k}^{i}(V,L) := \tk {\k}{\sy {k-i}L}\ex iV $$
for $i=0,\ldots,k$. In particular, for $k=0$ and $i$ positive, $K_{k}^{i}(V,L)=\{0\}$
and for $k$ positive, $K_{k}^{i}(V,L)=\{0\}$ for $i>k$ or $i>\dim V$. The map 
$$ \xymatrix{ K_{k}^{k}(V,L) \ar[rr]^-{\theta _{k}} && D_{k}^{k+r}(V,L)}, \qquad
\varphi \longmapsto \varphi \wedge \eta $$ 
with $\eta $ a generator of $\ex rL$, is an augmentation of $K_{k}^{\bullet}(V,L)$.
Denote by $\overline{K}_{k}^{\bullet}(V,L)$ the augmented complex so defined.

\begin{lemma}\label{lco2}
Let $X'$ be the subset of elements $x$ of $X$ such that $L(x)$ has dimension $r$.

{\rm (i)} For $k=1,\ldots,\dim V-r$, the support of the cohomology of 
$\overline{K}_{k}^{\bullet}(V,L)$ is contained in $X\setminus X'$.

{\rm (ii)} The complex $K_{k}^{\bullet}(V,L)$ has no cohomology of degree $0$.
\end{lemma}

\begin{proof}
(i) Let $Y$ be an affine open subset of $X'$ and $E$ a subspace of $V$ satisfying the 
condition of Lemma~\ref{l3co1}(i). Set
$$ L_{Y} := \tk {\k[X]}{\k[Y]}L \quad  \text{and} \quad
D(L_{Y}) := \tk {\k[Y]}{\es S{L_{Y}}}\ex {}{L_{Y}}$$
so that $D(L_{Y})$ is a differential graded subalgebra of $\tk {\k}{\k[Y]}D(V)$. For
$k$ nonnegative integer, set:
$$D_{k}^{\bullet}(L_{Y}) := \bigoplus _{i=0}^{k} \tk {\k[Y]}{\sy {k-i}{L_{Y}}}
\ex i{L_{Y}}.$$ 
Then the complex $K_{k}^{\bullet}(V,L_{Y})$ is isomorphic to
$$ \bigoplus _{j=0}^{k} \tk {\k}{D_{k-j}^{\bullet}(L_{Y})[-j]}{\ex jE} .$$
Hence $K_{k}^{\bullet}(V,L_{Y})$ has no cohomology of degree different from $k$ by 
Lemma~\ref{lco1}(ii) since $L_{Y}$ is a free module. Moreover, the space of cocycles of 
degree $k$ of $\overline{K}_{k}^{\bullet}(V,L_{Y})$ is equal to
$$ \bigoplus _{j=1}^{k} \ex j{L_{Y}}\wedge \ex {k-j}E.$$ 
Hence $\overline{K}_{k}^{\bullet}(V,L_{Y})$ is acyclic and the support of 
the cohomology of $\overline{K}_{k}^{\bullet}(V,L)$ is contained in $X\setminus X'$
since $X'$ has a cover by affine open subsets satisfying the condition of 
Lemma~\ref{l3co1}(i).

(ii) By (i), $K_{k}^{\bullet}(V,L)$ has no cohomology of degree $0$ since 
$\sy kL$ is a torsion free module.
\end{proof}

Denote by $\poi {\eta }1{,\ldots,}{r}{}{}{}$ a basis of $L$. Let ${\goth I}$ and
${\goth I}_{j}$ be as in the Tables of Notations (the sets) for $n=r$. For
$\iota   = \poi i1{ < \cdots < }{j}{}{}{}$ in ${\goth I}_{j}$, 
set:
$$ \{\iota \} := \{\poi i1{,\ldots,}{j}{}{}{}\}, \qquad \vert \iota   \vert = j, \qquad
\eta _{\iota } := \poi {\eta }{i_{1}}{ \wedge \cdots \wedge }{i_{j}}{}{}{} .$$
For $s=(\poi s1{,\ldots,}{r}{}{}{})$ in ${\Bbb N}^{r}$, set:
$$ \vert s \vert := \poi s1{+\cdots +}{r}{}{}{} \quad  \text{and} \quad
\eta ^{s} := \poie {\eta }1{\cdots }{r}{}{}{}{{s_{1}}}{{s_{r}}}.$$
For $k=1,\ldots,\dim V-r$ and $j$ nonnegative integer, denote by
$K_{k,j}^{\bullet}(V,L)$ the graded subcomplex of $K_{k}^{\bullet}(V,L)$ whose subspace
of degree $i$ is
$$ K_{k,j}^{i}(V,L) := \tk {\k}{\sy {k-i}L}{\ex {i-j}V}\wedge \ex jL .$$
In particular, $K_{k,j}^{i}(V,L)=\{0\}$ for $i>k$ or $i<j$ or $j>\sup\{r,k\}$.
Let $Z_{k}^{i}$ and $B_{k}^{i}$ be the spaces of cocycles and coboundaries of degree $i$
of $K_{k}^{\bullet}(V,L)$. Denote by $Z_{k,j}^{i}$ the space of elements $\varphi $
of $K_{k,j}^{i}(V,L)$ such that $\dd \varphi $ is in $K_{k,j+2}^{i+1}(V,L)$. 

\begin{lemma}\label{l2co2}
Let $k=1,\ldots,\dim V -r$. Suppose that $\ex {i-1}V\wedge L$ is the kernel of 
$\theta _{i}$ for $i=1,\ldots,k$. 

{\rm (i)} For $j=0,\ldots,k$ and $i=j,\ldots,k$, $Z_{k,j}^{i}$ is contained in
$B_{k}^{i}+K_{k,j+1}^{i}(V,L)$.

{\rm (ii)} The complex $\overline{K}_{k}^{\bullet}(V,L)$ is acyclic.
\end{lemma}

\begin{proof}
(i) For $i=0,\ldots,k$ and $\varphi := (\varphi _{s,\iota }, \, (s,\iota ) \in
{\Bbb N}^{r}_{k-i}\times {\goth I}_{j})$ in $\tk {\k}{\k[X]}\ex {i-j}V$, denote by
$\nu _{1}(\varphi )$ the biggest integer such that for some $(s,\iota )$,
$$ \varphi _{s,\iota } \neq 0 \quad  \text{and} \quad
\nu _{1}(\varphi ) = k-i - \sup \{\poi s1{,\ldots,}{r}{}{}{}\} .$$
Let $E(\varphi )$ be the subset of elements $(s,\iota )$ sayisfying these two conditions.
Set $\nu _{2}(\varphi ) := \vert E(\varphi ) \vert$ and
$\nu (\varphi ) := (\nu _{1}(\varphi ),\nu _{2}(\varphi ))$.
For $\psi $ in $K_{k,j}^{i}(V,L)$, denote by
$\nu (\psi )$ the smallest element of ${\Bbb N}^{2}$ such that for some
$\varphi := (\varphi _{s,\iota }, \, (s,\iota ) \in {\Bbb N}^{r}_{k-i}
\times {\goth I}_{j})$ in $\tk {\k}{\k[X]}\ex {i-j}V$,
$$\nu (\varphi )=\nu (\psi ) \quad  \text{and} \quad
\psi = \sum_{(s,\iota )\in {\Bbb N}^{r}_{k-i}\times {\goth I}_{j}}
\eta ^{s}\tens \varphi _{s,\iota }\wedge \eta _{\iota } .$$

For $\iota $ in ${\goth I}_{j}$ and $l$ in $\{1,\ldots,r\}\setminus \{\iota \}$,
denote by $\upsilon (l,\iota )$ the element of ${\goth I}_{r-j-1}$ such that
$\{\upsilon (l,\iota )\}$ is the complement to $\{l\}\cup \{\iota \}$ in
$\{1,\ldots,r\}$. Let $\varepsilon (l,\iota )$ be the element of $\{\pm 1\}$ such that
$$ \eta _{\iota }\wedge \eta _{l}\wedge \eta _{\upsilon (l,\iota )}=
\varepsilon (l,\iota ) \eta .$$

Prove by induction on $\nu =(\nu _{1},\nu _{2})$ that for $\varphi $ in $Z_{k,j}^{i}$
such that $\nu (\varphi )= \nu $, $\varphi $ is in $B_{k}^{i}+K_{k,j+1}^{i}(V,L)$. Let
$\tilde{\varphi } := (\varphi _{s,\iota }, \, (s,\iota ) \in
{\Bbb N}^{r}_{k-i}\times {\goth I}_{j}$) be such that ,
$$ \varphi = \sum_{(s,\iota )} \eta ^{s}\tens \varphi _{s,\iota }\wedge \eta _{\iota }
\quad  \text{and} \quad \nu (\varphi ) = \nu (\tilde{\varphi }) = \nu .$$
Suppose $\nu _{1}= 0$. For $(s,\iota )$ in ${\Bbb N}^{r}_{k-i}\times {\goth I}_{j}$
such that $\varphi _{s,\iota }\neq 0$, for some $m$ in $\{1,\ldots,r\}$,
$\eta ^{s}=\eta _{m}^{k-i}$. For such $(s,\iota )$,
$$ \eta ^{s}\tens \varphi _{s,\iota }\wedge \eta _{\iota } = \pm
\frac{1}{k-i+1} \dd \eta _{m}^{k-i+1}\tens \varphi _{s,\iota }\wedge \eta _{\iota '} $$
if $m$ is in $\iota $ and $\iota '$ is the element of ${\goth I}_{j-1}$ such that
$\{\iota \}=\{m\}\cup \{\iota '\}$. Otherwise,
$$ (k-i)\eta _{m}^{k-i-1}\tens \varphi _{s,\iota }\wedge \eta _{\iota }\wedge \eta _{m}
\wedge \eta _{\upsilon (m,\iota )} = 0 .$$
So, by hypothesis,
$\varphi _{s,\iota } \in \ex {i-j-1}V\wedge L$. As a result, $\varphi $ is in
$B_{k}^{i}+K_{k,j+1}^{i}(V,L)$.

Suppose $\nu _{1}>0$ and the assertion true for the elements of ${\Bbb N}^{2}$ smaller
than $\nu $. Let $(s_{*},\iota _{*})$ be in ${\Bbb N}^{r}_{k-i}\times {\goth I}_{j}$
such that $\nu _{1} = k-i - \sup \{\poi s{*,1}{,\ldots,}{*,r}{}{}{}\}$ and
$\varphi _{s_{*},\iota _{*}}\neq 0$. Denote by $\Lambda $ the subset of elements
$\lambda $ of $K_{k,j}^{i}(V,L)$ such that
$\nu (\lambda )$ is smaller than $\nu $. Let $m$ be an integer such that
$s_{*,m}=k-i-\nu _{1}$. If $m$ is in $\{\iota _{*}\}$ then for some $\lambda $ in
$\Lambda $,
$$ \eta ^{s_{*}}\tens \varphi _{s_{*},\iota _{*}}\wedge \eta _{\iota _{*}} =
\pm \frac{1}{s_{*,m}+1} \dd (\eta _{m}\eta ^{s_{*}}\tens \varphi _{s_{*},\iota _{*}}
\wedge \eta _{\iota '})  + \lambda $$
with $\iota '$ the element of ${\goth I}_{j-1}$ such that
$\{\iota _{*}\}=\{m\}\cup \{\iota '\}$. So, by induction hypothesis, $\varphi $ is in
$B^{i}_{k}+K^{i}_{k,j+1}(V,L)$ since
$$ \nu (\varphi - \pm \frac{1}{s_{*,m}+1}
\dd (\eta _{m}\eta ^{s_{*}}\tens \varphi _{s_{*},\iota _{*}}\wedge \eta _{\iota '}))
< \nu .$$

Suppose that $m$ is not in $\{\iota _{*}\}$ and denote by $s'$ the element of
${\Bbb N}^{r}_{k-i-1}$ such that $s'_{m}=s_{*,m}-1$ and
$s'_{*,l}=s_{l}$ for $l\neq m$. Let $E_{*}$ be the set of elements $(s,\iota )$ of
${\Bbb N}^{r}_{k-i}\times {\goth I}_{j}\setminus \{(s_{*},\iota _{*})\}$ such that
$$ \varphi _{s,\iota }\neq 0, \exists ! q_{\iota } \in \{1,\ldots,r\}\setminus \{m\},
{\mathrm {such \, that }} $$
$$ p\neq q_{\iota } \Rightarrow s_{p} = s'_{p}, \, 
s_{q_{\iota }} = s'_{q_{\iota }}+1, \quad \{m\}\cup \{\iota _{*}\}=
\{q_{\iota }\}\cup \{\iota \} .$$
As $\varphi $ is in $Z^{i}_{k,j}$,
$$s_{*,m}\varphi _{s_{*},\iota _{*}}\wedge \eta _{\iota _{*}}\wedge \eta _{m}
+ \sum_{(s,\iota )\in E_{*}} s_{q_{\iota }}\varphi _{s,\iota }\wedge \eta _{\iota }
\wedge \eta _{q_{\iota }} \in \ex {i-j-1}V\wedge \ex {j+2}L .$$
As a result, after a multiplication of the left hand side by 
$\eta _{\upsilon (m,\iota _{*})}$,
$$ \varepsilon (m,\iota _{*})s_{*,m}\varphi _{s_{*},\iota _{*}}\wedge
\eta + \sum_{(s,\iota )\in E_{*}} \varepsilon (q_{\iota },\iota )s_{q_{\iota }}
\varphi _{s,\iota }\wedge \eta = 0 .$$
So, by hypothesis, for some $\psi $ in $\ex {i-j-1}V\wedge L$,
$$ \varphi _{s_{*},\iota _{*}} = -\varepsilon (m,\iota _{*})\frac{1}{s_{*,m}}
\sum_{(s,\iota )\in E_{*}} \varepsilon (q_{\iota },\iota )s_{q_{\iota }}
\varphi _{s,\iota } + \psi ,$$
whence
$$ \eta ^{s_{*}}\tens \varphi _{s_{*},\iota _{*}}\wedge \eta _{\iota _{*}} +
\sum_{(s,\iota )\in E_{*}} \eta ^{s}\tens \varphi _{s,\iota }\wedge \eta _{\iota } =
\sum_{(s,\iota )\in E_{*}} \mu _{s,\iota } +
\eta ^{s_{*}}\tens \psi \wedge \eta _{\iota _{*}}$$
with
$$ \mu _{s,\iota } := \frac{1}{s_{*,m}}\eta ^{s'}(
-\varepsilon (m,\iota _{*})\varepsilon (q_{\iota },\iota )s_{q_{\iota }}\eta _{m}\tens
\varphi _{s,\iota }\wedge \eta _{\iota _{*}} + s_{*,m}\eta _{q_{\iota }}\tens
\varphi _{s,\iota }\wedge \eta _{\iota }) $$
for $(s,\iota )$ in $E_{*}$. Let $\iota '$ be the element of ${\goth I}_{j-1}$ such that
$\{\iota _{*}\}=\{\iota '\}\cup \{q_{\iota }\}$ and
$$ \tilde{\mu }_{s,\iota } := (-1)^{i-j}\frac{1}{s_{*,m}}
\varepsilon (m,\iota _{*})\varepsilon (q_{\iota },\iota )\eta ^{s'}\eta _{m}
\eta _{q_{\iota }}\tens \varphi _{s,\iota }\wedge \eta _{\iota '}.$$
Then 
$$ \mu _{s,\iota } = \dd \tilde{\mu }_{s,\iota } + \lambda _{s,\iota } 
- \tau \eta ^{s}\tens \varphi _{s,\iota }\wedge \eta _{\iota },$$
for some $\lambda _{s,\iota }$ in $\Lambda $ and $\tau $ in $\{0,2\}$. As a matter of
fact, $\tau = 0$ if 
$$ -\varepsilon (m,\iota _{*})\varepsilon (q_{\iota },\iota )
\eta _{\iota '}\wedge \eta _{m} = \eta _{\iota }$$
and $2$ otherwise. As a result,
$$ \nu (\varphi - \sum_{(s,\iota )\in E_{*}}
\dd \tilde{\mu }_{s,\iota } -
\eta ^{s_{*}}\tens \psi \wedge \eta _{\iota _{*}}) < \nu .$$
Then, by induction hypothesis, $\varphi \in B^{i}_{k} + K^{i}_{k,j+1}(V,L)$
since $\eta ^{s_{*}}\tens \psi \wedge \eta _{\iota _{*}}$ is in $K^{i}_{k,j+1}(V,L)$.

(ii) Let $i=0,\ldots,k$. Since $Z_{k}^{i}\cap K_{k,j}^{i}(V,L)$ is contained in
$Z_{k,j}^{i}$ for $j=0,\ldots,i$, $Z_{k}^{i}$ is contained in
$B_{k}^{i}+K_{k,i}^{i}(V,L)$ by (i) and an induction on $j$. As $L$ is a free module,
the complex $D_{k}^{\bullet}(L)$ is acyclic by Lemma~\ref{lco1}(ii). Hence
$Z_{k}^{i}=B_{k}^{i}$ and $\overline{K}_{k}^{\bullet}(V,L)$ is an acyclic complex.
\end{proof}

\begin{prop}\label{pco2}
Suppose that $X$ is normal and $X'$ is a big open subset of $X$. Let
$k=1,\ldots,\dim V-r$. Then the following conditions are equivalent:

{\rm (i)} For $i=1,\ldots,k$, the complex $\overline{K}_{i}^{\bullet}(V,L)$ is acyclic.  

{\rm (ii)} For $i=1,\ldots,k$, the complex $D_{i}^{\bullet}(V,L)$ has no cohomolgy of
degree different from $r$.

{\rm (iii)} For $i=1,\ldots,k$, $\ex {i-1}V\wedge L$ is the kernel of $\theta _{i}$.
\end{prop}

\begin{proof}
Setting
$$ E_{k}^{i,j} := \left \{ \begin{array}{ccc} 
\tk {\k}{\sy {k-i}V}\tk {\k}{\sy {i-j}L}{\ex jV} & \mbox{ if } & j\leq i \leq k \\
\tk {\k}{\sy {k-j}V}\ex {j}V\wedge \ex rL & \mbox{ if } & i=j-1<k \\
0 & \mbox{ otherwise } & \end{array}\right. ,$$
we have canonical isomorphisms 
$$ E_{k}^{j-1,j} \sim D_{k}^{j+r}(V,L), \quad 
E_{k}^{i,j} \sim \tk {\k}{D_{k-i+j}^{j}(V)}\sy {i-j}L \sim 
\tk {\k}{\sy {k-i}V}K_{i}^{j}(V,L) $$ 
for $j\leq i \leq k$. Denote by $\dd _{*}$ the maps
$$ \xymatrix{ E_{k}^{i,j} \ar[rr]^-{\dd _{*}} && E_{k}^{i,j+1}}, \qquad
a\tens b \tens \omega \longmapsto a \tens \dd (b\tens \omega ),$$
$$ \xymatrix{ E_{k}^{j,j} \ar[rr]^-{\dd _{*}} && E_{k}^{j-1,j}}, \qquad
a\tens \omega \longmapsto a \tens \omega \wedge \eta $$
and again by $\dd $ the map
$$ \xymatrix{ E_{k}^{i,j} \ar[rr]^-{\dd } && E_{k}^{i+1,j+1}}, \qquad
a\tens b \tens \omega \longmapsto \psi ^{-1}(b \tens \dd (a\tens \omega ))$$
with $a$ in $\es SV$, $b$ in $\es SL$, $\omega $ in $\ex {}V$ and $\psi $ the isomorphism
$$ a\tens b\tens \omega \longmapsto b\tens a \tens \omega .$$
We have the double complex
$$ \xymatrix{ & E_{k}^{i,j+1} &  \\ 
\ar[r]^-{\dd } & E_{k}^{i,j} \ar[r]^-{\dd} \ar[u]^{\dd _{*}} & E_{k}^{i+1,j+1} 
 \\ & \ar[u]^{\dd _{*}} &   } .$$
Along a line, $i-j$ is constant and a line corresponding to a nonnegative constant 
is acyclic by Lemma~\ref{lco1}(ii). The map
$\xymatrix{E_{k}^{i,i} \ar[r]^-{\dd _{*}} & E_{k}^{i-1,i}}$ is an augmentation of
the column $E_{k}^{i,\bullet}$ so that we have an augmented double complex.

We prove by induction on $k$ the proposition. By Lemma~\ref{l3co1}(iii), the proposition
is true for $k=1$. Suppose that it is true for $k-1$. If
$\overline{K}_{i}^{\bullet}(V,L)$ is acyclic for $i=1,\ldots,k$, then the columns of
the augmented double complex $E_{k}^{\bullet,\bullet}$ are acyclic. Hence
$D_{k}^{\bullet}(V,L)$ has no cohomology of degree different from $r$. Conversely,
if $D_{i}^{\bullet}(V,L)$ has no cohomology of degree different from $r$ for
$i=1,\ldots,k$ then $\overline{D}_{i}^{\bullet}(V,L)$ is acyclic for $i=1,\ldots,k$ by
Lemma~\ref{l3co1}(iii). By induction hypothesis, the augmented column of
$E_{k}^{i,\bullet}$ is acyclic for $i<k$. Hence so is the augmented column
$E_{k}^{k,\bullet}$ that is $\overline{K}_{k}^{\bullet}(V,L)$ is acyclic. 

Suppose that $\overline{K}_{i}^{\bullet}(V,L)$ is acyclic for $i=1,\ldots,k$. Then
$D_{i}^{\bullet}(V,L)$ has no cohomology of degree different from $r$ for $i=1,\ldots,k$.
Let $a$ be in the kernel of $\theta _{k}$. For some $b$ in $E_{k}^{k-1,k-1}$, $a=\dd b$
and for some $c$ in $E_{k}^{k-1,k-2}$, $\dd _{*}b=\dd \dd_{*}c$ since $\dd _{*}a=0$
and $D_{k}^{\bullet}(V,L)$ has no cohomology of degree $k-1$. Then, for some $c_{1}$
in $E_{k}^{k-1,k-2}$, $b-\dd c=\dd _{*}c_{1}$ since $\overline{K}_{k}^{\bullet}(V,L)$ is
acyclic. As a result $a$ is in $\ex {k-1}V\wedge L$ since $a=\pm \dd _{*}\dd c_{1}$.
Conversely, by Lemma~\ref{l2co2}(ii), $\overline{K}_{k}^{\bullet}(V,L)$ is acyclic
if $\ex {i-1}V\wedge L$ is the kernel of $\theta _{i}$ for $i=1,\ldots,k$. 
\end{proof}

\subsection{Examples of complexes} \label{co3}
Let ${\goth l}$ be a reductive Lie algebra, ${\goth z}$ its center and
$\poi {{\goth d}}1{,\ldots,}{\n}{}{}{}$ its simple components. Set:
$$ n_{1} := \b d1 - \j d1,\ldots,n_{\n} := \b d\n - \j d\n, \quad
{\Bbb I}' := \{(\poi i0{,\ldots,}{\n}{}{}{}) \in {\Bbb N}^{\n+1} \; \vert \;
i_{1} \leq n_{1},\ldots,i_{\n} \leq n_{\n}\},$$
$$ n_{{\goth l}} = \poi n1{+\cdots +}{\n}{}{}{}, \quad
{\Bbb I}'' := {\Bbb I}'\cap \{0\}\times {\Bbb N}^{\n}, \quad 
{\Bbb I}'_{k} := {\Bbb N}_{k}^{\n+1}\cap {\Bbb I}', 
\quad {\Bbb I}''_{k} := {\Bbb I}'_{k} \cap {\Bbb I}''$$
for $k$ nonnegative integer. The sets ${\Bbb I}''$ and ${\Bbb I}''_{k}$ identify with 
subsets of ${\Bbb N}^{\n}$. For $I$ subset of ${\Bbb I}'_{k}$ and $i=0,\ldots,k$, 
let $I_{i}$ be the subset of elements $\iota $ of ${\Bbb I}''_{i}$ such that 
$(k-i,\iota )$ is in $I$. For $k=0,\ldots,n_{{\goth l}}$ and $I$ subset of
${\Bbb I}''_{k}$, denote by $D_{k,I,\#}^{\bullet}({\goth l})$ the total complex deduced
from the multicomplex
$$ \bigoplus _{(\poi j1{,\ldots,}{\n}{}{}{}) \in I} 
\tk {\k}{D_{j_{1}}^{\bullet}({\goth d}_{1})}
\tk {\k}{\cdots }D_{j_{\n}}^{\bullet}({\goth d}_{\n}) .$$
Then $D_{k,I,\#}^{\bullet}({\goth l})$ is a graded subcomplex of 
$D_{k}^{\bullet}({\goth l})$. For $k=0,\ldots,n_{{\goth l}}$ and $I$ subset of
${\Bbb I}'_{k}$, the total complex $D_{k,I,\#}^{\bullet}({\goth l})$ deduced from the
double complex  
$$ \bigoplus _{i=0}^{k} \tk {\k}{D_{k-i}^{\bullet}({\goth z})}
D_{i,I_{i},\#}^{\bullet}({\goth l})$$
is a graded subcomplex of $D_{k}^{\bullet}({\goth l})$. For simplicity, we set:
$$D_{k,\#}^{\bullet}({\goth l}) := D_{k,{\Bbb I}'_{k},\#}^{\bullet}({\goth l}) .$$

By Theorem~\ref{tsc1}(i), $\bi l{}$ is a free submodule of rank $\b l{}$ of the
$\sgg l{}$-module $\tk {\k}{\sgg l{}}{\goth l}$. By Definition~\ref{dco1}, for
$k=0,\ldots,n_{{\goth l}}$, $D_{k}^{\bullet}({\goth l},\bi l{})$ is a graded subcomplex
of $\tk {\k}{\sgg l{}}D^{\bullet}({\goth l})$. Then, for $I$ subset of ${\Bbb I}'_{k}$,
$$ D_{k,I,\#}^{\bullet}({\goth l},\bi l{}) := 
D_{k,I,\#}^{\bullet}({\goth l})[-\b l{}]\wedge \ex {\b l{}}{\bi l{}}$$
is a graded subcomplex of $D_{k}^{\bullet}({\goth l},\bi l{})$. For simplicity, we set:
$$ D_{k,\#}^{\bullet}({\goth l},\bi l{}) :=
D_{k,{\Bbb I}'_{k},\#}^{\bullet}({\goth l},\bi l{}) .$$

\begin{lemma}\label{lco3}
Let $k=0,\ldots,n_{{\goth l}}$. Denote by ${\goth d}$ the derived algebra of ${\goth l}$.

{\rm (i)} For $I$ subset of $I''_{k}$, the complex
$D^{\bullet}_{k,I,\#}({\goth d},\bi d{})$ is the total complex
deduced from the multicomplex
$$ \bigoplus _{(\poi i1{,\ldots,}{\n}{}{}{}) \in I} 
\tk {\k}{D_{i_{1}}^{\bullet}({\goth d}_{1},\bi d1{})}
\tk {\k}{\cdots }D_{i_{\n}}^{\bullet}({\goth d}_{\n},\bi d\n) .$$

{\rm (ii)} For $I$ subset of ${\Bbb I}'_{k}$, 
$$D_{k,I,\#}^{\bullet}({\goth l},\bi l{}) = \tk {\k}{\sgg z{}}
(\bigoplus _{i=0}^{k}
\tk {\k}{\sx {k-i}z}D_{i,I_{i},\#}^{\bullet}({\goth d},\bi d{})
[-\dim {\goth z}]\wedge \ex {\dim {\goth z}}{{\goth z}}).$$

{\rm (iii)}  For $I$ subset of ${\Bbb I}'_{k}$, the subcomplex
$D_{k,I,\#}^{\bullet}({\goth l})$ is a direct factor of $D_{k}^{\bullet}({\goth l})$.
In particular, it is acyclic for $k>0$.
\end{lemma}

\begin{proof}
(i) By definition ${\goth d}$ is the direct sum of
$\poi {{\goth d}}1{,\ldots,}{\n}{}{}{}$ and for $x$ in ${\goth d}$, its centralizer
in ${\goth d}$ is the direct sum of its centralizers in
$\poi {{\goth d}}1{,\ldots,}{\n}{}{}{}$. So, by definition,
$$ \bi d{} = \tk {\sgg d1}{\sgg d{} }\bi d1
\oplus \cdots \oplus \tk {\sgg d\n}{\sgg d{}}\bi d\n ,$$
whence
$$ \ex {\b d{}}{\bi d{}} =
\tk {\k}{\ex {\b d1}{\bi d1}}\tk {\k}{\cdots }\ex {\b d\n}{\bi d\n}$$
and the assertion.

(ii) As ${\goth l}$ is the direct sum of ${\goth d}$ and its center ${\goth z}$,
$$ \bi l{} = \tk {\k}{\sgg l{}}{\goth z} \oplus
\tk {\sgg d1}{\sgg l{}}\bi d1
\oplus \cdots \oplus \tk {\sgg d\n}{\sgg l{}}\bi d\n,$$
whence
$$ \ex {\b l{}}{\bi l{}} = \tk {\k}{\sgg z{}}\ex {\b d{}}{\bi d{}}\wedge
\ex {\dim {\goth z}}{\goth z}$$
and the assertion.

(iii) Let $J$ be a subset of ${\Bbb I}''_{k}$. Denote by $\tilde{J}$ the subset of
elements $(\poi i1{,\ldots,}{\n}{}{}{})$ of ${\Bbb N}_{k}^{\n}$ such that $i_{j}>n_{j}$
for at least some $j$ in $\{1,\ldots,\n\}$. Then
$$D_{k}({\goth d}) = \bigoplus _{(\poi j1{,\ldots,}{\n}{}{}{}) \in J} 
\tk {\k}{D_{j_{1}}({\goth d}_{1})}
\tk {\k}{\cdots }D_{j_{\n}}({\goth d}_{\n}) \oplus
\bigoplus _{(\poi j1{,\ldots,}{\n}{}{}{}) \in \tilde{J}} 
\tk {\k}{D_{j_{1}}({\goth d}_{1})}\tk {\k}{\cdots }D_{j_{\n}}({\goth d}_{\n})$$
and the second sum of the right hand side is a subcomplex of $D_{k}({\goth d})$.
As a result, the complex $D_{k,I,\#}^{\bullet}({\goth l})$ is a direct factor of
$D_{k}^{\bullet}({\goth l})$, whence the assertion by Lemma~\ref{lco1}(ii).
\end{proof}

\subsection{Other examples of complexes} \label{co4}
Let ${\goth p}$ be a parabolic subalgebra of ${\goth g}$ containing ${\goth b}$. Denote
by ${\goth l}$ its reductive factor containing ${\goth h}$, ${\goth p}_{\u}$ its
nilpotent radical, ${\goth z}$ the center of ${\goth l}$ and
$\poi {{\goth d}}1{,\ldots,}{\n}{}{}{}$ the simple factors of ${\goth l}$. Set:
$$ d := \dim {\goth p}_{\u}, \quad d_{0} = \dim {\goth z}, \quad
n_{1} := \b d1-\rg_{{\goth d}_{1}},\ldots,n_{\n} := \b d\n-\rg_{{\goth d}_{\n}}, $$ $$
{\Bbb I} := \{(\poi i{-1}{,\ldots,}{\n}{}{}{}) \in {\Bbb N}^{\n+2} \, \vert \,
i_{1}\leq n_{1},,\ldots,i_{\n}\leq n_{\n}\}, \quad
{\Bbb I}_{k} := {\Bbb N}^{\n+2}_{k} \cap {\Bbb I}$$
for $k$ nonnegative integer. The sets ${\Bbb I}'$ and ${\Bbb I}'_{k}$ of
Subsection~\ref{co3} identify with the intersections of ${\Bbb I}$ and ${\Bbb I}_{k}$
with $\{0\}\times {\Bbb N}^{\n+1}$. For $I$ subset of ${\Bbb I}_{k}$ and $i=0,\ldots,k$,
denote by $I_{i}$ the subset of elements $\iota $ of ${\Bbb I}'$ such that
$(k-i,\iota )$ is in $I$.

Let ${\goth p}_{-}$ be the parabolic subalgebra of ${\goth g}$, containing ${\goth b}$
and opposite to ${\goth p}$, ${\goth p}_{-,\u}$ the nilpotent radical of ${\goth p}_{-}$
and ${\goth p}_{\pm ,\u}$ the sum ${\goth p}_{\u}+{\goth p}_{-,\u}$. For $k=0,\ldots,n$
and $I$ subset of ${\Bbb I}_{k}$, denote by $D_{k,I,{\goth p}}^{\bullet}({\goth g})$ and
$D_{k,I,{\goth p}}^{\bullet}({\goth p}_{-})$ the total graded subcomplexes of
$D_{k}^{\bullet}({\goth g})$ deduced from the doube complexes,
$$ \bigoplus _{i=0}^{k} \tk {\k}{D_{k-i}^{\bullet}({\goth p}_{\pm,\u})}
D_{i,I_{i},\#}^{\bullet}({\goth l}) \quad  \text{and} \quad
\bigoplus _{i=0}^{k} \tk {\k}{D_{k-i}^{\bullet}({\goth p}_{-,\u})}
D_{i,I_{i},\#}^{\bullet}({\goth l})$$
respectively. For simplicity, we set:
$$D_{k,{\goth p}}^{\bullet}({\goth g}) :=
D_{k,{\Bbb I}_{k},{\goth p}}^{\bullet}({\goth g}) \quad  \text{and} \quad
D_{k,{\goth p}}^{\bullet}({\goth p}_{-}) :=
D_{k,{\Bbb I}_{k},{\goth p}}^{\bullet}({\goth p}_{-}).$$

Let ${\goth l}_{*}$ be the subset of elements $x$ of ${\goth l}$ such that
$$ \det \ad _{{\goth g}/{\goth l}} x \neq 0 .$$
Denote by $\thetaup $ the map
$$ \xymatrix{ G\times \lg l* \ar[rr]^-{\thetaup } && \gg g{}},
\qquad (g,x,y) \longmapsto (g(x),g(y)) .$$

\begin{lemma}\label{lco4}
Set $\Omega _{*} := \Omega _{{\goth g}}\cap \lg l*$.

{\rm (i)} The subset $G({\goth l}_{*})$ of ${\goth g}$ is open and $\thetaup $ is a 
faithfully flat morphism.  

{\rm (ii)} The subset $\Omega _{*}$ of $\lg l*$ is a big open subset.
\end{lemma}

\begin{proof}
(i) Since the map $(g,x)\mapsto g(x)$ from $G\times {\goth l}_{*}$ to ${\goth g}$ is a 
submersion, $G({\goth l}_{*})$ is an open subset of ${\goth g}$ and 
this map is a smooth surjective morphism from $G\times {\goth l}_{*}$ to 
$G({\goth l}_{*})$. As a result, $\thetaup $ is a faithfully flat morphism from 
$G\times \lg l*$ onto the open subset $G({\goth l}_{*})\times {\goth g}$ of $\gg g{}$
since the endomorphism of $G\times {\goth l}_{*}\times {\goth g}$,
$(g,x,y)\mapsto (g,x,g(y))$ is an isomorphism.

(ii) By (i), the fibers of $\thetaup $ are equidimensional of dimension 
$\dim {\goth l}$. Hence $\Omega _{*}$ is a big open subset of $\lg l*$ since 
$\Omega _{{\goth g}}$ is a $G$-invariant big open subset of $\gg g{}$.  
\end{proof}

Let $\bii$ and $\oi$ be the restrictions of $\bi g{}$ to $\lg l*$ and $\lp$ respectively.
Since $\bi g{}$ is a free $\sgg g{}$-module of rank $\b g{}$ generated by
$\varepsilon _{i}^{(m)}, \, (i,m)\in I_{0}$, $\bii$ and $\oi$ are free modules of rank
$\b g{}$ over $\k[\lg l*]$ and $\k[\lp]$ respectively since $(h,e)$ is in
$\Omega _{*}\cap \lp$. By definition~\ref{dco1}, for $k$ nonnegative integer,
$D_{k}^{\bullet}({\goth g},\bii)$ and $D_{k}^{\bullet}({\goth g},\oi)$ are graded
subcomplexes of $\tk {\k}{\k[\lg l*]}D_{k}^{\bullet}({\goth g})$ and
$\tk {\k}{\k[\lp]}D_{k}^{\bullet}({\goth g})$ respectively. Moreover,
the graded subspace $D_{k}^{\bullet}({\goth p}_{-})[-\b g{}]\wedge \ex {\b g{}}\oi$ of
$D_{k}^{\bullet}({\goth g},\oi)$ is a subcomplex. Denote by
$D_{k,{\goth p}}^{\bullet}({\goth g},\bii)$ the graded subcomplex of
$D_{k}^{\bullet}({\goth g},\bii)$:
$$ D_{k,{\goth p}}^{\bullet}({\goth g},\bii) :=
D_{k,{\goth p}}^{\bullet}({\goth g})[-\b g{}]\wedge \ex {\b g{}}{\bii},$$
and 
$D_{k,{\goth p}}^{\bullet}({\goth p}_{-},\oi)$ the graded subcomplex of
$D_{k}^{\bullet}({\goth g},\oi)$:
$$ D_{k,{\goth p}}^{\bullet}({\goth p}_{-},\oi) :=
D_{k,{\goth p}}^{\bullet}({\goth p}_{-})[-\b g{}]\wedge \ex {\b g{}}{\oi}.$$

Let $\an {}{}$ be the local ring of $\k[G]$ at the identity, ${\goth m}$ its maximal
ideal and $\han {}{}$ the ${\goth m}$-adic completion of $\an {}{}$. Set:
$$ \widetilde{\bii} := \tk {\k[G]}{\an {}{}}\thetaup ^{*}(\bi g{}) \quad  \text{and} \quad
\widehat{\bii} := \tk {\an {}{}}{\han {}{}}\widetilde{\bii} .$$
By Lemma~\ref{lco4}(i), $\widetilde{\bii}$ and $\widehat{\bii}$ are free submodules of
rank $\b g{}$ of $\tk {\k}{\an {}{}}\tk {\k}{\k[\lg l*]}{\goth g}$ and
$\tk {\k}{\han {}{}}\tk {\k}{\k[\lg l*]}{\goth g}$ respectively. Denote by
$D_{k,{\goth p}}^{\bullet}({\goth g},\widetilde{\bii})$ and
$D_{k,{\goth p}}^{\bullet}({\goth g},\widehat{\bii})$ the graded subcomplexes of
$\tk {\k}{\han {}{}}\tk {\k}{\k[\lg l*]}D_{k}^{\bullet}({\goth g})$:
$$ D_{k,{\goth p}}^{\bullet}({\goth g},\widetilde{\bii}) :=
D_{k,{\goth p}}^{\bullet}({\goth g})[-\b g{}]\wedge \ex {\b g{}}{\widetilde{\bii}}
\quad  \text{and} \quad
D_{k,{\goth p}}^{\bullet}({\goth g},\widehat{\bii}) :=
D_{k,{\goth p}}^{\bullet}({\goth g})[-\b g{}]\wedge \ex {\b g{}}{\widehat{\bii}} .$$

\section{The Property \texorpdfstring{$({\bf P})$}{}} \label{pp}
According to Definition~\ref{dint}, the reductive algebra ${\goth g}$ has Property
$({\bf P})$ if the complex $D_{k}^{\bullet}({\goth g},\bi g{})$ has no cohomology of
degree different from $\b g{}$ for $k=0,\ldots,n$. As it is clear that ${\goth g}$
has not Property $({\bf P})$ if ${\goth g}$ has at least two simple factors, the main
question is:

{\it Does each simple Lie algebra have Property $({\bf P})$ ?}

\noindent In this section we prove that a positive answer to this question gives a proof
of Theorem~\ref{t3int} (see Corollary~\ref{cpp}). 

According to the notations of Section~\ref{pm} and Subsection~\ref{co1},
$C_{i+\b g{}}=D_{i}^{i+\b g{}}({\goth g},\bi g{})$. The embedding
of $\tk {\sgg g{}}{\sy k{\bi g{}}}\ex {\b g{}}{\bi g{}}$ in 
$D_{k}^{\b g{}}({\goth g},\bi g{})$ is an augmentation of 
$D_{k}^{\bullet}({\goth g},\bi g{})$. Denote by 
$\overline{D}_{k}^{\bullet}({\goth g},\bi g{})$ this augmented complex.

\begin{prop}\label{ppp}
Let $k$ be a nonnegative integer.

{\rm (i)} The complex $\overline{D}_{k}^{\bullet}({\goth g},\bi g{})$ has no cohomology 
of degree smaller than $\b g{}+1$. 

{\rm (ii)} For $k=0,1$, $\overline{D}_{k}^{\bullet}({\goth g},\bi g{})$ is acyclic. In 
particular, ${\goth {sl}}_{2}(\k)$ has Property $({\bf P})$.

{\rm (iii)} If ${\goth g}$ is simple and has Property $({\bf P})$, then for 
$i=1,\ldots,n$, $C_{i+\b g{}}({\goth g})$ has projective dimension at most $i$.
\end{prop}

\begin{proof}
(i) By definition, $\overline{D}_{k}^{\bullet}({\goth g},\bi g{})$ has no cohomology of 
degree smaller than $\b g{}$. According to Theorem~\ref{tsc1}, $\bi g{}$ is a free module
of rank $\b g{}$. Since $\Omega _{{\goth g}}$ is a big open subset of $\gg g{}$ and 
$\gg g{}$ is normal, $\overline{D}_{k}^{\bullet}({\goth g},\bi g{})$ has no cohomology of
degree $\b g{}$ for all positive integer $k$ by Lemma~\ref{l3co1}(iii).

(ii) By definition $\overline{D}_{0}^{\bullet}({\goth g},\bi g{})$ is acyclic.  
As $\overline{D}_{1}^{\bullet}({\goth g},\bi g{})$ has no cohomology of degree bigger  
than $\b g{}$ by definition, $\overline{D}_{1}^{\bullet}({\goth g},\b g{})$ is acyclic
by (i). For ${\goth g}={\goth {sl}}_{2}(\k)$, $\b g{}=2$ and $\rg =1$. Hence 
${\goth {sl}}_{2}(\k)$ has Property $({\bf P})$.

(iii) Prove the proposition by induction on $i$. By (i), $C_{1+\b g{}}({\goth g})$ has
projective dimension at most $1$. Suppose that $C_{j+\b g{}}({\goth g})$ has projective 
dimension at most $j$ for $j<i$. By (i) and Property $({\bf P})$, the complex 
$\overline{D}_{i}^{\bullet}({\goth g},\bi g{})$ is acyclic. Then, by induction hypothesis
and Corollary~\ref{c2pdc}, $C_{i+\b g{}}({\goth g})$ has projective dimension at most $i$.
\end{proof}

\begin{coro}\label{cpp}
Suppose that each simple factors of ${\goth g}$ has Property $({\bf P})$. Then, for 
$i=1,\ldots,n$, $C_{i+\b g{}}({\goth g})$ has projective dimension at most $i$.
\end{coro}

\begin{proof}
Let ${\goth z}$ be the center of ${\goth g}$ and ${\goth d}$ the derived algebra of 
${\goth g}$. Denote by $\j d{}$ the rank of ${\goth d}$. As ${\goth z}$ is 
contained in $\bi g{}$, for $i=1,\ldots,n$, we have an isomorphism 
$$ \xymatrix{\tk {\k}{\sgg z{}}\tk {\k}{\ex {\dim {\goth z}}{{\goth z}}}
C_{i+{\mathrm {b}}_{{\goth d}}}({\goth d}) \ar[rr] && 
C_{i+\b g{}}}({\goth g}) .$$ 
Hence the proposition for ${\goth g}$ results from the proposition for 
${\goth g}={\goth d}$ since ${\mathrm {b}}_{{\goth d}} - \j d{}=n$. 

Denote by $\poi {{\goth d}}1{,\ldots,}{\n}{}{}{}$ the simple factors of 
${\goth g}={\goth d}$ and prove the proposition by induction on $\n$. For $\n=1$, the 
proposition results from the hypothesis by Proposition~\ref{ppp}(iii). Suppose 
$\n\geq 2$ and the proposition true for $\n-1$. Let ${\goth a}$ be the direct product of 
$\poi {{\goth d}}1{,\ldots,}{\n-1}{}{}{}$. From the equalities:
\begin{eqnarray*}
\ex {\b g{}}{\bi g{}} = & \ex {\b a{}}{\bi a{}}\wedge 
\ex {\b g{\n}}{\bii _{{\goth d}_{\n}}} \\
\ex i{{\goth g}} = & \bigoplus _{j=0}^{i} \ex j{{\goth a}}\wedge \ex {i-j}{{\goth d}_{\n}}
\\
\b g{} - \rg = & \b a{} - \j a{} + \b d{\n} - \j d{\n} 
\end{eqnarray*}
we deduce an isomorphism 
$$ \xymatrix{ \bigoplus _{j=0}^{i} \tk {\k}{C_{j+\b a{}}({\goth a})}
C_{i-j+\b g\n}({\goth d}_{\n}) \ar[rr] && C_{i+\b g{}}({\goth g})}$$
for $i=1,\ldots,n$. By induction hypothesis, $C_{j+\b a{}}({\goth a})$ has projective 
dimension at most $j$. By the hypothesis and Proposition~\ref{ppp}(iii), 
$C_{i-j+\b g{\n}}({\goth d}_{\n})$ has projective dimension at most $i-j$. Hence 
$C_{i+\b g{}}({\goth g})$ has projective dimension at most $i$.
\end{proof}

Let ${\goth d}$, ${\goth z}$, $\poi {{\goth d}}1{,\ldots,}{\n}{}{}{}$ be as in the proof 
of Corollary~\ref{cpp}. 

\begin{lemma}\label{lpp}
Suppose that $\poi {{\goth d}}1{,\ldots,}{\n}{}{}{}$ have Property $({\bf P})$. 
Then for $k=0,\ldots,n$ and $I$ subset of ${\Bbb I}'_{k}$, 
$D_{k,I,\#}^{\bullet}({\goth g},\bi g{})$ has no cohomology of degree different from 
$\b g{}$.
\end{lemma}

\begin{proof}
By Lemma~\ref{lco3}(i), for $j=0,\ldots,n$ and $J$ subset of ${\Bbb I}''_{j}$, the
complex $D_{j,J,\#}^{\bullet}({\goth d},\bi d{})$ has no cohomology of degree different
from $\b d{}$ since $\poi {{\goth d}}1{,\ldots,}{\n}{}{}{}$ have Property $({\bf P})$.
Then, by Lemma~\ref{lco3}(ii), $D_{k,I,\#}^{\bullet}({\goth g},\bi g{})$ has no
cohomology of degree different from $\b g{}$ since $\b g{}=\b d{}+\dim {\goth z}$.
\end{proof}

For $k=0,\ldots,n$, $I$ subset of ${\Bbb I}'_{k}$ and $j=0,\ldots,\n$, let $I_{j,*}$ be
the subset of elements $i$ of $I$ such that $i_{j}>0$ and $I_{j,-}$ the image of
$I_{j,*}$ by the map
$$ \xymatrix{ {\Bbb N}^{\n+1} \ar[rr] && {\Bbb N}^{\n+1}}, \qquad
i \longmapsto (\poi i{0}{,\ldots,}{j-1}{}{}{},i_{j}-1,\poi i{j+1}{,\ldots,}{\n}{}{}{}).$$
As ${\goth g}$ is the direct sum of ${\goth z}$, $\poi {{\goth d}}1{,\ldots,}{\n}{}{}{}$,
$\k[{\goth d}_{j}]$ is a subalgebra of $\k[{\goth g}]$ for $j=1,\ldots,\n$. Set:
$$ \tilde{K}^{\bullet}_{k,I}({\goth g},\bi g{}) :=
\tk {\k}{\sgg g{}}D^{\bullet}_{k-1,I_{0,-},\#}({\goth g})[-1]\wedge {\goth z} +
\sum_{j=1}^{\n}
\tk {\sgg dj}{\sgg g{}}D^{\bullet}_{k-1,I_{j,-},\#}({\goth g})[-1]\wedge
\bi {{\goth d}}{j} .$$
Denote by $K_{k,I}^{\bullet}({\goth g},\bi g{})$ the kernel of the
morphism 
$$\xymatrix{ \tk {\k}{\sgg g{}}D_{k,I,\#}^{\bullet}({\goth g}) \ar[rr] && 
D_{k,I,\#}^{\bullet}({\goth g},\bi g{})[\b g{}]}, \qquad \varphi \longmapsto 
\varphi \wedge \varepsilon .$$ 
In particular, $\tilde{K}^{\bullet}_{k,I}({\goth g},\bi g{})$ and
$K_{k,I}^{\bullet}({\goth g},\bi g{})$ are graded subcomplexes of 
$\tk {\k}{\sgg g{}}D_{k,I,\#}^{\bullet}({\goth g})$.

\begin{prop}\label{p2pp}
Suppose that $\poi {{\goth d}}1{,\ldots,}{\n}{}{}{}$ have Property $({\bf P})$. Let 
$k=1,\ldots,n$ and $I\subset {\Bbb I}'_{k}$. Then $K_{k,I}^{\bullet}({\goth g},\bi g{})$
is equal to $\tilde{K}^{\bullet}_{k,I}({\goth g},\bi g{})$.
\end{prop}
 
\begin{proof}
Since $\bi g{}$ is the $\sgg g{}$-module generated by ${\goth z}$ and
$\bi d1,\ldots,\bi d{\n}$, $\tilde{K}^{\bullet}_{k,I}({\goth g},\bi g{})$ is contained
in $K^{\bullet}_{k,I}({\goth g},\bi g{})$. Then it remains to prove that
$K_{k,I}^{\bullet}({\goth g},\bi g{})$ is contained in
$\tilde{K}_{k,I}^{\bullet}({\goth g},\bi g{})$. Prove the assertion by induction on $k$.
For $k=1$, it is true by Proposition~\ref{ppp}(i). Suppose $k>1$ and
$K_{j,J}^{\bullet}({\goth g},\bi g{})$ contained in
$\tilde{K}_{j,J}^{\bullet}({\goth g},\bi g{})$ for $j<k$ and $J\subset {\Bbb I}'_{j}$. 

Let $j=1,\ldots,k-1$. For $\upsilon  = (\poi i{0}{,\ldots,}{\n}{}{}{})$ in 
${\Bbb N}^{\n+1}_{k-j}$, set:
\begin{eqnarray*}
V_{\upsilon } := & \tk {\k}{\sx {i_{0}}z}\tk {\k}{\sy {i_{1}}{{\goth d}_{1}}}
\tk {\k}{\cdots }\sy {i_{\n}}{{\goth d}_{\n}} \\
I_{\upsilon } := &
\{(\poi l{0}{,\ldots,}{\n}{}{}{}) \in {\Bbb N}^{\n+1} \; \vert \;
(i_{0}+l_{0},\ldots,i_{\n}+l_{\n}) \in I\} .
\end{eqnarray*}
In particular, $I_{\upsilon }$ is contained in ${\Bbb I}'_{j}$ when it is not empty. Then
$$ D_{k,I,\#}^{j}({\goth g}) = 
\bigoplus _{\upsilon \in {\Bbb N}^{\n+1}_{k-j}} \tk {\k}{V_{\upsilon }}
D_{j,I_{\upsilon },\#}^{j}({\goth g}).$$
So, by induction hypothesis, $K_{k,I}^{j}({\goth g},\bi g{})$ is contained in 
$\tilde{K}^{j}_{k,I}({\goth g},\bi g{})$.

We have a commutative diagram
$$\xymatrix{ & & 0 & 0 & \\
0 \ar[r] & K_{k,I}^{k}({\goth g},\bi g{}) \ar[r] & 
\tk {\k}{\sgg g{}}D_{k,I,\#}^{k}({\goth g}) \ar[u] \ar[r] & 
D_{k,I,\#}^{k+\b g{}}({\goth g},\bi g{}) \ar[u]\ar[r] & 0 \\ 
0 \ar[r] &  K_{k,I}^{k-1}({\goth g},\bi g{}) \ar[u]^{\dd} 
\ar[r] & \tk {\k}{\sgg g{}}D_{k,I,\#}^{k-1}({\goth g}) \ar[u]^{\dd} 
\ar[r] & D_{k,I,\#}^{k-1+\b g{}}({\goth g},\bi g{})  \ar[u]^{\dd} \ar[r] & 0 \\
0 \ar[r] & K_{k,I}^{k-2}({\goth g},\bi g{}) \ar[u]^{\dd} 
\ar[r] & \tk {\k}{\sgg g{}}D_{k,I,\#}^{k-2}({\goth g}) \ar[u]^{\dd} 
\ar[r] & D_{k,I,\#}^{k-2+\b g{}}({\goth g},\bi g{})  \ar[u]^{\dd} \ar[r] & 0 } .$$
By definition, the rows are exact, by Lemma~\ref{lpp}, the right column is exact and by 
Lemma~\ref{lco3}(iii), the middle column is exact since the complex
$D_{k,I,\#}^{\bullet}({\goth g})$ is a direct factor of the complex
$D_{k}^{\bullet}({\goth g})$. Denoting by $\delta $ the horizontal 
arrows, for $a$ in $K_{k,I}^{k}({\goth g},\bi g{})$, $a=\dd b$ for some $b$ in 
$\tk {\k}{\sgg g{}}D_{k,I,\#}^{k-1}({\goth g})$, $\delta b = \dd c$ for $c$ in 
$D_{k,I,\#}^{k-2+\b g{}}({\goth g},\bi g{})$ and $c=\delta c'$ 
for some $c'$ in $\tk {\k}{\sgg g{}}D_{k,I,\#}^{k-2}({\goth g})$, whence 
$b-\dd c'=\delta b'$ for some $b'$ in $K_{k,I}^{k-1}({\goth g},\bi g{})$ and $a=\dd b'$. 
As $K^{k-1}_{k,I}({\goth g},\bi g{})=\tilde{K}^{k-1}_{k,I}({\goth g},\bi g{})$ and
$\dd \tilde{K}^{k-1}_{k,I}({\goth g},\bi g{}) \subset
\tilde{K}^{k}_{k,I}({\goth g},\bi g{})$, $a$ is in
$\tilde{K}^{k}_{k,I}({\goth g},\bi g{})$, whence the proposition.
\end{proof}

\section{First step to the proof of Theorem~\ref{t4int}} \label{stp}
From now on, we assume that ${\goth g}$ is simple of rank at least $2$. According to
the notations of Subsection~\ref{co4}, we suppose that
$D_{k,{\goth p}}^{\bullet}({\goth g},\widehat{\bii})$ has no cohomology of degree
different from $\b g{}$ for $k=1,\ldots,n$ and any parabolic subalgebra ${\goth p}$
of ${\goth g}$ containing ${\goth b}$.

\subsection{Some remarks on supports} \label{stp1}
Let $\varpi _{1}$ be the first projection $\xymatrix{\gg g{} \ar[r] & {\goth g}}$.
For $k=1,\ldots,n$, denote by $S_{k}$ the support in $\gg g{}$ of the cohomology of 
$D_{k}^{\bullet}({\goth g},\bi g{})$ of degree different from $\b g{}$. 

\begin{lemma}\label{lstp1}
Let $k=1,\ldots,n$.

{\rm (i)} The set $S_{k}$ is a closed subset of $\gg g{}$ invariant under 
the actions of $G$ and ${\mathrm {GL}}_{2}(\k)$ in $\gg g{}$.

{\rm (ii)} The subset $\varpi _{1}(S_{k})$ of ${\goth g}$ is closed and $G$-invariant. 

{\rm (iii)} If $\varpi _{1}(S_{k})\cap {\goth h} = \{0\}$, then 
$S_{k}$ has codimension at least $n+2\rg$ in $\gg g{}$.  
\end{lemma}

\begin{proof}
(i) According to Proposition~\ref{psc1}(vi), $\bi g{}$ is a free module generated by 
a basis of $G$-equivariant elements. Moreover, by definition, $\bi g{}$ is invariant
under the action of ${\mathrm {GL}}_{2}(\k)$. Hence $\ex {\b g{}}{\bi g{}}$ is generated
by a  $G\times {\mathrm {GL}}_{2}(\k)$-semi-invariant element of
$\tk {\k}{\sgg g{}}\ex {\b g{}}{{\goth g}}$. Then, as the differential
of $D^{\bullet}({\goth g})$ is $G$-invariant, the differential of
$D_{k}^{\bullet}({\goth g},\bi g{})$ is $G\times {\mathrm {GL}}_{2}(\k)$-semi-invariant.
Hence $S_{k}$ is invariant under $G\times {\mathrm {GL}}_{2}(\k)$.  

(ii) As $S_{k}$ is invariant under $\k^{*}\times \k^{*}$, 
$\varpi _{1}(S_{k})\times \{0\}=S_{k}\cap {\goth g}\times \{0\}$ so that 
$\varpi _{1}(S_{k})$ is a closed subset of ${\goth g}$. As $S_{k}$ is $G$-invariant so is
$\varpi _{1}(S_{k})$. 

(iii) Suppose $\varpi _{1}(S_{k})\cap {\goth h} = \{0\}$. By (ii), 
$\varpi _{1}(S_{k})$ is contained in the nilpotent cone ${\goth N}_{{\goth g}}$ of 
${\goth g}$. Then $S_{k}$ is contained in the nilpotent bicone ${\cal N}_{{\goth g}}$
since $S_{k}$ is invariant under the action of ${\mathrm {GL}}_{2}(\k)$. As a result, 
$S_{k}$ has codimension at least $n+2\rg$ in $\gg g{}$ since ${\cal N}_{{\goth g}}$ has 
dimension $3n$ by \cite[Theorem 1.2]{CMo}.
\end{proof}

\begin{prop}\label{pstp1}
Let $k=2,\ldots,n$. Suppose that $D_{i}({\goth g},\bi g{})$ has no cohomology of degree
different from $\b g{}$ for $i=1,\ldots,k-1$. If
$\varpi _{1}(S_{k})\cap {\goth h}=\{0\}$ then $D_{k}({\goth g},\bi g{})$ has no
cohomology of degree different from $\b g{}$.
\end{prop}

\begin{proof}
Suppose $\varpi _{1}(S_{k})\cap {\goth h}=\{0\}$. By Lemma~\ref{lstp1}(iii), $S_{k}$ has
codimension at least $k+2$ in $\gg g{}$. According to the hypothesis and
Proposition~\ref{ppp}(i), for $i=1,\ldots,k-1$, the augmented complex
$\overline{D}_{i}^{\bullet}({\goth g},\bi g{})$ is acyclic. So, by
Corollary~\ref{c2pdc}, for $i=1,\ldots,k-1$,
$\ex i{{\goth g}}\wedge \ex {\b g{}}{\bi g{}}$ has projective dimension at most $i$.
Again by Proposition~\ref{ppp}(i), $S_{k}$ is the support in $\gg g{}$ of the cohomology
of $\overline{D}_{k}^{\bullet}({\goth g},\bi g{})$. Then, by Corollary~\ref{c2pdc},
$\overline{D}_{k}^{\bullet}({\goth g},\bi g{})$ is acyclic, whence the proposition. 
\end{proof}

\subsection{An equality of modules} \label{stp2}
Let ${\goth p}$ be a parabolic subalgebra of ${\goth g}$ containing ${\goth b}$ and
${\goth l}$ its reductive factor containing ${\goth h}$.
Denote by $\hat{\varepsilon }$ the map
$$ \xymatrix{G\times \lg l* \ar[rr]^-{\hat{\varepsilon }} && \ex {\b g{}}{{\goth g}}},
\qquad (g,x,y)\longmapsto g.\varepsilon (x,y) .$$
Then $\hat{\varepsilon }$ is a generator of the $\tk {\k}{\han {}{}}\k[\lg l*]$-module
$\ex {\b g{}}{\widehat{\bii}}$. For $k=1,\ldots,n$, let
$K_{k,{\goth p}}({\goth g},\widehat{\bii})$ be the kernel of the morphism 
$$\xymatrix{\tk {\k}{\han {}{}}\tk {\k}{\k[\lg l*]}D_{k,{\goth p}}^{k}({\goth g})
\ar[rr] && D_{k,{\goth p}}^{k}({\goth g},\widehat{\bii})[\b g{}]}, \qquad
\varphi \longmapsto \varphi \wedge \hat{\varepsilon }.$$ 

\begin{prop}\label{pstp2}
For $k=1,\ldots,n$, $K_{k,{\goth p}}({\goth g},\widehat{\bii})$ is the
intersection of $\ex {k-1}{{\goth g}}\wedge \widehat{\bii}$ and 
$\tk {\k}{\han {}{}}\tk {\k}{\k[\lg l*]}D_{k,{\goth p}}^{k}({\goth g})$.
\end{prop}

\begin{proof}
Since $\tk {\k}{\han {}{}}\tk {\k}{\k[\lg l*]}D_{k,{\goth p}}^{k}({\goth g})\cap 
\ex {k-1}{{\goth g}}\wedge \widehat{\bii}$ is clearly contained in 
$K_{k,{\goth p}}({\goth g},\widehat{\bii})$, it is sufficient to prove that 
$K_{k,{\goth p}}({\goth g},\widehat{\bii})$ is contained in 
$\ex {k-1}{{\goth g}}\wedge \widehat{\bii}$. Prove the assertion by 
induction on $k$. For $k=1$, it is true by Lema~\ref{l3co1}(iii) since
$\Omega _{{\goth g}}\cap \lg l*$ is a big open subset of $\lg l*$ by Lemma~\ref{lco4}(ii).

Suppose $k>1$ and $K_{j,{\goth p}}({\goth g},\widehat{\bii})$ contained in 
$\ex {j-1}{{\goth g}}\wedge \widehat{\bii}$ for $j<k$. By induction hypothesis,
the kernel $K_{k,{\goth p}}^{k-1}({\goth g},\widehat{\bii})$ of the map
$$\xymatrix{ \tk {\k}{\han {}{}}\tk {\k}{\k[\lg l*]}D_{k,{\goth p}}^{k-1}({\goth g})
\ar[rr] && D_{k,{\goth p}}^{k-1}({\goth g},\widehat{\bii})}, \qquad
\varphi \longmapsto \varphi \wedge \hat{\varepsilon }$$
is the intersection of
$\tk {\k}{\han {}{}}\tk {\k}{\k[\lg l*]}D_{k,{\goth p}}^{k-1}({\goth g})$ and
$D_{k-1}^{k-2}({\goth g})\wedge \widehat{\bii}$. 

We have a commutative diagram
$$\xymatrix{ & & 0 & 0 & \\
0 \ar[r] & K_{k,{\goth p}}({\goth g},\widehat{\bii}) \ar[r] & 
\tk {\k}{\han {}{}}\tk {\k}{\k[\lg l*]}D_{k,{\goth p}}^{k}({\goth g}) \ar[u] \ar[r] & 
D_{k,{\goth p}}^{k+\b g{}}({\goth g},\widehat{\bii}) \ar[u]\ar[r] & 0 \\ 
0 \ar[r] &  K_{k,{\goth p}}^{k-1}({\goth g},\widehat{\bii}) \ar[u]^{\dd} 
\ar[r] & \tk {\k}{\han {}{}}\tk {\k}{\k[\lg l*]}D_{k,{\goth p}}^{k-1}({\goth g})
\ar[u]^{\dd} \ar[r] & D_{k,{\goth p}}^{k-1+\b g{}}({\goth g},\widehat{\bii})
\ar[u]^{\dd} \ar[r] & 0 \\
& & \ar[u]^{\dd} 
\ar[r] \tk {\k}{\han {}{}}\tk {\k}{\k[\lg l*]}D_{k,{\goth p}}^{k-2}({\goth g})
\ar[u]^{\dd} \ar[r] & D_{k,{\goth p}}^{k-2+\b g{}}({\goth g},\widehat{\bii})
\ar[u]^{\dd} \ar[r] & 0 } .$$
By definition, the top row is exact. As already observed, so is the middle row.
By hypothesis, the right column is exact and by 
Lemma~\ref{lco3}(iii), the middle column is exact since the complex
$D_{k,{\goth p}}^{\bullet}({\goth g})$ is a direct factor of the complex
$D_{k}^{\bullet}({\goth g})$. Denoting by $\delta $ the
horizontal arrows, for $a$ in $K_{k,{\goth p}}({\goth g},\widehat{\bii})$, $a=\dd b$ for
some $b$ in $\tk {\k}{\han {}{}}\tk {\k}{\k[\lg l*]}D_{k,{\goth p}}^{k-1}({\goth g})$,
$\delta b =\dd c$ for some $c$ in
$D_{k,{\goth p}}^{k-2+\b g{}}({\goth g},\widehat{\bii})$ and $c=\delta c'$ 
for some $c'$ in
$\tk {\k}{\han {}{}}\tk {\k}{\k[\lg l*]}D_{k,{\goth p}}^{k-2}({\goth g})$, whence 
$b-\dd c'=\delta b'$ for some $b'$ in $K_{k,{\goth p}}^{k-1}({\goth g},\widehat{\bii})$
and $a=\dd b'$ is in $\ex {k-1}{{\goth g}}\wedge \widehat{\bii}$ since
$\dd D_{k-1,{\goth p}}^{k-2}({\goth g})\wedge \widehat{\bii} \subset 
D_{k-1,{\goth p}}^{k-1}({\goth g})\wedge \widehat{\bii}$, whence the proposition.
\end{proof}

\subsection{Equality of supports} \label{stp3}
Let ${\goth p}$ and ${\goth l}$ be as in Subsection~\ref{stp2}. We recall that
$\thetaup $ is the map
$$\xymatrix{ G\times \lg l* \ar[rr]^-{\thetaup } && \gg g{} }, \qquad
(g,x,y) \longmapsto (g(x),g(y)) .$$
By Lemma~\ref{lco4}(i), $\thetaup $ is a faithfully flat morphism. For any
$\sgg g{}$-module $M$, denote by ${\cal M}$ the restriction to 
$G({\goth l}_{*})\times {\goth g}$ of $\tk {\sgg g{}}{\an {\gg g{}}{}}M$ and 
$\overline{M}$ the space of global sections of $\thetaup ^{*}({\cal M})$.

\begin{lemma}\label{lstp3}
Let $M$ be a $\sgg g{}$-module and $N$ a submodule of $M$. The
$\tk {\k}{\an {}{}}\k[\lg l*]$-modules $\tk {\k[G]}{\an {}{}}\overline{M}$ 
and $\tk {\k[G]}{\an {}{}}\overline{N}$ are equal if and only if $\lg l*$ has an 
empty intersection with the support of $M/N$.
\end{lemma}

\begin{proof}
Denote by ${\cal M}'$ the restriction to $G({\goth l}_{*})\times {\goth g}$ of
$\tk {\sgg g{}}{\an {\gg g{}}{}}M/N$. As the localization functor is exact, we have a
short exact sequence
$$\xymatrix{ 0 \ar[r] & {\cal N} \ar[r] & {\cal M} \ar[r] & {\cal M}' \ar[r] & 0}.$$
Since $\thetaup $ is a faithfully flat morphism from 
$G\times {\goth l}_{*}\times {\goth g}$ to $G({\goth l}_{*})\times {\goth g}$
the short sequence
$$\xymatrix{ 0 \ar[r] & \thetaup ^{*}({\cal N}) \ar[r] & 
\thetaup ^{*}({\cal M}) \ar[r] & \thetaup ^{*}({\cal M}') \ar[r] & 0}$$
is exact. Then the short sequence
$$\xymatrix{ 0 \ar[r] & \tk {\k[G]}{\an {}{}}\overline{N} \ar[r] & 
\tk {\k[G]}{\an {}{}}\overline{M} \ar[r] &
\tk {\k[G]}{\an {}{}} \ar[r]\Gamma (G\times \lg l*,\thetaup ^{*}({\cal M}')) & 0}$$
is exact since $G\times \lg l*$ is affine and the localization functor is exact.
Hence $\tk {\k[G]}{\an {}{}}\overline{M}=\tk {\k[G]}{\an {}{}}\overline{N}$ if and only
if $\tk {\k[G]}{\an {}{}}\Gamma (G\times \lg l*,\thetaup ^{*}({\cal M}'))=0$, whence the
lemma since this last equality holds if and only if the support of ${\cal M}'$ has an
empty intersection with $\lg l*$ by faithfully flatness.
\end{proof}

For $k=1,\ldots,n$, let $\widetilde{K}_{k}({\goth g},\widetilde{\bii})$ be the
kernel of the map
$$ \xymatrix{ \tk {\k}{\an {}{}}\tk {\k}{\k[\lg l*]}D_{k}^{k}({\goth g})
\ar[rr] &&
D_{k}^{k+\b g{}}({\goth g},\widetilde{\bii})}, \qquad
\varphi \longmapsto \varphi \wedge \hat{\varepsilon }.$$

\begin{prop}\label{pstp3}
For $k=1,\ldots,n$, $\widetilde{K}_{k}({\goth g},\widetilde{\bii})$ is equal to
$\ex {k-1}{{\goth g}}\wedge \widetilde{\bii}$.
\end{prop}

\begin{proof}
We consider the action of $G$ in $\tk {\k}{\k[G\times \lg l*]}{\goth g}$ given by 
$$ k.a\tens v (g,x,y) := a(k^{-1}g,x,y) k.v .$$
This action has a natural extension to $\tk {\k}{\k[G\times \lg l*]}D({\goth g})$
and induces an action of ${\goth g}$ in
$\tk {\k}{{\cal O}}\tk {\k}{\k[\lg l*]}D({\goth g})$. Under this action,
$\hat{\varepsilon }$ is invariant. As a result,
$\ex {k-1}{{\goth g}}\wedge \ex {\b g{}}{\widetilde{\bii}}$ and
$\widetilde{K}_{k}({\goth g},\widetilde{\bii})$ are invariant ${\goth g}$-sumodules of
$\tk {\k}{{\cal O}}\tk {\k}{\k[\lg l*]}D({\goth g})$. Moreover,
$\ex {k-1}{{\goth g}}\wedge \ex {\b g{}}{\widetilde{\bii}}$ is contained in
$\widetilde{K}_{k}({\goth g},\widetilde{\bii})$. Set:
$$ E' := \tk {\k}{{\cal O}}\tk {\k}{\k[\lg l*]}D_{k,{\goth p}}^{k}({\goth g}), \quad 
N' := \widetilde{K}_{k}({\goth g},\widetilde{\bii}) ,$$
$$M := \tk {\k}{\k[G]}\tk {\k}{\k[\lg l*]}\ex {k}{{\goth g}}, \quad
M' := \tk {\k}{\k[G]}\tk {\k}{\k[\lg l*]}D_{k,{\goth p}}^{k}({\goth g}), \quad
N := N'\cap M .$$
Then $M$ is a rational ${\goth g}$-module since so is $\k[G]$, $M'$ is a
${\goth l}$-submodule of $M$, $N$ is a ${\goth g}$-submodule of $M$ such that
$N'=\tk {\k[G]}{\an {}{}}N$. By \cite[Theorem 1.1]{Ch}, $D_{k,{\goth p}}^{k}({\goth g})$
generates the $G$-module $\ex k{{\goth g}}$. So $M'$ generates the $G$-module $M$ since
$\tk {\k}{\k[G]}\k[\lg l*]$ is a $G$-module. Then, by Proposition~\ref{prep},
$N\cap M'$ generates the $G$-module $N$. Since $\han {}{}$ is a faithfully flat
extension of $\an {}{}$, $N'\cap \tk {\k[G]}{\an {}{}}M'$ is contained in
$\ex {k-1}{{\goth g}}\wedge \widetilde{\bii}$ by Proposition~\ref{pstp2}. Then
$N'=\ex {k-1}{{\goth g}}\wedge \widetilde{\bii}$ since
$N'\cap \tk {\k[G]}{\an {}{}}M'=\tk {\k[G]}{\an {}{}}(N\cap M')$.
\end{proof}

For $k=1,\ldots,n$, let $K_{k}({\goth g})$ be the kernel of the map
$$ \xymatrix{\tk {\k}{\sgg g{}}\ex k{{\goth g}} \ar[rr] &&
\ex k{{\goth g}}\wedge \ex {\b g{}}{\bi g{}}}, \qquad
\varphi \longmapsto \varphi \wedge \varepsilon .$$
Then $\ex {k-1}{{\goth g}}\wedge \ex {\b g{}}{\bi g{}}$ is a $\sgg g{}$-submodule of
$K_{k}({\goth g})$.

\begin{coro}\label{cstp3}
For $k=1,\ldots,n$, denote by $S_{k,k}$ the support in $\gg g{}$ of the quotient of
$K_{k}({\goth g})$ by $\ex {k-1}{{\goth g}}\wedge \ex {\b g{}}{\bi g{}}$. Then
$S_{k,k}\cap G({\goth l}_{*})\times {\goth g}$ is empty.
\end{coro}

\begin{proof}
As $\bi g{}$ is a $G$-invariant $\sgg g{}$-submodule of $\tk {\k}{\sgg g{}}{\goth g}$,
$K_{k}({\goth g})$ and $\ex {k-1}{{\goth g}}\wedge \ex {\b g{}}{\bi g{}}$ are
$G$-invariant submodules of $\tk {\k}{\sgg g{}}\ex {}{{\goth g}}$. Then $S_{k,k}$ is
a $G$-invariant subvariety of $\gg g{}$.

Let ${\cal K}_{k}$ and ${\cal K}'_{k}$ be the restrictions to
$G({\goth l}_{*})\times {\goth g}$ of
$\tk {\sgg g{}}{\an {\gg g{}}{}}K_{k}({\goth g})$ and
$\tk {\sgg g{}}{\an {\gg g{}}{}}\ex {k-1}{{\goth g}}\wedge \ex {\b g{}}{\bi g{}}$
respectively. By Proposition~\ref{pstp3},
$$ \tk {\k[G]}{\an {}{}}\Gamma (G\times \lg l*,\thetaup ^{*}({\cal K}_{k})) =
\tk {\k[G]}{\an {}{}}\Gamma (G\times \lg l*,\thetaup ^{*}({\cal K}'_{k})) .$$
So, by Lemma~\ref{lstp3}, $S_{k,k}\cap G({\goth l}_{*})\times {\goth g}$ is empty since
$S_{k,k}$ is $G$-invariant, whence the corollary. 
\end{proof}

\subsection{Theorem of reduction} \label{stp4}
For $z$ in ${\goth h}$, the centralizer ${\goth g}^{z}$ of $z$ in ${\goth g}$ is a
reductive subalgebra of ${\goth g}$.

\begin{lemma}\label{lstp4}
Let $z$ be in ${\goth h}\setminus \{0\}$. Then, for some $g$ in the normalizer of
${\goth h}$ in $G$, ${\goth g}^{g(z)}$ is the reductive factor containing ${\goth h}$ of
a parabolic subalgebra of ${\goth g}$ containing ${\goth b}$.
\end{lemma}

\begin{proof}
As $z$ is in ${\goth h}$, ${\goth h}$ is contained in ${\goth g}^{z}$. For some
Borel subalgebra ${\goth b}_{z}$ of ${\goth g}$, containing ${\goth h}$,
${\goth b}_{z}\cap {\goth g}^{z}$ is a Borel subalgebra of ${\goth g}^{z}$. Since
${\goth h}$ is contained in ${\goth b}_{z}$, for some $g$ in the normalizer of
${\goth h}$ in $G$, $g({\goth b}_{z})$ is equal to ${\goth b}$. Denoting by
${\cal R}_{+,g(z)}$ the set of positive roots $\alpha $ such that $\alpha (g(z))\neq 0$
and setting
$${\goth u}_{g(z)} := \bigoplus _{\alpha \in {\cal R}_{+,g(z)}} {\goth g}_{\alpha },$$
${\goth g}^{g(z)}$ normalizes ${\goth u}_{g(z)}$ since ${\goth g}^{g(z)}$ is the
subalgebra of ${\goth g}$ generated by ${\goth h}$ and the root subspaces
${\goth g}_{\pm \beta }, \, \beta \in \Pi \setminus {\cal R}_{+,g(z)}$. As a result,
${\goth g}^{g(z)}+{\goth u}_{g(z)}$ is a subalgebra of ${\goth g}$ containing
${\goth b}$, whence the lemma.
\end{proof}

\begin{theo}\label{tstp4}
Suppose that $D_{k,{\goth p}}^{\bullet}({\goth g},\widehat{\bii})$ has no cohomology
of degree different from $\b g{}$ for $k=1,\ldots,n$ and each parabolic subalgebra
${\goth p}$ of ${\goth g}$ containing ${\goth b}$. Then ${\goth g}$ has Property
\rm $({\bf P})$.
\end{theo}
 
\begin{proof}
Let $z$ be in ${\goth h}\setminus \{0\}$ and $g$ as in Lemma~\ref{lstp4}. By hypothesis
and Corollary~\ref{cstp3}, for some affine open subset $O$ of $\gg g{}$, containing
$\{g(z)\}\times {\goth g}$,
$$ \tk {\sgg g{}}{\k[O]}K_{k}({\goth g})=
\tk {\sgg g{}}{\k[O]}\ex {k-1}{{\goth g}}\wedge \bi g{},$$
for $k=1,\ldots,n$. Then, by Proposition~\ref{pco2},
$\tk {\sgg g{}}{\k[O]}D_{k}^{\bullet}({\goth g},\bi g{})$ has no cohomology of
degree different from $\b g{}$ for $k=1,\ldots,n$. As a result,
$\varpi _{1}(S_{k})\cap {\goth h}=\{0\}$ for $k=1,\ldots,n$. By
Proposition~\ref{ppp}(ii), $D_{1}^{\bullet}({\goth g},\bi g{})$ has no cohomology
of degree different from $\b g{}$. Then, by induction on $k$ and Proposition~\ref{pstp1},
$D_{k}^{\bullet}({\goth g},\bi g{})$ has no cohomology
of degree different from $\b g{}$ for $k=2,\ldots,n$, whence the theorem.
\end{proof}

\section{Some results on the characteristic module} \label{rcm}
From now on ${\goth p}$ is a parabolic subalgebra of ${\goth g}$ containing ${\goth b}$.
Let ${\goth l}$ be the reductive factor of ${\goth p}$ containing ${\goth h}$,
${\goth d}$ the derived algebra of ${\goth l}$ and ${\goth z}$ its center.
Denote by ${\cal R}_{{\goth l}}$ the set of roots $\alpha $ such that
${\goth g}_{\alpha }$ is contained in ${\goth l}$ and set:
$$ d_{0} := \dim {\goth z}, \quad {\goth p}_{\u} := 
\bigoplus _{\alpha \in {\cal R}_{+}\setminus {\cal R}_{{\goth l}}} {\goth g}_{\alpha },
\qquad d := \dim {\goth p}_{\u} ,$$
$${\goth p}_{-,\u} := \bigoplus _{\alpha \in {\cal R}_{+}\setminus {\cal R}_{{\goth l}}}
{\goth g}_{-\alpha }, \quad {\goth p}_{-} := {\goth l}\oplus {\goth p}_{-,\u},
\quad {\goth p}_{\pm,\u} := {\goth p}_{\u} \oplus {\goth p}_{-,\u} $$
so that ${\goth p}_{\u}$ is the nilpotent radical of ${\goth p}$, ${\goth p}_{-}$
is the parabolic subalgebra opposite to ${\goth p}$ and ${\goth p}_{-,\u}$ is its
nilpotent radical. Let $L$ be the centralizer of ${\goth z}$ in $G$.
According to~\cite[\S 3.2, Lemma 5]{Ko}, $L$ is connected. When ${\goth b}$ is strictly
contained in ${\goth p}$, we denote by $\poi {{\goth d}}1{,\ldots,}{\n}{}{}{}$ the simple
factors of ${\goth d}$.    

Let ${\goth l}_{*}$ be the open subset of ${\goth l}$ as in Subsection~\ref{lco4}. The
usual gradation of $\k[{\goth p}]$ induces a gradation of the polynomial algebra $\k[\lp]$
over $\k[{\goth l}_{*}]$. Let $\varpi $ be the canonical projection 
$\xymatrix{ {\goth p} \ar[r] & {\goth l}}$ and set:
$$\oi := \tk {\sgg g{}}{\k[\lp]}\bi g{}, \quad  
\oi _{{\goth l}} := \tk {\sgg l{}}{\k[\ll]}\bi l{} .$$
As $(h,e)$ is in $\Omega _{{\goth g}}$, $\Omega _{{\goth g}}\cap \lp$ is a dense open 
subset of $\lp$ so that $\oi$ is a free submodule of rank $\b g{}$ of 
$\tk {\k}{\k[\lp]}{\goth p}$ by Proposition~\ref{psc1}(iii), Theorem~\ref{tsc1}(i) and
Corollary~\ref{c2sc3}. Again by Proposition~\ref{psc1}(iii), $\oi _{{\goth l}}$ is a
free $\k[\ll]$-module of rank $\b l{}$. From the direct sums 
$$ {\goth p} = {\goth l}\oplus {\goth p}_{\u} \quad  \text{and} \quad
{\goth g} = {\goth p}\oplus {\goth p}_{-,\u}$$
we deduce the inclusions
$$ \k[{\goth l}] \subset \k[{\goth p}] \subset \k[{\goth g}] \quad  \text{and} \quad
\k[\ll] \subset \k[\lp] \subset \k[\lg l*].$$
For $(i,m)$ in $I_{0}$, the restriction of $\varepsilon _{i}^{(m)}$ to 
$\lp$ is again denoted by $\varepsilon _{i}^{(m)}$. According to this convention, 
$\varepsilon $ is a generator of the $\k[\lp]$-module $\ex {\b g{}}{\oi}$ 
(see Equality~\ref{eqpm}). As in Subsection~\ref{sc3}, for $(x,y)$ 
in $\gg l{}$, the image of $\bi l{}$ by the evaluation map at $(x,y)$ is denoted by 
$V_{x,y}^{{\goth l}}$.  

Denote by $\poi {\beta }1{,\ldots,}{{d_{0}}}{}{}{}$ the simple roots in 
$\Pi \setminus {\cal R}_{{\goth l}}$ and $\poi h1{,\ldots,}{d_{0}}{}{}{}$ the basis of 
${\goth z}$ dual of $\poi {\beta }1{,\ldots,}{d_{0}}{}{}{}$. Let 
$\poi q1{,\ldots,}{\rg}{}{}{}$ be homogeneous generators of $\e Sl^{L}$ and 
$\poie d1{,\ldots,}{\rg}{}{}{}{\prime}{\prime}$ their respective degrees, chosen so that 
\begin{itemize}
\item [{\rm (1)}] $\poie d1{ \leq \cdots \leq }{\rg}{}{}{}{\prime}{\prime}$,
\item [{\rm (2)}] for $i=1,\ldots,\rg$, $q_{i}\in 
\e Sz \cup \ec Sd{}1{}^{L}\cup \cdots \cup \ec Sd{}{\n}{}^{L}$,
\item [{\rm (3)}] for $i=1,\ldots,d_{0}$, $q_{i}=h_{i}$.
\end{itemize}
For $i=1,\ldots,\rg$, denote by $\eta _{i}$ the element of
$\tk {\k}{\k[{\goth l}]}{\goth l}$ equal to the differential of $q_{i}$. 

\begin{lemma}\label{lrcm}
{\rm (i)} For $i=1,\ldots,\rg$, there exists a unique sequence 
$c_{i,j}, \, j=1,\ldots,\rg$ in $\k[{\goth l}]$ such that 
$$ \varepsilon _{i}(x) = \sum_{j=1}^{\rg} c_{i,j}(x) \eta _{j}(x) $$
for all $x$ in ${\goth l}$. Moreover, $c_{i,j}$ is invariant under $L$ and homogeneous of
degree $d_{i}-d'_{j}$.

{\rm (ii)} For $x$ in ${\goth l}_{*}$, the matrix
$$ (c_{i,j}(x), \, 1\leq i,j\leq \rg)$$
is invertible.

{\rm (iii)} For all $x$ in a dense open subset of ${\goth h}$,
$\poi x{}{,\ldots,}{}{\varepsilon }{1}{\rg}$ are regular elements of ${\goth g}$. In
particular, for $i=1,\ldots,\rg$ and $x$ in a dense open subset of ${\goth l}_{*}$,
$\varepsilon _{i}(x)$ has a non zero component on each simple factor of ${\goth d}$. 
\end{lemma}

\begin{proof}
(i) Let $i=1,\ldots,\rg$. Denote by ${\goth l}_{{\goth l},\r}$ the set of regular
elements in ${\goth l}$. For $x$ in ${\goth l}_{{\goth l},\r}$, 
$\poi x{}{,\ldots,}{}{\eta }{1}{\rg}$ is a basis of ${\goth l}^{x}$ by 
\cite[Theorem 9]{Ko}. By Lemma~\ref{lsc3}(i), for all $x$ in ${\goth l}$, 
$\varepsilon _{i}(x)$ is in ${\goth l}^{x}$. Then, by Claim~\ref{clsc1}, there exists a
unique sequence $c_{i,j},\, j=1,\ldots,\rg$ in $\k[{\goth l}_{{\goth l},\r}]$ such that 
$$ \varepsilon _{i}(x) = \sum_{j=1}^{\rg} c_{i,j}(x) \eta _{j}(x) $$
for all $x$ in ${\goth l}_{{\goth l},\r}$. By \cite{Ve}, ${\goth l}_{{\goth l},\r}$ is a
big open subset of ${\goth l}$. Then $\poi c{i,1}{,\ldots,}{i,\rg}{}{}{}$ have a regular
extension to ${\goth l}$ since ${\goth l}$ is normal, whence 
$$ \varepsilon _{i}(x) = \sum_{j=1}^{\rg} c_{i,j}(x) \eta _{j}(x) $$
for all $x$ in ${\goth l}$. As $\varepsilon _{i}$ and $\eta _{j}$ are equivariant under
$L$ and homogeneous of degree $d_{i}-1$ and $d'_{j}-1$ respectively, $c_{i,j}$ is 
invariant under $L$ and homogeneous of degree $d_{i}-d'_{j}$.

(ii) For $x$ in ${\goth l}_{*}$, ${\goth g}^{x}$ is contained in ${\goth l}$. Then 
${\goth l}_{*}\cap {\goth l}_{{\goth l},\r}$ is contained in ${\goth g}_{\r}$. As a
result, by Claim~\ref{clsc1}, for $i=1,\ldots,\rg$, there exists a unique sequence
$c'_{i,j}, \, j=1,\ldots,\rg$ in $\k[{\goth l}_{*}\cap {\goth l}_{{\goth l},\r}]$ such
that
$$ \eta _{i}(x) = \sum_{j=1}^{\rg} c'_{i,j}(x) \varepsilon _{j}(x)$$
for all $x$ in ${\goth l}_{*}\cap {\goth l}_{{\goth l},\r}$ since 
$\poi x{}{,\ldots,}{}{\varepsilon }{1}{\rg}$ is a basis of ${\goth g}^{x}$ fo $x$ in 
${\goth g}_{\r}$ by \cite[Theorem 9]{Ko}. As ${\goth l}_{*}\cap {\goth l}_{{\goth l},\r}$
is a big open subset of ${\goth l}_{*}$, 
$\poie c{{i,1}}{,\ldots,}{{i,\rg}}{}{}{}{\prime}{\prime}$ have a regular extension to 
${\goth l}_{*}$. Then, for $i=1,\ldots,\rg$ and $x$ in ${\goth l}_{*}$, 
$$ \eta _{i}(x) = \sum_{j=1}^{\rg} c'_{i,j}(x) \varepsilon _{j}(x) ,$$
whence the assertion.

(iii) Suppose that $\varepsilon _{i}(x)$ is not regular for all $x$ in ${\goth h}$ for
some $i=1,\ldots,\rg$. A contradiction is expected. Then, for some root $\alpha $, 
$\alpha \rond \varepsilon _{i}(x)=0$ for all $x$ in ${\goth h}$. In particular, for $x$ in
${\goth h}$ and $g$ in $W({\cal R})$, $\alpha \rond \varepsilon _{i}(g(x)) = 0$. As 
$\varepsilon _{i}$ is $G$-equivariant, $\varepsilon _{i}(x)$ is in the kernel of 
$g^{-1}(\alpha )$. As ${\goth g}$ is simple, ${\goth h}$ is a simple $W({\cal R})$-module
so that the orthogonal complement in ${\goth h}$ to $g(\alpha ), \, g\in W({\cal R})$ is 
equal to $\{0\}$, whence a contradiction since for $x$ in ${\goth h}_{\r}$, 
$\varepsilon _{i}(x) \neq 0$ as an element of a basis of ${\goth h}$ by 
\cite[Theorem 9]{Ko}.

As a result, for $i=1,\ldots,\rg$ and all $x$ in a dense open subset of ${\goth l}_{*}$,
$\varepsilon _{i}(x)$ is in ${\goth g}_{\r}\cap {\goth l}_{*}$. Let $x$ be in
${\goth l}_{*}$ such that $\varepsilon _{i}(x)$ is a regular element of ${\goth g}$. As
${\goth l}^{\varepsilon _{i}(x)}$ is a commutative algebra,
$\poi {{\goth d}}1{,\ldots,}{\n}{}{}{}$ are not contained in
${\goth l}^{\varepsilon _{i}(x)}$. Hence the components of
$\varepsilon _{i}(x)$ on $\poi {{\goth d}}1{,\ldots,}{\n}{}{}{}$ are different from $0$.  
\end{proof}

For $a$ homogeneous of degree $d_{a}$ in $\k[{\goth l}]$ and for $k=0,\ldots,d_{a}$, 
denote by $a^{(k)}$ the $2$-polarization of bidegree $(d_{a}-k,k)$ of $a$. Set: 
$$ I'_{0} := \{(i,m) \in \{1,\ldots,\rg\}\times {\Bbb N} \; \vert \; 
0\leq m \leq d'_{i}-1\} .$$
Then $\vert I'_{0} \vert = \b l{}$. For $(i,m)$ in $I'_{0}$, let $\eta _{i}^{(m)}$ be the
$2$-polarization of $\eta _{i}$ of bidegree $(d'_{i}-1-m,m)$. For $(i,m)$ in $I_{0}$ and
$(j,m')$ in $I'_{0}$, set $c_{i,m,j,m'} := c_{i,j}^{(m-m')}$. In particular,
$c_{i,m,j,m'}=0$ if $m'>m$ or $m-m'>d_{i}-d'_{j}$ since $c_{i,j}$ is homogeneous of
degree $d_{i}-d'_{j}$.

\begin{lemma}\label{l2rcm}
{\rm (i)} For $(i,m)$ in $I_{0}$, 
$$ \varepsilon _{i}^{(m)}(x,y) = \sum_{(j,m') \in I'_{0}} 
c_{i,m,j,m'}(x,y) \eta _{j}^{(m')}(x,y)$$
for all $(x,y)$ in $\gg l{}$. 

{\rm (ii)} For $(i,m)$ in $I_{0}$ and $(j,m')$ in $I'_{0}$, the function 
$c_{i,m,j,m'}$ is invariant under the diagonal action of $L$ in $\gg l{}$.

{\rm (iii)} For $(i,m)$ in $I_{0}$ and $y'$ in ${\goth p}_{\u}$, 
$$\varpi \rond \varepsilon _{i}^{(m)}(x,y+y') = \varepsilon _{i}^{(m)}(x,y)$$
for all $(x,y)$ in $\gg l{}$.
\end{lemma}

\begin{proof}
(i) By Lemma~\ref{lrcm}(i), for $(x,y)$ in $\gg l{}$,
$$ \varepsilon _{i}(x+sy) = \sum_{j=1}^{\rg} c_{i,j}(x+sy) \eta _{j}(x+sy)$$
for all $s$ in $\k$, whence
$$ \sum_{m=0}^{d_{i}-1} s^{m}\varepsilon _{i}^{(m)}(x,y) = 
\sum_{l=0}^{d_{i}-d'_{j}} \sum_{k=0}^{d'_{j}-1} s^{l+k} c_{i,j}^{(l)}(x,y)
\eta _{j}^{(k)}(x,y) = \sum_{m=0}^{d_{i}-1} s^{m} \sum_{m'=0}^{\inf\{m,d'_{j}-1\}} 
c_{i,j}^{(m-m')}(x,y)\eta _{j}^{(m')}(x,y) .$$
As a result,
$$ \varepsilon _{i}^{(m)}(x,y) = \sum_{(j,m')\in I'_{0}} 
c_{i,m,j,m'}(x,y)\eta _{j}^{(m')}(x,y) $$
for all $(x,y)$ in $\gg l{}$.

(ii) By Lemma~\ref{lrcm}(i), $c_{i,m,j,m'}$ is invariant under the diagonal action of 
$L$ in $\gg l{}$.

(iii) According to Corollary~\ref{csc3}, for all $x$ in ${\goth p}$, 
$\varepsilon _{i}(x)$ is in ${\goth p}$. For $x$ in ${\goth l}$ and $y$ in 
${\goth p}_{\u}$, 
$$ \lim _{t\rightarrow 0} \rho (t)(x+y) = x \quad  \text{whence} \quad 
p_{i}(x+y) = p_{i}(x) \quad  \text{and} \quad \dv {\varepsilon _{i}(x)}v =
\dv {\varepsilon _{i}(x+y)}v$$
for all $v$ in ${\goth l}$. As a result, for all $x$ in ${\goth l}$ and for all $y$ in 
${\goth p}_{\u}$, $\varepsilon _{i}(x)-\varepsilon _{i}(x+y)$ is in ${\goth p}_{\u}$
since ${\goth p}_{\u}$ is the orthogonal complement to ${\goth l}$ in ${\goth p}$, 
whence the assertion.
\end{proof}

The order of $I_{0}$ induces an order of $I'_{0}$ and we get a square matrix of order
$\vert I'_{0} \vert$,
$$ M_{0} := (c_{i,m,j,m'}, \; ((i,m),(j,m')) \in I'_{0}\times I'_{0}),$$
with coefficients in $\k[\gg l{}]$.

\begin{coro}\label{crcm}
For all $(x,y)$ in $\ll$, $\det M_{0}(x,y)\neq 0$.
\end{coro}

\begin{proof}
Set:
$$ \varepsilon _{0} := \wedge _{(i,m) \in I'_{0}} \varepsilon _{i}^{(m)} 
\quad  \text{and} \quad \eta _{0} := \wedge _{(j,m')\in I'_{0}} \eta _{j}^{(m')}.$$
In these equalities, the order of the products is induced by the order of $I'_{0}$. Let 
$\overline{\varepsilon _{0}}$ be the restriction of $\varepsilon _{0}$ to $\ll$. 
By Lemma~\ref{l2rcm}(i), $\overline{\varepsilon _{0}}$ is in 
$\tk {\k}{\k[\ll]}\ex {\b l{}}{{\goth l}}$ and 
$$ \overline{\varepsilon _{0}} = \det M_{0} \eta _{0} .$$
As $\overline{\varepsilon _{0}}$ and $\eta _{0}$ are homogeneous of degree
$$ \sum_{i=1}^{\rg} \frac{d'_{i}(d'_{i}-1)}{2},$$
$\det M_{0}$ is in $\k[{\goth l}_{*}]$.  

Denote by $\Sigma $ the nullvariety of $\det M_{0}$ in ${\goth l}_{*}$. Suppose that 
$\Sigma $ is nonempty. A contradiction is expected. As the maps $\varepsilon _{i}^{(m)}$
and $\eta _{i}^{(m)}$ are homogeneous and invariant under $L$ for all $(i,m)$ in 
$I'_{0}$, $\Sigma $ is a closed cone of ${\goth l}_{*}$, invariant under $L$.  

\begin{claim}\label{clrcm}
For some $x$ in ${\goth h}\cap {\goth l}_{*}$, $x+\varpi (e)$ is in $\Sigma $.
\end{claim}

\begin{proof}{[Proof of Claim~\ref{clrcm}]}
As $\Omega _{{\goth l}}\cap \ll$ contains $(h,\varpi (e))$, for all $t$ in $\k$,
$h+t\varpi (e)$ is a regular element of ${\goth g}$ and ${\goth l}$. Then, by
Lemma~\ref{lrcm} and Claim~\ref{clsc1}, for some sequence
$p_{i,j}, \, 1\leq i,j\leq \rg$ of polynomials of one indeterminate over $\k$,
$$ \eta _{i}(h+t\varpi (e)) = p_{i,1}(t)\varepsilon _{1}(h+t\varpi (e)) + \cdots +
p_{i,\rg}(t)\varepsilon _{\rg}(h+t\varpi (e)$$
for all $t$ in $\k$ and $i=1,\ldots,\rg$. As a result, for all $(i,m)$ in $I'_{0}$,
$\eta _{i}^{(m)}(h,\varpi (e))$ is a linear combination of
$\varepsilon _{j}^{(m')}(h,\varpi (e), \ j=1,\ldots,\rg,\, 0\leq m'\leq m$. Hence
${\goth h}\cap {\goth l}_{*}$ is not contained in $\Sigma $. Let $\Sigma '$ be an
irreducible component of $\Sigma \cap {\goth b}$. Then $\Sigma '$ is an hypersurface of
${\goth b}\cap {\goth l}_{*}$ as an irreducible component of the nullvariety of
$\det M_{0}$ in ${\goth b}\cap {\goth l}_{*}$. Then, by lemma~\ref{lint},
$$ \Sigma ' = \Sigma '\cap {\goth h} + {\goth u}\cap {\goth l} $$
since ${\goth b}\cap {\goth l}_{*}={\goth h}\cap {\goth l}_{*}+{\goth u}\cap {\goth l}$
and $\Sigma $ and $\Sigma '$ are invariant under the one parameter subgroup
$t\mapsto \rho (t)$ of $L$. As a result, for $x$ in $\Sigma '\cap {\goth h}$,
$x+\varpi (e)$ is in $\Sigma '$, whence the claim.
\end{proof}

From the equalities
$$ \lim _{t\rightarrow \infty } t^{-2}\rho (t).(x+\varpi (e)) = \varpi (e)
\quad  \text{and} \quad \lim _{t\rightarrow 0} \rho (t).(x+\varpi (e)) = x,$$
we deduce $x+\varpi (e)\in {\goth l}_{{\goth l},\r}\cap {\goth l}_{*}$ since  
${\goth l}_{{\goth l},\r}$ and ${\goth l}_{*}$ are open subsets of ${\goth l}$,
invariant under $H$, containing $\varpi (e)$ and $x$ respectively. Hence $x+\varpi (e)$
is a regular element of ${\goth g}$. Then, from the equalities
$$ \lim _{t\rightarrow 0} \rho (t).(h+e-\varpi (e)) = h \quad  \text{and} $$ $$
\lim _{t\rightarrow \infty } t^{-2}\rho (t).(x+\varpi (e)+a(h+e-\varpi (e))) = 
\varpi (e)+a(e-\varpi (e)),$$ 
for $a$ in $\k$, we deduce  
$$h+e-\varpi (e) \in {\goth g}_{\r} \quad  \text{and} \quad
x+\varpi (e)+a(h+e-\varpi (e)) \in {\goth g}_{\r}$$
for all $a$ in $\k^{*}$ since $\varpi (e)+a(e-\varpi (e))$ is a regular element of 
${\goth g}$. Hence 
$$ (x+\varpi (e),h+e-\varpi (e)) \in \Omega _{{\goth g}} .$$
As a result, by Corollary~\ref{c2sc3} and Lemma~\ref{l2rcm}(iii), the elements
$$ \varepsilon _{i}^{(m)}(x+\varpi (e),h), \, (i,m) \in I'_{0}$$
are linearly independent, whence the contradiction.
\end{proof}

Let $\oi _{0}$ be the submodule of $\oi$ generated by 
$\varepsilon _{i}^{(m)}, \, (i,m) \in I'_{0}$. For $(i,m)$ in $I'_{0}$, denote by 
$\overline{\varepsilon _{i}^{(m)}}$ the restriction of $\varepsilon _{i}^{(m)}$ to $\ll$.
By Lemma~\ref{lsc3}(ii), $\overline{\varepsilon _{i}^{(m)}}$ is in 
$\tk {\k}{\k[\ll]}{\goth l}$.

\begin{prop}\label{prcm}
Let $\oi _{+}$ be the $\k[\lp]$-submodule of $\tk {\k}{\k[\lp]}{\goth g}$ generated by
$\oi_{{\goth l}}$ and $1\tens {\goth p}_{\u}$.

{\rm (i)} The $\k[\lp]$-modules $\oi$ and $\oi _{0}$ are free of rank $\b g{}$ and 
$\b l{}$ respectively.

{\rm (ii)} The module $\oi _{{\goth l}}$ is the image of $\oi _{0}$ by the restriction 
map from $\lp$ to $\ll$. In particular, 
$\overline{\varepsilon _{i}^{(m)}}, \, (i,m)\in I'_{0}$ is a basis of $\oi _{{\goth l}}$.

{\rm (iii)} The $\k[\lp]$-module $\oi_{+}$ is free of rank $\b g{}$. Moreover, $\oi$ is
contained in $\oi _{+}$.
\end{prop}

\begin{proof}
(i) As $\Omega _{{\goth g}}\cap \lp$ is non empty and $\oi$ is generated by 
$\varepsilon _{i}^{(m)}, \, (i,m) \in I_{0}$, the assertion results from Proposition
\ref{psc1}(iii) since $\vert I_{0} \vert = \b g{}$ and $\vert I'_{0} \vert = \b l{}$. 

(ii) By Lemma~\ref{l2rcm}(i), the restriction of $\oi _{0}$ to $\ll$ is contained in 
$\oi _{{\goth l}}$. By Corollary~\ref{crcm}, the matrix $M_{0}(x,y)$ is invertible for 
all $(x,y)$ in $\ll$. Then, for all $(i,m)$ in $I'_{0}$, $\eta _{i}^{(m)}$ is a linear 
combination with coefficients in $\k[\ll]$ of 
$\overline{\varepsilon _{j}^{(m')}}, \, (j,m') \in I'_{0}$, whence the assertion.

(iii) By definition, $\oi_{+}$ is the direct sum of the free module
$\tk {\k[\ll]}{\k[\lp]}\oi _{{\goth l}}$ of rank $\b l{}$ and the free module
$\tk {\k}{\k[\lp]}{\goth p}_{\u}$ of rank $d$. Hence $\oi_{+}$ is free of rank
$\b g{}$ since $\b g{}=\b l{}+d$. By Lemma~\ref{l2rcm}(iii) and (ii),
$\varepsilon _{i}^{(m)}$ is in $\oi_{+}$ for all $(i,m)$ in $I_{0}$, whence the
assertion.
\end{proof} 

\section{Decomposition and Expansion} \label{de}
According to the notations of Subsection~\ref{co4},
$\bii := \tk {\sgg g{}}{\k[\lg l{*}]}\bi g{}$. Let $I_{*,0}$ be the set
$I_{0}\setminus \{1,\ldots,\rg\}\times \{0\}$.

\subsection{Some notations}\label{de1}
Denote by $\poi {\alpha }{1}{,\ldots,}{n}{}{}{}$ the positive roots ordered so that
$\poi {\alpha }{1}{,\ldots,}{d}{}{}{}$ are not in ${\cal R}_{{\goth l}}$. For
$1\leq i,j \leq n$, set: 
$$v_{i}:= x_{\alpha _{i}}, \quad w_{j}:= x_{-\alpha _{j}} .$$ 
Then $\poi v1{,\ldots,}{d}{}{}{}$, $\poi w1{,\ldots,}{d}{}{}{}$,
$\poi v1{,\ldots,}{n}{}{}{}$, $\poi w1{,\ldots,}{n}{}{}{}$ are basis of ${\goth p}_{\u}$,
${\goth p}_{-,\u}$, ${\goth u}$, ${\goth u}_{-}$ respectively. 

Let $\poi {{\goth d}}1{,\ldots,}{\n}{}{}{}$ be the simple factors of ${\goth d}$
when ${\goth z}$ is strictly contained in ${\goth l}$. Set:
$$ n_{1} := \b d1-\j d1, \; ,\ldots,\; n_{\n} := \b d{\n} - \j d{\n} $$ 
so that 
$$ n-d  = \poi n1{+\cdots +}{\n}{}{}{} .$$
As ${\goth g}$ is the direct sum of ${\goth z}$, 
$\poi {{\goth d}}1{,\ldots,}{\n}{}{}{}$, ${\goth p}_{\u}$, ${\goth p}_{-,\u}$, for 
$i=1,\ldots,\n$, $\k[{\goth d}_{i}]$ is a subalgebra of $\k[{\goth l}]$ and 
$\k[{\goth g}]$ and $\k[{\goth l}_{*}\times {\goth d}_{i}]$ is a subalgebra of 
$\k[\lg l*]$. Set
$\bii _{i} := \tk {\k[\gg di]}{\k[\lp]}\bi di$. Then $\bii _{i}$ is a free submodule
of rank $\b d{i}$ of $\tk {\k}{\k[\lp]}{\goth d}_{i}$. According to 
Lemma~\ref{lrcm}(ii), for some $\poi {\lambda }{i,1}{,\ldots,}{i,n_{i}}{}{}{}$ in 
$\tk {\sgg di}{\k[\lp]}\bi di$, the $\k[\lp]$-module $\bii _{i}$ is generated by
$\poi {\lambda }{i,1}{,\ldots,}{i,n_{i}}{}{}{}$ and $\j di$ elements of
$\tk {\k}{\k[{\goth l}_{*}]}{\goth l}$. 

Let ${\goth I}$ be the union of $\{0\}$ and the set of strictly increasing sequences in
$\{1,\ldots,n\}$. For $\kappa = \poi k1{<\cdots <}{l}{}{}{}$ in ${\goth I}$ and
$j=0,\ldots,n$, set:
$$ \{\kappa \} := \{\poi k1{,\ldots,}{l}{}{}{}\}, \qquad \vert \kappa \vert = l, \qquad
w_{\kappa } := \poi w{k_{1}}{ \wedge \cdots \wedge }{k_{l}}{}{}{}, \qquad
{\goth I}_{j} := \{\iota \in {\goth I} \; \vert \; \vert \iota  \vert = j\}.$$
By definition, $w_{\kappa }$ is in $\ex {\vert \kappa \vert}{{\goth u}_{-}}$ and  for
$\kappa  = 0$, $\vert \kappa \vert := 0$ and $w_{\kappa }=1$.
Denote by $\mu _{+}$ the generator $\poi v1{\wedge \cdots \wedge }{d}{}{}{}$ of
$\ex d{{\goth p}_{\u}}$ and ${\goth I}_{+}$ the subset of elements $\iota $ of
${\goth I}$ such that $\{\iota \}$ is contained in $\{1,\ldots,d\}$.

Let ${\goth J}$ be the union of ${\goth I}_{+}$ and the set of strictly increasing
sequences in $$\{(i,j) \; \vert \; i=1,\ldots,\n, \, j=1,\ldots,n_{i}\} .$$
The order of ${\Bbb N}^{2}$ induces an order of ${\goth J}\setminus {\goth I}_{+}$ and
the set ${\goth J}$ is ordered so that ${\goth I}_{+}$ and
${\goth J}\setminus {\goth I}_{+}$ are ordered subsets and the elements of
${\goth I}_{+}\setminus \{0\}$ are bigger than the elements of
${\goth J}\setminus {\goth I}_{+}$. For $i$ in $\{1,\ldots,d\}$,
$\lambda _{i}$ is the element $1\tens v_{i}$ of $\tk {\k}{\k[\lp]}{\goth p}_{\u}$ and for
$\upsilon = \poi i{1}{<\cdots <}{k}{}{}{}$ in ${\goth J}$ and $j=1,\ldots,n$, set:
$$ \{\upsilon \} := \{\poi i1{,\ldots ,}{k}{}{}{}\}, \quad 
\vert \upsilon  \vert := k, \quad
\lambda _{\upsilon } := \poi {\lambda }{i_{1}}{\wedge \cdots \wedge }{i_{k}}{}{}{}, \quad
{\goth J}_{j} := \{\upsilon \in {\goth J} \; \vert \; \vert \upsilon \vert = j\}.$$
For $\upsilon =0$, $\vert \upsilon  \vert := 0$ and $\lambda _{\upsilon }=1$. Denote
by $\bjj$ the $\k[\lp]$-submodule of $\tk {\k}{\k[\lp]}{\goth g}$ generated by
$\lambda _{i}, \, i \in {\goth J}_{1}$.

\subsection{Some open subset of $\lp$} \label{de2}
We recall that $(h,e)$ is an element of $\lp\cap \Omega _{{\goth g}}$.

\begin{lemma}\label{lde2}
{\rm (i)} For some dense open subset $O$ of $\lp$, containing $(h,e)$ and contained in
$\Omega _{{\goth g}}$, ${\goth u}_{-}$ is a complement to $V_{x,y}$ in ${\goth g}$ for
all $(x,y)$ in $O$. 

{\rm (ii)} For $\varphi $ in $\tk {\k}{\k[O]}\ex {}{{\goth g}}$, there exists a unique
element $\varphi '$ of $\tk {\k}{\k[O]}\ex {}{{\goth u}_{-}}$ such that
$\varphi (x,y)-\varphi '(x,y)$ is in $V_{x,y}\wedge \ex {}{{\goth g}}$ for all
$(x,y)$ in $O$.
\end{lemma}

\begin{proof}
(i) By Corollary~\ref{csc3}, $V_{h,e}={\goth b}$. Then, for some dense open subset $O$
of $\lp$, contained in $\Omega _{{\goth g}}$, ${\goth u}_{-}$ is a complement to
$V_{x,y}$ in ${\goth g}$ for all $(x,y)$ in $O$ since the map
$$ \xymatrix{ \Omega _{{\goth g}} \ar[rr] && \ec {Gr}g{}{}{\b g{}}}, \qquad
(x,y) \longmapsto V_{x,y}$$
is regular.

(ii) Let $i$ be a positive integer and $\varphi $ in $\tk {\k}{\k[O]}\ex i{{\goth g}}$.
By (i), for $(x,y)$ in $O$, $\varphi (x,y)$ has a unique decomposition
$$ \varphi (x,y) = \varphi '(x,y) + \varphi ''(x,y) \quad  \text{with} \quad
\varphi '(x,y) \in \ex i{{\goth u}_{-}} \quad  \text{and} \quad
\varphi ''(x,y) \in \bigoplus _{j=1}^{i} \ex j{V_{x,y}}\wedge \ex {i-j}{{\goth u}_{-}}.$$
As $\varepsilon _{i}^{(m)}(x,y), \, (i,m) \in I_{0}$ is a basis of $V_{x,y}$ for all
$(x,y)$ in $\Omega _{{\goth g}}$, the map $\varphi '$ and $\varphi ''$ so defined are
regular by Claim~\ref{clsc1}, whence the assertion.
\end{proof}

\begin{coro}\label{cde2}
Let $(i,m)$ be in $I_{*,0}$. For some regular maps
$\omega _{i,m,k,-}, \, k=0,\ldots,m$ and $\omega _{i,m,k,+}, \, k=0,\ldots,m$ in
$\tk {\k}{\k[O\times {\goth p}_{-,\u}]}{\goth g}$,
$$ \varepsilon _{i}^{(m)}(x,y+ty') = \sum_{k=0}^{m} t^{k}\omega _{i,m,k,-}(x,y,y')
+ \sum_{k=0}^{m} t^{k}\omega _{i,m,k,+}(x,y,y') \quad  \text{with}  $$ $$
\omega _{i,m,k,-}(x,y,y') \in {\goth u}_{-} \quad  \text{and} \quad
\omega _{i,m,k,+}(x,y,y') \in V_{x,y}$$
for all $(t,x,y,y')$ in $\k\times O\times {\goth p}_{-,\u}$.
\end{coro}

\begin{proof}
As the map $y'\longmapsto \varepsilon _{i}^{(m)}(x,y+y')$ is a polynomial map
of degree $m$ with coefficients in $\tk {\k}{\k[O]}{\goth g}$, the corollary results from
Lemma~\ref{lde2}(ii).
\end{proof}

\subsection{Some expansions} \label{de3}
Let $I_{0,*}$ be the set:
$$ I_{0,*} := \{(i,m,k) \; \vert \; (i,m) \in I_{*,0}, \, k\in \{1,\ldots,m\}\} .$$ 
Set:
$$  \varepsilon _{0,0} := \poi {\varepsilon }1{\wedge \cdots \wedge }{\rg}{}{}{}, \quad
\varepsilon _{*} := \wedge _{(i,m)\in I_{*,0}} \varepsilon _{i}^{(m)} .$$
In these equalities, the order of the products is induced by the order of $I_{0}$. Then 
$\varepsilon _{0,0}\wedge \varepsilon _{*}$ is a generator of $\ex {\b g{}}{\bii}$. 

By Corollary~\ref{cde2}, for $(x,y,y')$ in $O\times {\goth p}_{-,\u}$ and 
$(i,m)$ in $I_{*,0}$, the polynomial map $t \mapsto \varepsilon _{i}^{(m)}(x,y+ty')$ 
has an expansion
$$ \varepsilon _{i}^{(m)}(x,y+ty') = \varepsilon _{i}^{(m)}(x,y) +
\sum_{k=1}^{m} t^{k} \omega _{i,m,k,-}(x,y,y') + 
\sum_{k=1}^{m} t^{k} \omega _{i,m,k,+}(x,y,y') $$
with 
$$ \omega _{i,m,k,-} \in \tk {\k}{\k[O]}\tk {\k}{\sy k{{\goth p}_{\u}}}
{\goth u}_{-}, \quad 
\omega _{i,m,k,+} \in \tk {\k}{\k[O]}\tk {\k}{\sy k{{\goth p}_{\u}}}{\goth g} $$
such that $\omega _{i,m,k,+}(x,y,y')\in V_{x,y}$. For $I$ subset of $I_{*,0}$ and $I'$
subset of $I_{0,*}$, set:
$$ \varepsilon _{I} := \wedge _{(i,m) \in I} \varepsilon _{i}^{(m)}, \quad
S_{I'} := \sum_{(i,m,k)\in I'} k, $$ $$ 
\omega _{I',-} := \wedge _{(i,m,k) \in I'} \omega _{i,m,k,-}, \quad
\omega _{I',+} := \wedge _{(i,m,k) \in I'} \omega _{i,m,k,+}.$$
In these equalities, the order of the products are induced by the orders of $I_{0}$ and 
$I_{0,*}$. In this subsection, the restriction of $\varepsilon _{I}$ to $O$ is denoted
by $\widetilde{\varepsilon _{I}}$ so that $\widetilde{\varepsilon _{I}}(x,y)$ is an
element of $\ex {\vert I \vert}{V_{x,y}}$ for all $(x,y)$ in $O$. For $I$ subset of
$I_{0,*}$, denote by $I^{\#}$ the image of $I$ by the projection $(i,m,k)\mapsto (i,m)$.
When $S_{I}=\vert I \vert$, we identify $I$ and $I^{\#}$ since $I$ is contained in
$I_{*,0}\times \{1\}$.  

For $\upsilon $ in ${\goth J}$ and $j=0,\ldots,n$, set: 
$$\varepsilon _{\upsilon } := \lambda _{\upsilon }\wedge
\varepsilon _{0,0}\wedge \varepsilon _{*}, \quad
{\cal J}_{j} := \{ I \subset I_{*,0} \; \vert \; \vert I \vert = n-j\}, \quad
\Lambda _{j} := {\Bbb N}^{d}_{j}\times {\goth I}_{j}.$$

\begin{lemma}\label{lde3}
Let $j=0,\ldots,n$, $\upsilon $ in ${\goth J}_{j}$ and $(x,y,y')$ in 
$O\times {\goth p}_{-,\u}$. Denote by $c_{\upsilon }(x,y,y')$ the coefficient of $t^{j}$
of the polynomial map $t\mapsto \varepsilon _{\upsilon }(x,y+ty')$.

{\rm (i)} The polynomial map $t\mapsto \varepsilon _{\upsilon }(x,y+ty')$ is
divisible by $t^{j}$ in $\tk {\k}{\k[t]}\ex {\b g{}+j}{\goth g}$.

{\rm (ii)} For a well defined map
$$ \xymatrix{ {\cal J}_{j} \ar[rr]^{\epsilon } && \{-1,1\}} $$
$$ c_{\upsilon }(x,y,y') := \sum_{I\in {\cal J}_{j}}
\epsilon (I) \hskip 0.15em \lambda _{\upsilon }(x,y)\wedge
\varepsilon _{I}(x,y) \wedge \omega _{I_{*,0}\setminus I,-}(x,y,y') .$$

{\rm (iii)} For $I$ in ${\cal J}_{j}$, for well defined functions 
$a_{r,\kappa ,I,\upsilon }, \; (r,\kappa ) \in \Lambda _{j}$ in $\k[\lp]$, 
$$ \epsilon (I) \hskip 0.15em \lambda _{\upsilon }(x,y) \wedge \varepsilon _{0,0}(x)
\wedge \varepsilon _{I}(x,y) \wedge \omega _{I_{*,0}\setminus I,-}(x,y,y') = 
\sum_{(r,\kappa ) \in \Lambda _{j}} v^{r}(y')a_{r,\kappa ,I,\upsilon }(x,y) 
w_{\kappa }\wedge \overline{\varepsilon _{0}}(x,y)\wedge \mu _{+}.$$
\end{lemma}

\begin{proof}
The cofficient of $t^{k}$ of the polynomial function
$t\longmapsto \varepsilon _{\upsilon }(x,y+ty')$  is the sum of the value at
$(x,y,y')$ of products 
$$ \epsilon (I,I_{-},I_{+}) \hskip 0.15em
\lambda _{\upsilon }\wedge \varepsilon _{0,0}
\wedge \widetilde{\varepsilon _{I}}\wedge \omega _{I_{-},-}\wedge \omega _{I_{+},+}$$
with 
$$I\subset I_{*,0}, \quad I_{-}\subset I_{0,*}, \quad I_{+}\subset I_{0,*},
\quad \epsilon (I,I_{-},I_{+}) \in \{-1,1\} $$
such that
$$ I \cup I_{-}^{\#} \cup I_{+}^{\#} = I \sqcup I_{-}^{\#} \sqcup I_{+}^{\#} , \quad 
\vert I \vert + \vert I_{-}^{\#} \vert + \vert I_{+}^{\#}\vert  = n, \quad
S_{I_{-}} + S_{I_{+}}   = k.$$   
According to Corollary~\ref{c2sc3} and Proposition~\ref{prcm}(ii), 
$$ \lambda _{\upsilon }\wedge \varepsilon _{0,0}\wedge \widetilde{\varepsilon _{I}} 
\in \bigoplus _{m=0}^{d} \ex {\vert I \vert+j+\rg -m}{\oi _{0}}
\wedge \ex {m}{{\goth p}_{\u}} .$$
Then 
$$ \vert I \vert + \vert I_{+}^{\#} \vert > n-j \Longrightarrow 
\lambda _{\upsilon }\wedge \varepsilon _{0,0} \wedge \widetilde{\varepsilon _{I}}
\wedge \omega _{I_{-},-} \wedge \omega _{I_{+},+} = 0 $$
since the $\k[\lp]$-module generated by $\oi_{0}$ and ${\goth p}_{\u}$ has rank $\b g{}$
by Proposition~\ref{prcm}.
As a result, if the coefficient of $t^{k}$ is different from $0$ then
$$ \vert I \vert + \vert I_{+}^{\#} \vert \leq n-j , \quad 
j \leq  \vert I_{-}^{\#} \vert .$$ 
Moreover,
$$ \vert I_{-}^{\#} \vert +\vert I_{+}^{\#}  \vert  \leq  k  
\quad  \text{since} \quad
\vert I_{-}^{\#} \vert +\vert I_{+}^{\#} \vert  \leq S_{I_{-}} + S_{I_{+}} .$$
As a result, the coefficient of $t^{k}$ of the polynomial map 
$t\mapsto \varepsilon _{\upsilon }(x,y+ty')$ is equal to $0$ when $k<j$ and
for $k=j$, it is the sum of the value at $(x,y,y')$ of the  products
$$ \epsilon (I,I_{-},I_{+})\hskip 0.15em  
\lambda _{\upsilon }\wedge \varepsilon _{0,0} \wedge \widetilde{\varepsilon _{I}}
\wedge \omega _{I_{-},-}$$ 
with 
$$ \vert I \vert + \vert I_{+}^{\#} \vert  \leq n - j, \quad 
j \leq  \vert I_{-}^{\#} \vert  , \quad
\vert I_{-} \vert = \vert I_{-}^{\#} \vert = S_{I_{-}}, \quad
\vert I_{+} \vert = \vert I_{+}^{\#} \vert = S_{I_{+}}, $$ 
whence 
$$ \vert I_{+} \vert = 0, \quad \vert I \vert = n-j, \quad \vert I_{-} \vert  = j .$$
For $I$ in ${\cal J}_{j}$, set:
$$ \epsilon (I) := \epsilon (I,I'_{*,0}\setminus I,\emptyset).$$ 
Then
$$ c_{\upsilon }(x,y,y') = \sum_{I\in {\cal J}_{j}} \epsilon (I) \hskip 0.15em
\lambda _{\upsilon }(x,y)\wedge \varepsilon _{0,0}(x)\wedge \varepsilon _{I}(x,y) \wedge 
\omega _{I'_{*,0}\setminus I,-}(x,y,y') .$$
By Proposition~\ref{prcm}(iii), for $I$ in ${\cal J}_{j}$,
$\lambda _{\upsilon }(x,y)\wedge \varepsilon _{0,0}(x)\wedge \varepsilon _{I}(x,y)$ is in
$\overline{\varepsilon _{0}}(x,y)\wedge \ex d{{\goth p}_{\u}}$, then for well defined
functions $a_{r,\kappa ,I,\upsilon }, \, (r,\kappa ) \in \Lambda _{j}$ in $\k[O]$, 
$$ \epsilon (I) \hskip 0.15em \lambda _{\upsilon }(x,y)\wedge
\varepsilon _{0,0}(x)\wedge \varepsilon _{I}(x,y) \wedge
\omega _{I_{*,0}\setminus I,-}(x,y,y') = 
\sum_{(r,\kappa ) \in \Lambda _{j}} v^{r}(y')  a_{r,\kappa ,I,\upsilon }(x,y) 
w_{\kappa }\wedge \overline{\varepsilon _{0}}(x,y)\wedge \mu _{+} ,$$
whence the lemma.
\end{proof}

For $j=1,\ldots,n$, $(\upsilon , \kappa )$ in ${\goth J}_{j}\times {\goth I}_{j}$,
denote by $a_{\upsilon ,\kappa }$ the function on $O\times {\goth p}_{-,\u}$,
$$ (x,y,y') \longmapsto a_{\upsilon ,\kappa }(x,y,y') :=
\sum_{I\in {\cal J}_{j}} \sum_{r\in {\Bbb N}^{d}_{j}}
v^{r}(y')a_{r,\kappa ,I,\upsilon }(x,y) .$$

\begin{coro}\label{cde3}
Let $j=0,\ldots,n$, $\upsilon $ in ${\goth J}_{j}$ and $(x,y,y')$ in
$O\times {\goth p}_{-,\u}$. The polynomial map
$t\mapsto \varepsilon _{\upsilon }(x,y+ty')$ is divisible by $t^{j}$ in
$\tk {\k}{\k[t]}\ex {\b g{}+j}{\goth g}$ and the coefficient of $t^{j}$ is equal to
$$ \sum_{\kappa \in {\goth I}_{j}} a_{\upsilon ,\kappa }(x,y,y')
w_{\kappa }\wedge \overline{\varepsilon _{0}}(x,y)\wedge \mu _{+} .$$
\end{coro}

The corollary  results from Lemma~\ref{lde3}.

\section{Some spaces} \label{sp}
In this section we consider some spaces related to ${\goth p}$ and we use the notations
${\goth I}$, ${\goth J}$, ${\goth J}_{j}$, $\upsilon $, $\lambda _{\upsilon }$ of
Subsection~\ref{de1}.

\subsection{Interior derivation} \label{sp1}
We recall that $\bk g{}$ is the orthogonal complement to $\bi g{}$ in
$\tk {\k}{\sgg g{}}{\goth g}$. Let $\bkk$ be the image of $\bk g{}$ by the restriction
morphism from $\gg g{}$ to $\lg l*$ and $\ok$ the image of $\bkk$ by the restriction
morphism from $\lg l*$ to $\lp$. As usual, for $\nu $ in $\bkk$, $\iota _{\nu }$ denotes
the $\tk {\k}{\k[\lg l*]}\e Sg$-derivation of the algebra
$\tk {\k}{\k[\lg l*]}D({\goth g})$
such that $\iota _{\nu }(v) := \dv {\nu }v$ and $\iota _{\nu }(p) := 0$ for all $(v,p)$
in ${\goth g}\times \tk {\k}{\k[\lg l*]}\e Sg$. In the same way, for $\tilde{\nu }$ in
$\ok$, $\iota _{\tilde{\nu }}$ is a $\tk {\k}{\k[\lp]}\e Sg$-derivation of the algebra
$\tk {\k}{\k[\lp]}D({\goth g})$.

\begin{lemma}\label{lsp1}
Let $\upsilon $ be in ${\goth J}$ and $\varphi $ in
$\tk {\k}{\k[\lp]}\tk {\k}{\e Sg}\ex {}{{\goth u}_{-}}$.

{\rm (i)} For all $\nu $ in $\bkk$, $\iota _{\nu }(\lambda _{\upsilon })$ is
in $J\tk {\k}{\e Sg}\ex {\vert \upsilon  \vert-1}{\bjj}$.

{\rm (ii)} If the restriction of $\iota _{\nu }(\varphi )$ to $\lp$ is equal to $0$
for all $\nu $ in $\bkk$ then $\varphi $ is in $\tk {\k}{\k[\lp]}\e Sg$ or $\varphi =0$.
\end{lemma}

\begin{proof}
(i) Let $\nu $ be in $\bkk$ and $\tilde{\nu }$ its restriction to $\lp$. By
Proposition~\ref{prcm}, $\tilde{\nu }$ is orthogonal to $\lambda _{i}$ for all $i$ in
${\goth J}_{1}$. As a result, $\iota _{\nu }(\lambda _{i})$ is in $J$ and by induction on
$\vert \upsilon  \vert$, $\iota _{\upsilon }(\lambda _{\upsilon })$ is in
$J\tk {\k}{\e Sg}\ex {\vert \upsilon  \vert-1}{\bjj}$.

(ii) Suppose that for some $i=1,\ldots,n$, $\varphi $ is in
$\tk {\k}{\k[\lp]}\tk {\k}{\e Sg}\ex i{{\goth u}_{-}}$ and the restriction of
$\iota _{\nu }(\varphi )$ to $\lp$ is equal to $0$ for all $\nu $ in $\bkk$. Prove by
induction on $i$ that $\varphi =0$. Suppose $i=1$ and the restriction of
$\iota _{\nu }(\varphi )$ to $\lp$ equal to $0$ for all $\nu $ in $\bkk$. Let $O$ be
an open subset of $\lp$ as in Lemma~\ref{lde2}. By Proposition~\ref{psc2}, for all
$(x,y)$ in $O$, $V_{x,y}$ is the orthogonal complement to the image of $\bkk$ by the
evaluation map $\nu \mapsto \nu (x,y)$. As a result, $\varphi (x,y)=0$ for all $(x,y)$
in $O$ since $\varphi (x,y)$ is in $\tk {\k}{\e Sg}{\goth u}_{-}$ and ${\goth u}_{-}$ is
the complement to $V_{x,y}$ in ${\goth g}$, whence the assertion
for $i=1$ since $\varphi $ is in $\tk {\k}{\k[\lp]}\tk {\k}{\e Sg}{\goth u}_{-}$.

Suppose $i>1$, the assertion true for $i-1$ and $\varphi \neq 0$. A contradiction is
expected. For $j=1,\ldots,n$, denote by $F_{j}$ the subspace of ${\goth u}_{-}$
generated by $\poi wj{,\ldots,}{n}{}{}{}$. Let $\nu $ be in $\bkk$ and $j_{*}$ the
smallest integer such that $\varphi $ is in
$\tk {\k}{\k[\lp]}\tk {\k}{\e Sg}\ex i{F_{j_{*}}}$. Then
$$ \varphi  = \varphi '\wedge w_{j_{*}} + \varphi '' \quad  \text{with} $$ $$
\varphi '\in \tk {\k}{\k[\lp]}\tk {\k}{\e Sg}\ex {i-1}{F_{j_{*}+1}}
\quad  \text{and} \quad
\varphi '' \in \tk {\k}{\k[\lp]}\tk {\k}{\e Sg}\ex {i}{F_{j_{*}+1}}$$
and
$$ \iota _{\nu }(\varphi ) = \iota _{\nu }(\varphi ')\wedge w_{j_{*}} \pm
\dv {\nu }{w_{j_{*}}}\varphi ' + \iota _{\nu }(\varphi '').$$
By hypothesis, since the spaces $\tk {\k}{\k[\lp]}\tk {\k}{\e Sg}\ex {}{F_{j}}, \,
j=1,\ldots,n$ are invariant by $\iota _{\nu }$, the restriction of
$\iota _{\nu }(\varphi ')$ to $\lp$ is equal to $0$. So, by induction hypothesis,
$\varphi '=0$, whence the contradiction by minimality of $j_{*}$ and the assertion.
\end{proof}

\subsection{An equivalence} \label{sp2}
Set:
$$ {\Bbb N}_{{\goth p}} := {\Bbb N}^{d}\times {\goth J}, \quad
\vert (r,\upsilon ) \vert := \vert r \vert + \vert \upsilon  \vert  , $$ 
for $(r,\upsilon )$ in ${\Bbb N}_{{\goth p}}$. Let ${\cal L}$ be the subspace of
$(\tk {\k}{\k[\lg l*]}D({\goth g}))^{{\Bbb N}_{{\goth p}}}$:
$$ {\cal L} := \{\varphi \in (\tk {\k}{\k[\lg l*]}D({\goth g}))
^{{\Bbb N}_{{\goth p}}} \; \vert \; \exists N_{\varphi } \in {\Bbb N} \quad  
\text{such that} \quad \vert r \vert \geq N_{\varphi } \Longrightarrow 
\varphi _{r,\upsilon } = 0 \}$$
and $\chiup $ the map
$$ \xymatrix{ {\cal L} \ar[rr]^-{\chiup } &&
\tk {\k}{\k[\lg l*]}D({\goth g})}, \qquad
(\varphi _{r,\upsilon }, \, (r,\upsilon ) \in {\Bbb N}_{{\goth p}}) \longmapsto
\sum_{(r,\upsilon )\in {\Bbb N}_{{\goth p}}} v^{r}\varphi _{r,\upsilon }
\wedge \lambda _{\upsilon } .$$
For $k=0,\ldots,n$, denote by ${\cal L}^{(k)}$ the subspace of elements $\varphi $ of
${\cal L}$ such that 
$$ \chiup (\varphi ) \in \tk {\k}{\k[\lg l*]} D^{k}({\goth g}) .$$ 
For $l$ nonnegative integer, let ${\cal P}_{l}$ and ${\cal P}_{l,+}$ be the subspaces of 
${\cal L}$ defined by the following conditions:
$$\varphi \in {\cal P}_{l} \Longleftrightarrow (  
\vert (r,\upsilon ) \vert \neq l \Longrightarrow \varphi _{r,\upsilon } = 0 ), \quad
\varphi \in {\cal P}_{l,+} \Longleftrightarrow (  
\vert (r,\upsilon ) \vert < l \Longrightarrow \varphi _{r,\upsilon } = 0 ),$$
and set:
$${\cal P}_{l,k}:= {\cal P}_{l}\cap {\cal L}^{(k)}, \quad  
{\cal P}_{l,+,k}:= {\cal P}_{l,+}\cap {\cal L}^{(k)}$$
for $k=0,\ldots,n$. Denote by $M_{l}$ the subspace of elements $\varphi $ of 
$\tk {\k}{\k[\lg l*]}D({\goth g})$ such that $\varphi \wedge \varepsilon $ is in 
$\tk {\k}{J^{l}}D({\goth g})$. 

\begin{prop}\label{psp2}
Let $k=0,\ldots,n$, $l$ a nonnegative integer and $\varphi $ in ${\cal P}_{l,k}$.

{\rm (i)} The element $\chiup (\varphi )$ is in $M_{l}$.

{\rm (ii)} The element $\chiup (\varphi )$ is in $M_{l+1}$ if and only if the restriction
of $\varphi _{r,\upsilon }\wedge \overline{\varepsilon _{0}}\wedge \mu _{+}$ to $\lp$ is
equal to $0$ for all $(r,\upsilon )$.
\end{prop}

\begin{proof}
(i) Let $O$ be an open subset of $\lp$ as in Lemma~\ref{lde2}. For $\psi $ in
$\k[\lg l*]$, $\psi $ is in $J^{l}$ if and only if its restriction to
$O\times {\goth p}_{-,\u}$ is in $J^{l}\k[O\times {\goth p}_{-,\u}]$. Hence it is
sufficient to prove that the restriction of $\chiup (\varphi )\wedge \varepsilon $
to $O\times {\goth p}_{-,\u}$ is in
$\tk {\k}{J^{l}\k[O\times {\goth p}_{-,\u}]}D^{k+\b g{}}({\goth g})$.

Let $(x,y,y')$ be in $O\times {\goth p}_{\u}$. According to
Corollary~\ref{cde3}, for $j=1,\ldots,n$ and $\upsilon $ in ${\goth J}_{j}$, the
polynomial map
$$ t\longmapsto \lambda _{\upsilon }\wedge \varepsilon (x,y+ty')$$ 
is divisible by $t^{j}$ in $\tk {\k}{\k[t]}\ex {}{{\goth g}}$ and 
the coefficient $c_{\upsilon }(x,y,y')$ of $t^{j}$ satisfies the equality
$$ c_{\upsilon }(x,y,y') = \sum_{\kappa \in {\goth I}_{j}} 
a_{\upsilon ,\kappa }(x,y,y') w_{\kappa }\wedge \overline{\varepsilon _{0}(x,y)}
\wedge \mu _{+} .$$ 
Hence the polynomial map
$$ t \longmapsto \chiup (\varphi )(x,y+ty')\wedge \varepsilon (x,y+ty')$$
is divisible by $t^{l}$ in $\tk {\k}{\k[t]}D^{k+\b g{}}({\goth g})$ since $\varphi $ is
in ${\cal P}_{l,k}$. As a result, $\chiup (\varphi )$ is in $M_{l}$. 

(ii) Suppose that the restriction of
$\varphi _{r,\upsilon }\wedge \overline{\varepsilon _{0}}\wedge \mu _{+}$ to $\lp$ is
equal to $0$ for all $(r,\upsilon )$. Let $(x,y,y')$ be in $O\times {\goth p}_{\u}$.
By Proposition~\ref{prcm}(iii), the polynomial map
$$ t \longmapsto \chiup (\varphi )(x,y+ty')\wedge \varepsilon (x,y+ty')$$
is divisible by $t^{l+1}$ in $\tk {\k}{\k[t]}D^{k+\b g{}}({\goth g})$ since $\varphi $
is in ${\cal P}_{l,k}$. Hence the restriction of $\chiup (\varphi )\wedge \varepsilon $ to
$O\times {\goth p}_{_,\u}$ is in
$J^{l+1}\tk {\k}{\k[O\times {\goth p}_{-,\u}]}\ex {k+\b g{}}{{\goth g}}$. As a result,
the condition is sufficient.

Suppose that $\chiup (\varphi )\wedge \varepsilon $ is in $M_{l+1}$ and prove by induction
on $k$ that the restriction of
$\varphi _{r,\upsilon }\wedge \overline{\varepsilon _{0}}\wedge \mu _{+}$ to $\lp$ is
equal to $0$ for all $(r,\upsilon )$. For $k=0$, $\chiup (\varphi )$ is in
$\tk {\k}{J^{l}}\e Sg\cap M_{l+1}$ and for $r$ in ${\Bbb N}^{d}_{l}$,
$$ \varphi _{r,0} = \varphi '_{r,0} + \varphi ''_{r,0} \quad  \text{with} \quad
\varphi '_{r,0} \in \tk {\k}{\k[\lp]}\e Sg \quad  \text{and} \quad
\varphi ''_{r,0} \in \tk {\k}{J}\e Sg .$$
By Proposition~\ref{prcm}(iii), for some $p$ in $\k[\lp]\setminus \{0\}$, the restriction
of $\varepsilon $ to $\lp$ is equal to $p\overline{\varepsilon _{0}}\wedge \mu _{+}$.
Then, for all $(x,y,y')$ in $\lg l*$,
$$\sum_{r\in {\Bbb N}^{d}_{l}} \dv v{y'}^{r}\varphi '_{r,0}(x,y) p(x,y)
\overline{\varepsilon _{0}}(x,y)\wedge \mu _{+} = 0.$$
As a result, the restriction of
$\varphi _{r,0}\wedge \overline{\varepsilon _{0}}\wedge \mu _{+}$ to $\lp$ is
equal to $0$ for all $r$ in ${\Bbb N}^{d}_{l}$.

Suppose $k>0$ and the assertion true for $k-1$. By Lemma~\ref{lsp1}(i), for all $\nu $ in
$\bkk$ and $\psi $ in ${\cal P}_{l,k}$, $\iota _{\nu }(\chiup (\psi ))-\chiup (\psi ')$
is in $\chiup ({\cal P}_{l+1,+,k-1})$ for some $\psi '$ in ${\cal P}_{l,k-1}$. Suppose
that for some $(r,\upsilon )$ the restriction of
$\varphi _{r,\upsilon }\wedge \overline{\varepsilon _{0}}\wedge \mu _{+}$
to $\lp$ is different from $0$ and denote by $j$ the biggest integer such that
$\vert \upsilon  \vert = j$ for such $(r,\upsilon )$. A contradiction is expected. As
the condition is sufficient, we can suppose that $\varphi _{r,\upsilon }=0$ for all
$(r,\upsilon )$ such that $\vert \upsilon  \vert > j$. Then, for $\nu $ in $\bkk$,
$\iota _{\nu }(\chiup (\varphi ))-\chiup (\psi )$ is in $\chiup ({\cal P}_{l+1,+,k-1})$
for some $\psi $ in ${\cal P}_{l,k-1}$ such that
$$ \vert \upsilon  \vert > j \Longrightarrow \psi _{r,\upsilon } = 0 \quad  \text{and}
\quad
\vert \upsilon  \vert = j \Longrightarrow \psi _{r,\upsilon } =
\iota _{\nu }(\varphi _{r,\upsilon }).$$
As $\iota _{\nu }(\varepsilon )=0$, $\iota _{\nu }(\chiup (\varphi ))$ is in $M_{l+1}$.
So, by (i), $\chiup (\psi )$ is in $M_{l+1}$. As a result, the restriction of
$\iota _{\nu }(\varphi _{r,\upsilon })\wedge \overline{\varepsilon _{0}}\wedge
\mu _{+}$ to $\lp$ is equal to $0$ for all $\nu $ in $\bkk$ and for all $(r,\upsilon )$
such that $\vert \upsilon  \vert = j$ by induction hypothesis.
 
Let $(r,\upsilon )$ be such that $\vert \upsilon  \vert=j$ and the restriction
of $\varphi _{r,\upsilon }\wedge \overline{\varepsilon _{0}}\wedge \mu _{+}$ to $\lp$ is
different from $0$. The restriction of $\varphi _{r,\upsilon }$ to
$O\times {\goth p}_{\u}$ has a unique decomposition
$$ \varphi _{r,\upsilon } \left \vert \right. _{O\times {\goth p}_{-,\u}} =
\varphi '_{r,\upsilon } + \varphi ''_{r,\upsilon }+\varphi '''_{r,\upsilon }$$
with $\varphi '_{r,\upsilon }$ in $\tk {\k}{\k[O]}\tk {\k}{\e Sg}
\ex {k-\vert \upsilon  \vert}{{\goth u}_{-}}$,
$\varphi ''_{r,\upsilon }$ in $J\tk {\k}{\k[O\times {\goth p}_{-,\u}]}
D^{k-\vert \upsilon  \vert}({\goth g})$ and $\varphi '''_{r,\upsilon }$ in
$\tk {\k}{\k[O]}D^{k-\vert \upsilon  \vert}({\goth g})$ such that
$$\varphi '''_{r,\upsilon }(x,y) \in 
D^{k-\vert \upsilon  \vert-1}({\goth g})\wedge V_{x,y}   \qquad \forall (x,y) \in O.$$
In particular, the restriction to $O$ of $\iota _{\nu }(\varphi ''_{r,\upsilon })$ is
equal to $0$ and for all $(x,y)$ in $O$,
$$ \iota _{\nu }(\varphi '''_{r,\upsilon })(x,y) \in
D^{k-\vert \upsilon  \vert-2}({\goth g})\wedge V_{x,y}.$$
As a result, the restriction of $\iota _{\nu }(\varphi '_{r,\upsilon })$ to $O$ is equal
to $0$ since it is in
$\tk {\k}{\k[O]}\tk {\k}{\e Sg}\ex {k-\vert \upsilon  \vert}{{\goth u}_{-}}$.
Then, by Lemma~\ref{lsp1}(ii), for all $(x,y)$ in $O$,
$$\varphi '_{r,\upsilon }(x,y)=0 \quad  \text{and} \quad
\varphi _{r,\upsilon }(x,y)\wedge \overline{\varepsilon _{0}}(x,y)\wedge \mu _{+} = 0,$$
whence the contradiction.
\end{proof}

\section{Filtration of complexes} \label{fc}
In this section, we suppose that the simple factors of ${\goth d}$ have Property
$({\bf P})$ and we consider the complexes defined in Subsections ~\ref{co3} and
~\ref{co4}. We recall that $\poi v1{,\ldots,}{d}{}{}{}$, $\poi w1{,\ldots,}{d}{}{}{}$,
$\poi w1{,\ldots,}{n}{}{}{}$ are basis of ${\goth p}_{\u}$, ${\goth p}_{-,\u}$,
${\goth u}_{-}$ respectively. For $\iota = \poi i1{<\cdots <}{j}{}{}{}$ in ${\goth I}$
and $r=(\poi r1{,\ldots,}{d}{}{}{})$ in ${\Bbb N}^{d}$,
$$ w_{\iota } := \poi w{{i_{1}}}{\wedge \cdots \wedge }{{i_{j}}}{}{}{}, \quad
v^{r} := \poie v1{\cdots }{d}{}{}{}{{r_{1}}}{{r_{d}}} \quad  \text{and}
\quad w^{r} := \poie w1{\cdots }{d}{}{}{}{{r_{1}}}{{r_{d}}} .$$

According to the notations of Subsection~\ref{de1}, for $i=1,\ldots,\n$, $\bjj_{i}'$ is
the $\k[\lp]$-submodule of $\tk {\k}{\k[\lp]}{\goth l}$ generated by
$\poi {\lambda }{i,1}{,\ldots,}{i,n_{i}}{}{}{}{}$ and $\bjj$ is the
$\k[\lp]$-submodule of $\tk {\k}{\k[\lp]}{\goth g}$ generated by ${\goth p}_{\u}$ and 
$\poie {\bjj}{1}{,\ldots,}{\n}{}{}{}{\prime}{\prime}$. In particular, when
${\goth p}={\goth b}$, $\bjj=\tk {\k}{\k[\lp]}{\goth u}$. We recall that for
$\upsilon = \poi i1{<\cdots <}{k}{}{}{}$ in ${\goth J}$,
$$ \lambda _{\upsilon } = \poi {\lambda }{i_{1}}{\wedge \cdots \wedge }{i_{k}}{}{}{} .$$

Let $\poi z1{,\ldots,}{2n+\rg}{}{}{}$ be a system of coordinates of
the local ring $\an {}{}$ of $G$ at the identity and for
$s=(\poi s1{,\ldots,}{n+2\rg}{}{}{})$ in ${\Bbb N}^{n+2\rg}$,
$$ z^{s} := \poie z1{\cdots }{{n+2\rg}}{}{}{}{{s_{1}}}{{s_{n+2\rg}}} .$$
Let $\hat{J}$ be the ideal of $\tk {\k}{\han {}{}}\k[\lg l*]$ generated by ${\goth m}$
and $J$.

\subsection{Some complexes} \label{fc1}
For $I$ subset of ${\Bbb I}_{k}$ and $(j_{1},j_{2})$ in ${\Bbb N}^{2}$ such that
$j_{1}+j_{2}\leq k$, let $I_{j_{1},j_{2}}$ be the subset of elements $\iota $ of
${\Bbb I}''_{k-j_{1}-j_{2}}$ such that $(j_{1},j_{2},\iota )$ is in $I$. According to
the notations of Subsection~\ref{co3}, for $k=0,\ldots,n$, $I$ subset of
${\Bbb I}_{k}$ and $V\in\{{\goth l},{\goth p}_{-},{\goth g}\}$, denote by
$D_{k,I,{\goth p}}(V)$ the total graded subcomplex of $D_{k}^{\bullet}(V)$ deduced from
the triple complex
$$ \bigoplus _{i=0}^{k} \bigoplus _{(j_{1},j_{2})\in {\Bbb N}^{2}_{k-i}}
\tk {\k}{D_{j_{1}}^{\bullet}(V')}\tk {\k} {D_{j_{2}}^{\bullet}({\goth z})}
D_{i,I_{j_{1},j_{2}},\#}({\goth d}) ,$$
with $V'=\{0\}$ when $V={\goth l}$, $V'={\goth p}_{-,\u}$ when $V={\goth p}_{-}$ and
$V'={\goth p}_{\pm,\u}$ when $V={\goth g}$. For $V={\goth p}_{-}$ or $V={\goth g}$
and $j=(\poi j1{,\ldots,}{\n}{}{}{})$
in ${\Bbb I}''$, set:
$$ {\Bbb I}_{k}^{(j)} := \{(\poi i{-1}{,\ldots,}{\n}{}{}{}) \in
{\Bbb I}_{k-\vert j \vert} \; \vert \;
i_{1} \leq n_{1}-j_{1} ,\ldots, i_{\n} \leq n_{\n}-j_{\n}\} ,$$
$$ D_{k,{\goth p},j}^{\bullet}(V) :=
D_{k-\vert j \vert,{\Bbb I}_{k}^{(j)},{\goth p}}^{\bullet}(V)[-\vert j \vert]
\wedge \ex {j_{1}}{\bjj'_{1}}\wedge \cdots \wedge \ex {j_{\n}}{\bjj'_{\n}},$$
$$ D_{k,{\goth p},j,*}^{\bullet}(V) := \tk {\k}
{\ex {j_{1}}{\bjj'_{1}}\wedge \cdots \wedge \ex {j_{\n}}{\bjj'_{\n}}}
D_{k-\vert j \vert,{\Bbb I}_{k}^{(j)},{\goth p}}^{\bullet}(V)[-\vert j \vert] .$$
For $\upsilon $ in ${\goth J}$,  denote by $j_{\upsilon }$ the element of
${\Bbb I}''$ such that
$$ {j_{\upsilon }}_{i} := \vert \{m \in \{\upsilon \} \, \vert \,
\lambda _{m} \in \bjj '_{i}\} \vert $$
for $i=1,\ldots,\n$ and set
$$k_{\upsilon } := \vert \upsilon  \vert - \vert j_{\upsilon } \vert$$
so that
$$D_{k-k_{\upsilon },{\goth p},j_{\upsilon }}^{\bullet}(V)
[-\vert \lambda _{\upsilon }\vert]\wedge \lambda _{\upsilon } \subset 
\tk {\k}{\k[\lp]}D_{k,{\goth p}}^{\bullet}(V).$$

For $k=0,\ldots,n$ and $j$ non negative integer, denote by
$D_{k,{\goth p},j,*}^{\bullet}({\goth g})$ the graded submodule of
$\tk {\k}{\k[\lp]}\tk {\k}{\ex j{\bjj}}D_{k-j,{\goth p}}^{\bullet}({\goth g})$,
$$ D_{k,{\goth p},j,*}^{\bullet}({\goth g}) := \bigoplus _{i=0}^{j}
\bigoplus _{(\poi j1{,\ldots,}{\n}{}{}{})\in {\Bbb N}^{\n}_{j-i}} 
\ex i{{\goth p}_{\u}}\wedge  
D_{k-i,{\goth p},(\poi j1{,\ldots,}{\n}{}{}{}),*}^{\bullet}({\goth g})[-i] .$$ 

\subsection{Some spaces} \label{fc2}
Let $k=0,\ldots,n$. Set:
$$ {\Bbb N}_{{\goth p},+} := {\Bbb N}^{2n+\rg}\times {\Bbb N}^{d}\times {\goth J}, \quad
\vert (s,r,\upsilon ) \vert :=
\vert s \vert + \vert r \vert + \vert \upsilon  \vert  , $$ 
for $(s,r,\upsilon )$ in ${\Bbb N}_{{\goth p},+}$. For $k=0,\ldots,n$, denote by
${\cal M}_{k}$ the subspace of
$(\tk {\k}{\k[\lg l*]}D({\goth g}))^{{\Bbb N}_{{\goth p},+}}$:
$$ {\cal M}_{k} := \{\varphi \in (\tk {\k}{\k[\lg l*]}D({\goth g}))
^{{\Bbb N}_{{\goth p},+}} \; \vert \; \exists N_{\varphi } \in {\Bbb N} \quad  
\text{such that} \quad \vert r \vert \geq N_{\varphi } \Longrightarrow 
\varphi _{s,r,\upsilon } = 0 \quad  \text{and} \quad $$
$$\varphi _{s,r,\upsilon } \in
\tk {\k[\lp]}{\k[\lg l*]}
D_{k-k_{\upsilon },{\goth p},j_{\upsilon }}({\goth g}).\}$$  
For $\varphi $ in ${\cal M}_{k}$, $\varphi $ is identified with the sequence
$\varphi _{s,r,\upsilon }, \, (s,r,\upsilon ) \in S$ for any subset $S$ of
${\Bbb N}_{{\goth p},+}$ containing the support of $\varphi $. For $l$ nonnegative
integer, let ${\cal M}_{k,l}$ be the subspace of elements
$\varphi _{s,r,\upsilon }, \, (s,r,\upsilon )\in {\Bbb N}_{{\goth p},+}$
such that
$$ \vert (s,r,\upsilon )\vert \neq l \Longrightarrow
\varphi _{s,r,\upsilon } = 0$$
and denothe by ${\cal M}_{k,l,+}$ the subspace of elements
$\varphi _{s,r,\upsilon }, \, (s,r,\upsilon )\in {\Bbb N}_{{\goth p},+}$ such that
$$ \vert (s,r,\upsilon )\vert < l \Longrightarrow
\varphi _{s,r,\upsilon } = 0.$$
Let ${\cal M}$ be the sum of the spaces ${\cal M}_{k}, \, k=0,\ldots,n$ and $\kappa $
the map
$$ \xymatrix{ {\cal M} \ar[rr]^-{\kappa } &&
D({\goth g})\wedge \ex {\b g{}}{\widehat{\bii}}}, \qquad
(\varphi _{s,r,\upsilon }, \, (s,r,\upsilon ) \in {\Bbb N}_{{\goth p},+})
\longmapsto \sum_{(s,r,\upsilon )\in {\Bbb N}_{{\goth p},+}}
z^{s}v^{r}\varphi _{s,r,\upsilon }\wedge \lambda _{\upsilon }\wedge \hat{\varepsilon }$$
with $\hat{\varepsilon }$ as in Subsection~\ref{stp2}. 

\begin{lemma}\label{lfc2}
Let $l$ be a positive integer and $k=1,\ldots,n$. For $i=1,\ldots,\n$, let $j_{i}$ be the
element of ${\Bbb N}^{\n}_{1}$ whose $i$-th coordinate is equal to $1$.

{\rm (i)} For $m=1,\ldots,n-d$ and $\varphi $ in $\tk {\k}{\k[\lp]}D_{m,\#}({\goth l})$ 
such that $\varphi (x,y)\wedge \overline{\varepsilon _{0}}(x,y)=0$ for all $(x,y)$ in 
$\lp$, 
$$ \varphi \in D_{m-1,\#}({\goth l})\wedge \oi _{0,0} +
\sum_{i=1}^{\n} D_{m-1,{\goth p},j_{i}}({\goth l})\wedge \bjj_{i}'.$$

{\rm (ii)} For $\varphi $ in $\tk {\k}{\k[\lp]}D_{k,{\goth p}}({\goth p}_{-})$, 
if the restriction of $\varphi \wedge \varepsilon $ to $\lp$ is equal to $0$, then 
$$ \varphi \in D_{k-1,{\goth p}}({\goth p}_{-})\wedge \oi _{0,0} +
\sum_{i=1}^{\n} D_{k-1,{\goth p},j_{i}}({\goth p}_{-})\wedge \bjj_{i}'.$$

{\rm (iii)} For $\varphi $ in $\tk {\k}{\k[\lg l*]}D_{k,{\goth p}}({\goth g})$, 
the restriction of $\varphi \wedge \varepsilon $ to $\lp$ is equal to $0$ if and 
only if 
$$\varphi \in \tk {\k[\lp]}{\k[\lg l*]}D_{k-1,{\goth p}}({\goth g})\wedge \oi_{0,0} +
\sum_{i=1}^{\n} \tk {\k[\lp]}{\k[\lg l*]}D_{k-1,{\goth p},j_{i}}({\goth g})
\wedge \bjj_{i}'+ $$ $$JD_{k,{\goth p}}({\goth g}) + 
\tk {\k}{\k[\lg l*]}D_{k-1,{\goth p}}({\goth g})\wedge {\goth p}_{\u}.$$
\end{lemma}

\begin{proof}
(i) As the simple factors of ${\goth d}$ have Property $({\bf P})$, the assertion 
results from Proposition~\ref{p2pp} since the $\k[\lp]$-module 
$\tk {\k[\ll]}{\k[\lp]}\oi _{{\goth l}}$ is generated by $\oi _{0,0}$ and
$\poie {\bjj}1{,\ldots,}{\n}{}{}{}{\prime}{\prime}$, whence the assertion.

(ii) Let $\varphi $ be in $\tk {\k}{\k[\lp]}D_{k,{\goth p}}({\goth p}_{-})$ such that 
$\varphi (x,y)\wedge \varepsilon (x,y)=0$ for all $(x,y)$ in $\lp$. The element 
$\varphi $ has an expansion
$$ \varphi = \sum_{\kappa \in {\goth I}_{+}} \varphi _{\kappa }\wedge w_{\kappa }
\quad  \text{with} \quad \varphi _{\kappa } \in \tk {\k}{\k[\lp]}
\bigoplus _{m=0}^{\inf\{n-d,k\}} \tk {\k}{\sy {k-m}{{\goth p}_{-,\u}}}
D_{m,\#}({\goth l}) .$$ 
For $(x,y)$ in $\lp$, 
$$\varepsilon (x,y) \in \overline{\varepsilon _{0}}(x,y)\wedge \ex d{{\goth p}_{\u}}$$
by Proposition~\ref{prcm}(i), whence
$$ \varphi _{\kappa }(x,y)\wedge \overline{\varepsilon _{0}}(x,y) = 0$$
for all $\kappa $ in ${\goth I}_{+}$ and all $(x,y)$ in $\lp$. So, by (i), 
$$ \varphi _{\kappa } \in \bigoplus _{m=0}^{\inf\{n-d,k\}} 
\tk {\k}{\sy {k-m}{{\goth p}_{-,\u}}}D_{m-1,\#}({\goth l})\wedge \oi _{0,0} +
\sum_{i=1}^{\n} D_{m-1,{\goth p},j_{i}}({\goth l})\wedge \bjj_{i}',$$ 
for all $\kappa $, whence the assertion.

(iii) The condition is sufficient since
$$\varepsilon (x,y) \in \overline{\varepsilon _{0}}(x,y)\wedge \ex d{{\goth p}_{\u}}$$
for all $(x,y)$ in $\lp$. Suppose that the restriction of $\varphi \wedge \varepsilon $
to $\lp$ is equal to $0$. Since $\k[\lg l*]=\k[\lp] + J$, we can suppose that $\varphi $
is in $\tk {\k}{\k[\lp]}D_{k,{\goth p}}({\goth g})$. As ${\goth g}$ is the direct sum of 
${\goth p}_{-}$ and ${\goth p}_{\u}$,
$$ \ex {}{{\goth g}} = \ex {}{{\goth p}_{-}} \oplus \ex {}{{\goth g}}\wedge {\goth p}_{\u}
\quad  \text{and} \quad D_{k,{\goth p}}({\goth g}) = 
\bigoplus _{m=0}^{k} \tk {\k}{\sy m{{\goth p}_{\u}}}D_{k-m,{\goth p}}({\goth p}_{-})
\oplus D_{k-1,{\goth p}}({\goth g})\wedge {\goth p}_{\u} .$$ 
Then $\varphi = \varphi _{1} + \varphi _{2}$ with
$$ \varphi _{1} \in \tk {\k}{\k[\lp]}
\bigoplus _{m=0}^{k} \tk {\k}{\sy m{{\goth p}_{\u}}}D_{k-m,{\goth p}}({\goth p}_{-}) 
\quad  \text{and} \quad
\varphi _{2} \in \tk {\k}{\k[\lp]}D_{k-1,{\goth p}}({\goth g})\wedge {\goth p}_{\u} .$$ 
As ${\goth p}_{\u}\wedge \varepsilon (x,y) = \{0\}$ for all $(x,y)$ in $\lp$, 
the restriction of $\varphi _{2}\wedge \varepsilon $ to $\lp$ is equal to $0$. Then
so is the restriction of $\varphi _{1}\wedge \varepsilon $ to $\lp$, whence the
assertion by (ii).
\end{proof}

For $s$ in ${\Bbb N}^{2n+\rg}$, we denote by ${\Bbb N}^{2n+\rg}_{s}$ the subset of
elements $s'$ of ${\Bbb N}^{2n+\rg}$ such that $s$ is smaller than $s'$. 

\begin{coro}\label{cfc2}
Let $k=1,\ldots,n$, $(s,r,\upsilon )$ in ${\Bbb N}_{{\goth p},+}$ and $\varphi $ in
$\tk {\k}{\k[\lg l*]}D_{k,{\goth p}}({\goth g})$
such that the restriction of $\varphi \wedge \varepsilon $ to $\lp$ is equal to $0$.
Set $l:=\vert (s,r,\upsilon ) \vert$ and
$\psi := z^{s}v^{r}\varphi \wedge \lambda _{\upsilon }\wedge \hat{\varepsilon }$. 

{\rm (i)} The space $\kappa ({\cal M}_{k,l+1,+})$ contains $\psi $.

{\rm (ii)} Suppose that $\varphi $ is in $\tk {\k}{\k[\lp]}D_{k,{\goth p}}({\goth g})$.
For some 
$$\psi _{1} := (\varphi _{s,r,\nu }, \, (s,r,\nu ) \in \{s,r\}
\times {\goth J}_{\vert \upsilon \vert+1}) \in {\cal M}_{k,l,+}$$ and
$$\psi _{2} := (\varphi _{s',r,\upsilon }, \, (s',r,\upsilon ) \in {\Bbb N}^{2n+\rg}_{s}
\times \{r,\upsilon \}) \in {\cal M}_{k,l,+},$$
$\psi = \kappa (\psi _{1}+\psi _{2})$.
\end{coro}

\begin{proof}
(i) It is sufficient to prove that $\varphi \wedge \hat{\varepsilon }$ is in
${\cal M}_{k,1,+}$ since $\vert (s,r,\upsilon ) \vert = l$. By Lemma~\ref{lfc2}(iii),
$$ \varphi = \varphi _{1}+\varphi _{2} \quad  \text{with} \quad
\varphi _{1} \in D_{k-1,{\goth p}}({\goth g})\wedge \oi_{0,0} \quad  \text{and} $$
$$ \varphi _{2} \in
\sum_{i=1}^{\n} \tk {\k[\lp]}{\k[\lg l*]}D_{k-1,{\goth p},j_{i}}({\goth g})
\wedge \bjj_{i}'+ JD_{k,{\goth p}}({\goth g}) + 
\tk {\k}{\k[\lg l*]}D_{k-1,{\goth p}}({\goth g})\wedge {\goth p}_{\u}$$
since $\k[\lg l*]=J+\k[\lp]$ and the restriction of $\varphi \wedge \varepsilon $ to
$\lp$ is equal to $0$. By definition, $\varphi _{2}\wedge \hat{\varepsilon }$ is in
$\kappa ({\cal M}_{k,1,+})$

By Proposition~\ref{p2pp}, setting:
$$ {\Bbb I}_{k-1,-} := \{(\poi j{-1}{,\ldots,}{\n}{}{}{}) \in {\Bbb N}^{\n+2}_{k-1} \;
\vert \; j_{1}\leq n_{1}-1,\ldots,j_{\n}\leq n_{\n}-1\},$$
$$ \varphi _{1} = \sum_{j=1}^{\rg} \varphi _{1,j}\wedge \varepsilon _{j}^{(0)}
\quad \text{with} \quad \varphi _{1,j} \in 
\tk {\k}{\k[\lp]}D_{k,{\Bbb I}_{k-1,-},{\goth p}}({\goth g})$$
for $j=1,\ldots,\rg$ since $\varepsilon _{j}(x)$ has a nonzero 
component on ${\goth d}_{l}$ for $l=1,\ldots,\n$ for all $x$ in a dense open subset of 
${\goth l}_{*}$ by Lemma~\ref{lrcm}(iii). As a result,
$$ \varphi _{1}  \in D_{k-1,{\goth p}}({\goth g})\wedge \widehat{\bii} + 
\tk {\k}{{\goth m}}\tk {\k}{\k[\lp]}D_{k,{\goth p}}({\goth g}) $$
since for $j=1,\ldots,\rg$, the map
$$ \xymatrix{G\times \lg l* \ar[rr] && {\goth g}}, \qquad 
(g,x) \longmapsto \varepsilon _{j}(g(x)) - \varepsilon _{j}(x)$$
is in $\tk {\k}{{\goth m}}\tk {\k}{\k[{\goth l}_{*}]}{\goth g}$ by 
Proposition~\ref{psc1}(vi), whence
$\varphi _{1}\wedge \hat{\varepsilon }\in \kappa ({\cal M}_{k,1,+})$ by definition
of $\kappa ({\cal M}_{k,1,+})$ and the assertion.

(ii) As in (i), it is sufficient to prove the assertion for $l=0$. Let $\varphi _{1}$
and $\varphi _{2}$ be as in (i). Since $\varphi $ is in
$\tk {\k}{\k[\lp]}D_{k,{\goth p}}({\goth g})$, by Lemma~\ref{lfc2}(ii),
$$ \varphi _{2} = \sum_{i=1}^{\n} D_{k-1,{\goth p},j_{i}}({\goth g})
\wedge \bjj_{i}'+ \tk {\k}{\k[\lp]}D_{k-1,{\goth p}}({\goth g})\wedge {\goth p}_{\u},$$
so that $\varphi _{2}\wedge \hat{\varepsilon }=\kappa (\psi _{1})$ with $\psi _{1}$
defined above. By (i),
$$ \varphi _{1} \in D_{k-1,{\goth p}}({\goth g})\wedge \widehat{\bii} + 
\tk {\k}{{\goth m}}\tk {\k}{\k[\lp]}D_{k,{\goth p}}({\goth g}) $$
so that $\varphi _{1}\wedge \hat{\varepsilon }$ is in
$(\tk {\k}{{\goth m}}{\k[\lp]})D_{k,{\goth p}}({\goth g})\wedge \hat{\varepsilon }$,
whence the assertion.
\end{proof}

\subsection{On power series} \label{fc3}
As $\han {}{}$ is a power series ring, so is $\tk {\k}{\han {}{}}V$ for any $\k$-vector
$V$. For $\varphi $ in $\tk {\k}{\han {}{}}V$, $\varphi $ has an expansion
$$ \varphi = \sum_{s\in {\Bbb N}^{2n+\rg}} z^{s} \varphi _{s}$$
with $\varphi _{s}$ in $V$. We denote by $\gamma (\varphi )$ the smallest $s$ such that
$\varphi _{s}\neq 0$. For $\varphi $ in ${\cal M}$ and $m$ nonnegative integer, set:
$$ \gamma (\varphi ) := \gamma (\kappa (\varphi )), \quad
Q_{\varphi } := \{(r,\upsilon ) \in {\Bbb N}_{{\goth p}} \, \vert
\varphi _{\gamma (\varphi ),r,\upsilon }\neq 0\}, \quad 
Q_{\varphi ,m} := \{(r,\upsilon ) \in Q_{\varphi } \, \vert \,
\vert (r,\upsilon ) \vert= m \}.$$
Denote by $l_{\varphi }$ the smallest integer such that
$Q_{\varphi ,l_{\varphi }}$ is nonempty and for $i$ nonnegative integer, set:
$$ Q_{\varphi ,l_{\varphi },i} := \{(r,\upsilon ) \in Q_{\varphi ,l_{\varphi }}
\, \vert \, \vert \upsilon  \vert= i \} .$$
Let $i_{\varphi }$ be the smallest integer such that
$Q_{\varphi ,l_{\varphi },i_{\varphi }}$ is nonempty.

\begin{lemma}\label{lfc3}
Let $k=0,\ldots,n$, $l$ a positive integer, $\varphi $ in
$\tk {\k}{\han {}{}}\tk {\k}{\k[\lg l*]}D({\goth g})$ and $\psi $ in ${\cal M}_{k,l,+}$.
Suppose $\varphi  = \kappa (\psi )$ and $\gamma (\psi )$ smaller than
$\gamma (\varphi )$. Then, for some $\psi '$ in ${\cal M}_{k,l,+}$,
$\varphi  = \kappa (\psi ')$ and $\gamma (\psi )$ is smaller than $\gamma (\psi ')$.
\end{lemma}

\begin{proof}
We can suppose that $\psi _{s,r,\upsilon }$ is in $\tk {\k}{\k[\lp]}D({\goth g})$ for
all $(s,r,\upsilon )$ since $\k[\lg l*]$ is the direct sum of $J$ and $\k[\lp]$. By
definition, for some positive integer $N$, $\psi _{s,r,\upsilon }=0$ if
$\vert r \vert > N$. Since ${\goth J}_{i}$ is empty for $i>n$, $\psi _{s,r,\upsilon }=0$
if $\vert (r,\upsilon ) \vert > N+n$ and by induction, it is sufficient to find $\psi '$
in ${\cal M}_{k,l,+}$ such that $\varphi =\kappa (\psi ')$, $(l_{\psi },i_{\psi })$
smaller than $(l_{\psi '},i_{\psi '})$ and $\psi '_{s,r,\upsilon }=0$ if
$\vert r \vert > N$.

Since $\hat{\varepsilon }-\varepsilon $ is in
$\tk {\k}{{\goth m}}\tk {\k}{\k[\lg l*]}\ex {\b g{}}{{\goth g}}$,
$$ \kappa (\psi )_{\gamma (\psi )} = \sum_{(r,\upsilon )\in {\Bbb N}_{{\goth p}}}
v^{r}\psi _{\gamma (\psi ),r,\upsilon }\wedge \lambda _{\upsilon }\wedge
\varepsilon .$$
As $\varphi =\kappa (\psi )$ and $\gamma (\psi )$ is smaller than
$\gamma (\varphi )$,
$$ \sum_{(r,\upsilon )\in {\Bbb N}_{{\goth p}}}
v^{r}\psi _{\gamma (\psi ),r,\upsilon }\wedge \lambda _{\upsilon }\wedge
\varepsilon = 0 .$$
By minimality of $l_{\psi }$ and Proposition~\ref{psp2}(i),
$$ \sum_{(r,\upsilon )\in Q_{\psi ,l_{\psi }}}
v^{r}\psi _{\gamma (\psi ),r,\upsilon }\wedge \lambda _{\upsilon }\wedge
\varepsilon \in M_{l_{\psi }+1}.$$
So, by Proposition~\ref{psp2}(ii),
$$ \psi _{\gamma (\psi ),r,\upsilon }\wedge \overline{\varepsilon _{0}}\wedge \mu _{+}
= 0$$
for all $(r,\upsilon )$ in $Q_{\psi ,l_{\psi }}$. Denote by $Q'_{\psi }$ the subset
of elements $(r,\upsilon )$ of ${\Bbb N}_{{\goth p}}$ such that for some
$(r',\upsilon ')$ in $Q_{\psi ,l_{\psi }}$, $r=r'$ and
$\vert \upsilon  \vert=\vert \upsilon ' \vert + 1$. Then, by Corollary~\ref{cfc2}(ii),
for some $\tilde{\psi }$ and $\tilde{\psi }'$ in ${\cal M}_{k,l,+}$,
$\tilde{\psi }_{s,r,\upsilon }=0$ for $s\neq \gamma (\psi )$ or
$(r,\upsilon  )\notin Q'_{\psi }$, $\tilde{\psi }'_{s,r,\upsilon }=0$ for
$s$ at most equal to $\gamma (\psi )$ or $\vert r \vert > N$ and
$$ \sum_{(r,\upsilon )\in Q_{\psi ,l_{\psi }}}
z^{\gamma (\psi )}v^{r}\psi _{\gamma (\psi ),r,\upsilon }\wedge \lambda _{\upsilon }\wedge
\hat{\varepsilon } = \kappa (\tilde{\psi }+\tilde{\psi }').$$
Denoting by $\psi '$ the element of ${\cal M}_{k,l,+}$ such that
$$ (r,\upsilon ) \in Q_{\psi ,l_{\psi }}
\Longrightarrow \psi '_{\gamma (\psi ),s,\upsilon } =
\tilde{\psi }_{\gamma (\psi ),s,\upsilon }, \quad
(r,\upsilon ) \in Q'_{\psi }\setminus Q_{\psi ,l_{\psi }} \Longrightarrow
\psi '_{\gamma (\psi ),r,\upsilon } = \psi _{\gamma (\psi ),s,\upsilon },$$
$$ (s,r,\upsilon ) \in {\Bbb N}^{2n+\rg}_{\gamma (\psi )}\times {\Bbb N}_{{\goth p}}
\Longrightarrow \psi '_{s,r,\upsilon }=\psi _{s,r,\upsilon }+
\tilde{\psi }'_{s,r,\upsilon }$$
and $\psi '_{s,r,\upsilon }=0$ for $s$ smaller than $\gamma (\psi )$. Then
$\varphi =\kappa (\psi ')$, $\psi '_{s,r,\upsilon }=0$ for $\vert r \vert>N$ and
$\gamma (\psi )$ is smaller than $\gamma (\psi ')$ or $\gamma (\psi )=\gamma (\psi ')$
and $(l_{\psi },i_{\psi })$ is smaller than $(l_{\psi '},i_{\psi '})$, whence the
lemma.
\end{proof}

\begin{prop}\label{pfc3}
Let $l$ be a positive integer and $\varphi $ in ${\cal M}_{k,l}$. Then
$\kappa (\varphi )$ is in $\kappa ({\cal M}_{k,l+1,+})$ if and only if 
the restriction of $\varphi _{s,r,\upsilon }\wedge \varepsilon $ to $\lp$ is
equal to $0$ for all $(s,r,\upsilon )$. 
\end{prop}

\begin{proof}
By Corollary~\ref{cfc2}(i), the condition is sufficient. Suppose that
$\kappa (\varphi )=\kappa (\psi )$ for some $\psi $ in $\kappa ({\cal M}_{k,l+1,+})$ and
the restriction of $\varphi _{s,r,\upsilon }\wedge \varepsilon $ to $\lp$ is different
from $0$ for some $(s,r,\upsilon )$. A contradiction is expected. As the condition is
sufficient, we can suppose that for some $(r,\upsilon )$, the restriction of
$\varphi _{\gamma (\varphi ),r,\upsilon }\wedge \varepsilon $ to $\lp$ is different from
$0$. Since $\hat{\varepsilon }-\varepsilon $ is in
$\tk {\k}{{\goth m}}\tk {\k}{\k[\lg  l*]}{\goth g}$,
$$ \kappa (\varphi )_{\gamma (\varphi )} = \sum_{(r,\upsilon )\in {\Bbb N}_{{\goth p}}}
v^{r}\varphi _{\gamma (\varphi ),r,\upsilon }\wedge \lambda _{\upsilon }\wedge
\varepsilon \quad  \text{and} \quad
\kappa (\psi )_{\gamma (\psi )} = \sum_{(r,\upsilon )\in {\Bbb N}_{{\goth p}}}
v^{r}\psi _{\gamma (\psi ),r,\upsilon }\wedge \lambda _{\upsilon }\wedge
\varepsilon .$$
By Lemma~\ref{lfc3}, for some $\psi '$ in ${\cal M}_{k,l+1,+}$,
$\kappa (\varphi )=\kappa (\psi ')$ and $\gamma (\psi ')$ is at least equal to
$\gamma (\varphi )$. Hence, by Proposition~\ref{psp2}(i),
$$ \sum_{(r,\upsilon )\in {\Bbb N}_{{\goth p}}}
v^{r}\varphi _{\gamma (\varphi ),r,\upsilon }\wedge \lambda _{\upsilon } \in
M_{l-\vert \gamma (\varphi ) \vert+1}$$
since $\psi '$ is in ${\cal M}_{k,l+1,+}$. So, by Proposition~\ref{psp2}(ii),
the restriction of $\varphi _{\gamma (\varphi ),r,\upsilon }\wedge \varepsilon $ to
$\lp$ is equal to $0$ for all $(\gamma (\varphi ),r,\upsilon )$, whence the
contradiction.  
\end{proof}

\subsection{Filtration and associate graded space} \label{fc4}
Let $k=1,\ldots,n$. For $l$ nonnegative integer, set $F_{l}:=\kappa ({\cal M}_{k,l,+})$.
Then the sequence $\poi F0{,}{1}{}{}{},\ldots$ is a decreasing filtration of
$D_{k,{\goth p}}({\goth g},\widehat{\bii})$. Denote by 
$\gr D_{k,{\goth p}}({\goth g},\widehat{\bii})$ the associate graded space to this
filtration and $\gr _{l} D_{k,{\goth p}}({\goth g},\widehat{\bii})$ the subspace of
degree $l$ of $\gr D_{k,{\goth p}}({\goth g},\widehat{\bii})$.

For $j=0,\ldots,k$, let $D_{j,{\goth p}}^{\bullet}({\goth g},\oi)$ be the graded
subcomplex of $\tk {\k}{\k[\lp]}D^{\bullet}({\goth g})$, 
$$ D_{j,{\goth p}}^{\bullet}({\goth g},\oi) := 
D_{j,{\goth p}}^{\bullet}({\goth g})[-\b g{}]\wedge \ex {\b g{}}{\oi} ,$$
and $A$ the algebra 
$$ A := \tk {\k}{\k[\poi z1{,\ldots,}{2n+\rg}{}{}{}]}\tk {\k}{\es S{{\goth p}_{\u}}}
\ex {}{\bjj} .$$
This algebra has a bigradation $A_{\bullet}^{\bullet}$ such that
$$ A^{i} := \tk {\k}{\k[\poi z1{,\ldots,}{2n+\rg}{}{}{}]}\tk {\k}{\es S{{\goth p}_{\u}}}
\ex {i}{\bjj} \quad  \text{and} $$ $$
A_{l} := \bigoplus_{(j_{1},j_{2},j_{3})\in {\Bbb N}^{3}_{l}}
\tk {\k}{\k[\poi z1{,\ldots,}{2n+\rg}{}{}{}]_{j_{1}}}
\tk {\k}{\sy {j_{2}}{{\goth p}_{\u}}}\ex {j_{3}}{\bjj}$$
with $\k[\poi z1{,\ldots,}{2n+\rg}{}{}{}]_{j}$ the space of homogeneous polynomials
of degree $j$ of $\k[\poi z1{,\ldots,}{2n+\rg}{}{}{}]$. Consider on 
$\tk {\k}AD({\goth g})$ the total gradation deduced from the double 
gradation $\tk {\k}{A^{\bullet}}D^{\bullet}({\goth g})$. The structure of complex on
$D({\goth g})$ induces a structure of complex on $\tk {\k}{A}D({\goth g})$. For
$l=0,1,\ldots$ and $i$ nonnegative integer, set:
$$ C^{i}_{l} = \bigoplus _{j=0}^{i} \bigoplus _{(l_{1},l_{2})\in {\Bbb N}^{2}_{l-j}}
\tk {\k}{\k[\poi z1{,\ldots,}{2n+\rg}{}{}{}]_{l_{1}}}\tk {\k}{\sy {l_{2}}{{\goth p}_{\u}}}
\tk {\k[\ll]}{\ex j{\bjj}}D_{k-j,{\goth p}}^{i-j}({\goth g},\oi) .$$
Denoting by $C_{l}^{\bullet}$ the sum of $C_{l}^{i},\, i=0,1,\ldots$, $C_{l}^{\bullet}$
is a graded subcomplex of $\tk {\k}AD({\goth g})$.

\begin{prop}\label{pfc4}
{\rm (i)} For $l=0,1,\ldots$, $F_{l}^{\bullet}$ is a graded subcomplex of the graded 
complex $D_{k,{\goth p}}^{\bullet}({\goth g},\widehat{\bii})$.

{\rm (ii)} For $l=n,n+1,\ldots$ and $i$ a positive integer, $F^{i}_{l}$ is
contained in $\hat{J}^{l-n}D_{k}^{i}({\goth g},\widehat{\bii})$. 

{\rm (iii)} For $l=0,1,\ldots$, the graded complex 
$\gr _{l}D_{k,{\goth p}}^{\bullet}({\goth g},\widehat{\bii})$ is isomorphic to
$C_{l}^{\bullet}$.
\end{prop}

\begin{proof}
(i) By definition, $F_{l}$ is a graded subspace of
$D_{k,{\goth p}}({\goth g},\widehat{\bii})$. For $j=l,l+1,\ldots$,
$(s,r,\upsilon )$ in ${\Bbb N}_{{\goth p},+}$ such that
$\vert (s,r,\upsilon ) \vert=j$ and $\varphi $ in
$\tk {\k}{\k[\lg l*]}D_{k-\vert \upsilon  \vert,{\goth p}}({\goth g})$,
$$ \dd z^{s}v^{r} \varphi \wedge \lambda _{\upsilon } =
z^{s}v^{r}\dd (\varphi )\wedge \lambda _{\upsilon } .$$
Hence $F_{l}$ is a graded subcomplex of $D_{k,{\goth p}}({\goth g},\widehat{\bii})$.

(ii) By definition,
$$F^{i}_{l} \subset \bigoplus _{j=0}^{n} \hat{J}^{l-j}\ex j{\bjj}\wedge
D_{k-j}^{i-j}({\goth g},\widehat{\bii})\subset
\bigoplus _{j=0}^{n} \hat{J}^{l-j}D_{k}^{i}({\goth g},\widehat{\bii}),$$
whence the inclusion since $\hat{J}^{l-j}$ is contained in $\hat{J}^{l-n}$ for
$j=0,\ldots,n$.

(iii) Denote by $\tau _{l}$ the quotient map
$$ \xymatrix{ F_{l} \ar[rr]^-{\tau _{l}} &&
\gr_{l}D_{k,{\goth p}}({\goth g},\widehat{\bii})} .$$
Let $\kappa _{l}$ be map
$$ \xymatrix{ {\cal M}_{k,l} \ar[rr]^-{\kappa _{l}} && C_{l}}, \qquad
(\varphi _{s,r,\upsilon } ) \longmapsto  \sum_{(s,r,\upsilon )}
z^{s}\tens v^{r}\tens \lambda _{\upsilon }\tens \varphi _{s,r,\upsilon }\wedge
\varepsilon .$$
By Proposition~\ref{pfc3}, this map defines through the quotient an injective map
from $\tau _{l}\rond \kappa _{l}({\cal M}_{k,l})$ to $C_{l}$, whence a bijective map 
$$ \xymatrix{ \tau _{l}\rond \kappa _{l} ({\cal M}_{k,l}) \ar[rr] && C_{l}} $$
since it is clearly surjective. Moreover, it is an isomorphism of graded complexes from
$\gr_{l}D_{k,{\goth p}}^{\bullet}({\goth g},\widehat{\bii})$ onto $C_{l}^{\bullet}$.
\end{proof}

\begin{rema}\label{rfc4}
According to Proposition~\ref{pfc4}(iii), the graded complexes
$\gr_{l}D_{k,{\goth p}}^{\bullet}({\goth g},\widehat{\bii})$ and $C_{l}^{\bullet}$
are identified.
\end{rema}

\subsection{Some cohomological results} \label{fc5}
Set:
\begin{eqnarray*}
\oi _{+} := & \oi _{0} \oplus \tk {\k}{\k[\lp]}{\goth p}_{\u}  \\
D_{k,{\goth p}}^{\bullet}(V,\oi_{+}) := &
D_{k,{\goth p}}^{\bullet}(V)[-\b g{}]\wedge \ex {\b g{}}{\oi _{+}} \\ 
D_{l,{\goth p}}^{\bullet}({\goth l},\oi_{+}) := &
D_{l,\#}^{\bullet}({\goth l})[-\b g{}]\wedge \ex {\b g{}}{\oi _{+}}  
\end{eqnarray*}
with $V={\goth p}_{-}$ or $V={\goth g}$ and $l=0,\ldots,n-d$. Then $\oi _{+}$ is a free
module of rank $\b g{}$ containing $\oi$ by Proposition~\ref{prcm}(iii),
$D_{k,{\goth p}}^{\bullet}(V,\oi_{+})$ is a graded subcomplex of
$D_{k}^{\bullet}(V,\oi _{+})$ and $D_{l,\#}^{\bullet}({\goth l},\oi_{+})$ is a graded
subcomplex of $D_{l}^{\bullet}({\goth l},\oi _{+})$. 

\begin{lemma}\label{lfc5}
{\rm (i)} For $k=0,\ldots,n-d$, $D_{k,\#}^{\bullet}({\goth l},\oi _{+})$ has no 
cohomology of degree different from $\b g{}$.

{\rm (ii)} For $k=0,\ldots,n$,
$D_{k,{\goth p}}^{\bullet}({\goth p}_{-},\oi _{{\goth l}})$ has no cohomology of degree
different from $\b l{}$.

{\rm (iii)} For $k=0,\ldots,n$, $D_{k,{\goth p}}^{\bullet}({\goth p}_{-},\oi _{+})$ has
no cohomology of degree different from $\b g{}$.

{\rm (iv)} For $k=0,\ldots,n$, $D_{k,{\goth p}}^{\bullet}({\goth g},\oi _{+})$ has no
cohomology of degree different from $\b g{}$.
\end{lemma}

\begin{proof}
(i) According to Lemma~\ref{lpp}, the complex $D_{k,\#}^{\bullet}({\goth l},\bi l{})$
has no cohomology of degree different from $\b l{}$ since the simple factors of
${\goth d}$ have Property $({\bf P})$. By restriction to the principal open subset
$\lpq l*$ of $\gg l{}$, the subcomplex $D_{k,\#}^{\bullet}({\goth l},\oi _{{\goth l}})$
of $D_{k}^{\bullet}({\goth l},\oi _{{\goth l}})$ has no cohomology of degree different
from $\b l{}$. As $\k[\lp] = \tk {\k}{\k[{\goth p}_{\u}]}\k[\lpq l*]$, $\oi _{+}$ is
the direct sum of $\tk {\k}{\k[{\goth p}_{\u}]}\oi _{{\goth l}}$ and 
$\tk {\k}{\k[\lp]}{\goth p}_{\u}$ by Proposition~\ref{prcm}(ii). As a result,
$$ D_{k,\#}^{\bullet}({\goth l},\oi _{+}) = \tk {\k}{\k[{\goth p}_{\u}]}
D_{k,\#}^{\bullet}({\goth l,\oi _{{\goth l}}})[-d]\wedge \ex d{{\goth p}_{\u}} ,$$
whence an isomorphism of graded compexes
$$ \xymatrix{\tk {\k}{\ex d{{\goth p}_{\u}}}
\tk {\k}{\k[{\goth p}_{\u}]}{D_{k,\#}^{\bullet}({\goth l,\oi _{{\goth l}}})[-d]} 
\ar[rr] && D_{k,\#}^{\bullet}({\goth l},\oi _{+})} ,$$
and the assertion.

(ii) and (iii) As ${\goth p}_{-}$ is the direct sum of ${\goth l}$ and
${\goth p}_{-,\u}$, $D_{k,{\goth p}}^{\bullet}({\goth p}_{-},\oi _{{\goth l}})$ and 
$D_{k,{\goth p}}^{\bullet}({\goth p}_{-},\oi _{+})$ are isomorphic to the 
total complexes deduced from the double complexes:
$$ \bigoplus _{j=\sup\{0,k-n+d\}}^{k} \tk {\k}{D_{j}^{\bullet}({\goth p}_{-,\u})}
D_{k-j,\#}^{\bullet}({\goth l},\oi _{{\goth l}}) \quad  \text{and} \quad 
\bigoplus _{j=\sup\{0,k-n+d\}}^{k} \tk {\k}{D_{j}^{\bullet}({\goth p}_{-,\u})}
D_{k-j,\#}^{\bullet}({\goth l},\oi _{+}) $$
respectively. By Lemma~\ref{lco1}(ii), for $j$ positive integer, 
$D_{j}^{\bullet}({\goth p}_{-,\u})$ is an acyclic complex, whence the assertion by (i)
and its proof.

(iv) As ${\goth p}_{\u}$ is contained in $\oi _{+}$, 
$$ D_{k,{\goth p}}^{\bullet}({\goth g},\oi_{+}) =  \bigoplus_{j=0}^{k} 
\tk {\k}{\sy {k-j}{{\goth p}_{\u}}}
D_{k-j,{\goth p}}^{\bullet}({\goth p}_{-},\oi_{+}) .$$ 
So, by (iii), $D_{k,{\goth p}}^{\bullet}({\goth g},\oi_{+})$ has no cohomology of degree
different from $\b g{}$.
\end{proof}

\begin{coro}\label{cfc5}
For $k=0,\ldots,n$, $D_{k,{\goth p}}^{\bullet}({\goth g},\oi)$ has no cohomology of
degree different from $\b g{}$.
\end{coro}

\begin{proof}
As the modules $\oi _{+}$ and $\oi$ are free modules of rank $\b g{}$ and $\oi$ is 
contained in $\oi _{+}$, for some $p$ in $\k[\lp]\setminus \{0\}$, 
$\ex {\b g{}}{\oi}=p\ex {\b g{}}{\oi _{+}}$. As a result, the map
$$ \xymatrix{ \tk {\k}{\e Sg}\ex {}{{\goth g}}\wedge \ex {\b g{}}{\oi _{+}}
\ar[rr]^{\tau } && 
\tk {\k}{\e Sg}\ex {}{{\goth g}}\wedge \ex {\b g{}}{\oi }}, 
\qquad \varphi \longmapsto p\varphi $$
is an isomorphism of graded complexes such that
$$ \tau (D_{k,{\goth p}}^{\bullet}({\goth g},\oi _{+})) =
D_{k,{\goth p}}^{\bullet}({\goth g}, \oi ) .$$
So, by Lemma~\ref{lfc5}(iv), $D_{k,{\goth p}}^{\bullet}({\goth g}, \oi )$ has no
cohomology of degree different from $\b g{}$.
\end{proof}

For $j,k$ integers such that $0\leq j \leq k \leq n$, denote by 
$D_{k,j,{\goth p}}^{\bullet}({\goth g})$ the graded subspace of 
$D_{k,{\goth p}}^{\bullet}({\goth g})$, 
$$ D_{k,j,{\goth p}}^{\bullet}({\goth g}) := 
\sy j{{\goth p}_{\u}}D_{k-j,{\goth p}}^{\bullet}({\goth p}_{-})$$ 
and $D_{k,j,{\goth p}}^{\bullet}({\goth g},\oi)$ the graded subspace of 
$D_{k,{\goth p}}^{\bullet}({\goth g},\oi)$, 
$$ D_{k,j,{\goth p}}^{\bullet}({\goth g},\oi) := 
D_{k,j,{\goth p}}^{\bullet}({\goth g})[-\b g{}]\wedge \ex {\b g{}}{\oi} .$$
In particular, 
$$ D_{k,0,{\goth p}}({\goth g},\oi) = D_{k,{\goth p}}({\goth p}_{-},\oi) :=
D_{k,{\goth p}}^{\bullet}({\goth p}_{-})[-\b g{}]\wedge \ex {\b g{}}{\oi} .$$
Since ${\goth p}_{\u}\wedge \ex {\b g{}}{\oi}=\{0\}$, 
$D_{k,j,{\goth p}}^{\bullet}({\goth g},\oi)$ is a graded subcomplex of 
$D_{k,{\goth p}}^{\bullet}({\goth g},\oi)$.

\begin{lemma}\label{l2fc5}
Let $k=1,\ldots,n$.

{\rm (i)} For $\varphi $ in $D_{k,{\goth p}}^{\b g{}}({\goth p}_{-},\oi)$, $\varphi $ is
a cocycle of $D_{k,{\goth p}}^{\bullet}({\goth g},\oi)$ if and only if $\varphi $ is in 
$\tk {\k[\ll]}{\sy {k}{\oi _{{\goth l}}}}\ex {\b g{}}{\oi}$.

{\rm (ii)} For $0\leq j\leq k$ and $\varphi $ in
$D_{k,j,{\goth p}}^{\b g{}}({\goth g},\oi)$,
$\varphi $ is a cocycle of $D_{k,{\goth p}}^{\bullet}({\goth g},\oi)$ if and only if 
$$\varphi \in \tk {\k[\ll]}{\sy j{{\goth p}_{\u}}\sy {k-j}{\oi _{{\goth l}}}}
\ex {\b g{}}{\oi}.$$
\end{lemma}

\begin{proof}
Denote again by $\varepsilon $ the restriction of $\varepsilon $ to $\lp$. Then 
$\varepsilon $ is a generator of $\ex {\b g{}}{\oi}$. As
$$ \mu _{+} := \poi v1{\wedge \cdots \wedge }{d}{}{}{},$$
$\overline{\varepsilon _{0}}\wedge \mu _{+}$ is a generator of the $\k[\lp]$-module
$\ex {\b g{}}{\oi_{+}}$.

(i) Let $\psi $ be in $\tk {\k}{\k[\lp]}D_{k,{\goth p}}^{0}({\goth p}_{-})$. Then
$\psi $ has an expansion
$$ \psi = \sum_{j=0}^{k} \sum_{r\in {\Bbb N}^{d}_{j}} w^{r}\psi _{r} \quad  \text{with} 
\quad \psi _{r} \in \tk {\k}{\k[\lp]}D^{0}_{k-\vert r \vert,\#}({\goth l}) .$$ 
Since $\ex {\b g{}}{\oi}$ is a submodule of $\ex {\b g{}}{\oi_{+}}$,
$\psi \wedge \varepsilon $ is a cocycle of $D_{k,{\goth p}}^{\bullet}({\goth g},\oi)$ if
and only if
$\psi \wedge \overline{\varepsilon _{0}}\wedge \mu _{+}$ is a cocycle of
$D_{k,{\goth p}}^{\bullet}({\goth g},\oi_{+})$. In particular, the condition of the
assertion is sufficient by Proposition~\ref{prcm}(ii).

Suppose that $\psi \wedge \varepsilon $ is a cocycle of 
$D_{k,{\goth p}}^{\bullet}({\goth g},\oi)$. Denote by $\nu (\psi )$ the biggest element 
$(\vert r \vert,r)$ of ${\Bbb N}\times {\Bbb N}^{d}$ such that $\psi _{r}\neq 0$. Suppose
that $\nu (\psi )$ is different from $(0,0)$. A contradiction is expected. Let $j$ be 
the smallest indice such that $r_{j}\neq 0$. Denote by $\tilde{r}$ the element of
${\Bbb N}^{d}_{\vert r \vert-1}$ such that $\tilde{r}_{l}=r_{l}$ for
$l\neq j$. Let ${\goth R}_{r,\tilde{r}}$ be the subset of elements $r'$ of
${\Bbb N}^{d}_{\vert r \vert}$ such that $w^{\tilde{r}}$ divides $w^{r'}$. For $r'$ in
${\goth R}_{r,\tilde{r}}$, let $k_{r'}$ be the indice such that
$w^{r'}=w_{k_{r'}}w^{\tilde{r}}$. By Corollary~\ref{csc3}, as $\psi \wedge \varepsilon $
is a cocycle, by maximality of $\nu (\psi )$ and minimality of $j$,
$$ \sum_{r'\in {\goth R}_{r,\tilde{r}}}
r'_{k_{r'}}w^{\tilde{r}}\psi _{r'}\tens w_{k_{r'}} = 0.$$
In particular, $\psi _{r}=0$, whence the contradition. As a result, $\psi =\psi _{0}$.
As a matter of fact, $\psi = 0$ when $k>n-d$. Otherwise,
$\psi \wedge \overline{\varepsilon _{0}}$ is a cocycle of 
$\tk {\k}{\k[{\goth p}_{\u}]}{D_{k,\#}^{\bullet}({\goth l},\oi_{{\goth l}}})$. So,
by Lemma~\ref{l3co1}(iii) and Proposition~\ref{prcm}(ii), 
$\psi $ is in $\tk {\k}{\k[{\goth p}_{\u}]}\sy k{\oi_{{\goth l}}}$, whence the assertion.

(ii) Let $\psi $ be in $\tk {\k}{\k[\lp]}D_{k,j,{\goth p}}^{0}({\goth g})$. Since 
$\psi \wedge \varepsilon $ is a cocycle if and only if 
$\psi \wedge \overline{\varepsilon _{0}}\wedge \mu _{+}$ is a cocycle of degree $\b g{}$
of $D_{k,{\goth p}}^{\bullet}({\goth g},\oi_{+})$, the condition is sufficient. Suppose
that $\psi \wedge \varepsilon $ is a cocycle of
$D_{k,{\goth p}}^{\bullet}({\goth g},\oi)$. The element $\psi $ has an expansion
$$ \psi = \sum_{r\in {\Bbb N}^{d}_{j}} v^{r}\psi _{r} \quad  \text{with} 
\quad \psi _{r} \in \tk {\k}{\k[\lp]}D^{0}_{k-j,{\goth p}}({\goth p}_{-}) .$$
Since ${\goth p}_{\u}\wedge \ex {\b g{}}{\oi}=\{0\}$, for all $r$, 
$\psi _{r}\wedge \varepsilon $ is a cocycle of $D_{k-j,{\goth p}}({\goth p}_{-},\oi)$,
whence the assertion by (i).
\end{proof}

\subsection{Annulation of cohomology} \label{fc6}
Let $k=1,\ldots,n$. For $i$ integer, denote by $Z^{i}$ and $B^{i}$ the spaces of 
cocycles and coboundaries of degree $i$ of
$D_{k,{\goth p}}^{\bullet}({\goth g},\widehat{\bii})$.
For $j=1,\ldots,n$, denote by $D_{j,\#}^{\bullet}(\bjj')$ the total graded submodule of
$\tk {\k}{\k[\lp]}D_{j}^{\bullet}({\goth l})$ deduced from the
multigraded module
$$ \bigoplus _{i\in {\Bbb I}''_{j}} D_{i_{1}}^{\bullet}(\bjj_{1}')\wedge \cdots \wedge
D_{i_{\n}}^{\bullet}(\bjj'_{\n}) .$$
For $l=1,\ldots,n$, let $D_{l,{\goth p}}^{\bullet}(\bjj)$ be the total graded submodule
of $D_{l}^{\bullet}(\bjj)$ deduced from the bigraded module
$$ \bigoplus _{j=0}^{l} \tk {\k}{D_{l-j}^{\bullet}({\goth p}_{\u})}
D_{j,\#}^{\bullet}(\bjj') .$$
Set: 
$$D_{l,*}^{\bullet} := D_{l,{\goth p}}^{\bullet}(\bjj)[-\b g{}]\wedge 
\ex {\b g{}}{\widehat{\bii}} .$$
Then $D_{l,*}^{\bullet}$ is a graded subcomplex of
$D_{l,{\goth p}}^{\bullet}({\goth g},\widehat{\bii})$.

\begin{lemma}\label{lfc6}
Let $l=1,\ldots,n$. The morphism 
$$ \xymatrix{ \tk {\k}{\hat{\an {}{}}}\tk {\k[\lp]}{\k[\lg l*]}
D_{l,{\goth p}}^{\bullet}(\bjj)[-\b g{}] \ar[rr] && D_{l,*}^{\bullet}}, 
\qquad \varphi \longmapsto \varphi \wedge \hat{\varepsilon }$$
is an isomorphism of graded complexes. In particular, $D_{l,*}^{\bullet}$ is acyclic.
\end{lemma}

\begin{proof}
Denote by $\tilde{D}_{l,*}^{\bullet}$ the graded subcomplex of 
$D_{l,{\goth p}}^{\bullet}({\goth g},\widetilde{\bii})$,
$$ \tilde{D}_{l,*}^{\bullet} :=
D_{l,{\goth p}}^{\bullet}(\bjj)[-\b g{}]\wedge \ex {\b g{}}{\widetilde{\bii}}.$$
For $(x,y)$ in $\lg l*$, denote by $\bjj _{x,y}$ the image of $\bjj$ by the evalutaion
map $(x,y)\mapsto \varphi (x,y)$. Recall that $(h,e)$ is an element of
$\Omega _{{\goth g}}\cap (({\goth h}\cap {\goth l}_{*})\times {\goth u})$. For $g$ in
the normalizer of ${\goth h}$ in $G$ such that $g({\goth b})={\goth b}_{-}$, $V_{g(h),g(e)}={\goth b}_{-}$ and $\bjj_{h,e}={\goth u}$ by Corollary~\ref{csc3}. As a result, for all
$(g,x,y)$ in a nonempty open subset of $G\times \lg l*$,
$V_{g(x),g(y)}\cap \bjj _{x,y}=\{0\}$. Hence the morphism
$$ \xymatrix{ \tk {\k}{\an {}{}}\tk {\k[\lp]}{\k[\lg l*]}
D_{l,{\goth p}}^{\bullet}(\bjj)[-\b g{}] \ar[rr] && \tilde{D}_{l,*}^{\bullet}}, 
\qquad \varphi \longmapsto \varphi \wedge \hat{\varepsilon }$$
is an isomorphism of graded complexes since it is surjective. By Lemma~\ref{lco1}(ii), 
$D_{l}^{\bullet}(\bjj)$ is an acyclic complex since $l$ is positive and $\bjj$ is a free
$\k[\lp]$-module. By Lemma~\ref{lco3}(iii), the complex
$D_{l,{\goth p}}^{\bullet}(\bjj)$ is a direct factor of $D_{l}^{\bullet}(\bjj)$.
Hence $\tilde{D}_{l,*}^{\bullet}$ and $D_{l,*}^{\bullet}$ are acyclic
since $\hat{\an {}{}}$ is a faithfully flat extension of $\an {}{}$.
\end{proof}

For $j=0,\ldots,k$, denote by $D^{\bullet}_{k,j,{\goth p}}$ the intersection of 
$\tk {\k}{\k[\lp]}D_{k,{\goth p}}^{\bullet}({\goth g})$
and $\sy {j}{\oi_{0,0}}D_{k-j}^{\bullet}(\bjj)$. 
For $i=\b g{},\ldots,k+\b g{}$ and $l$ nonnegative integer, set:
$$ K^{i} := \sum_{j=0}^{k-i+\b g{}}D^{i-\b g{}}_{k,j,{\goth p}}\wedge 
\ex {\b g{}}{\widehat{\bii}} 
\quad  \text{and} \quad $$
$$K_{l}^{i} = \bigoplus _{(l_{1},l_{2})\in {\Bbb N}^{2}_{l-i+\b g{}}}
\tk {\k}{\k[\poi z1{,\ldots,}{2n+\rg}{}{}{}]_{l_{2}}}
\tk {\k}{\sy {l_{2}}{{\goth p}_{\u}}}\sum_{j=0}^{k-i+\b g{}}
D^{i-\b g{}}_{k,j,{\goth p}}\wedge \hat{\varepsilon }.$$
For $j=0,\ldots,k-i+\b g{}$, denote by $D^{0}_{k,j,i-\b g{},{\goth p},*}$ the image of 
$D_{k,j,{\goth p}}^{i-\b g{}}\wedge \hat{\varepsilon }$ by the quotient morphism
$\xymatrix{F_{i-\b g{}} \ar[r] &
\gr_{i-\b g{}}D_{k,{\goth p}}({\goth g},\widehat{\bii})}$.

\begin{lemma}\label{l2fc6}
Let $i=\b g{}+1,\ldots,k+\b g{}$ and $l$ a nonnegative integer.

{\rm (i)} The space $K^{i}\cap F^{i}_{l}$ is contained in $K^{i}_{l}+F^{i}_{l+1}$.
Moreover the sum $K^{i}_{l}+F^{i}_{l+1}$ is direct.

{\rm (ii)} The subspace $Z^{i}\cap F^{i}_{l}$ of $F^{i}_{l}$ is contained
$\dd F^{i-1}_{l} + K^{i}_{l} + F^{i}_{l+1}$.
\end{lemma}

\begin{proof}
(i) By definition, $K^{i}_{l}$ is contained in $K^{i}\cap F^{i}_{l}$. The image of
$K^{i}\cap F^{i}_{l}$ in $\gr_{l}D_{k,{\goth p}}({\goth g},\widehat{\bii})$ is contained
in the space of bidegree $(i-\b g{},\b g{})$ of $\tk {\k}AD_{k,{\goth p}}({\goth g},\oi)$.
As this space is the image of $K^{i}_{l}$ in
$\gr_{l}D_{k,{\goth p}}({\goth g},\widehat{\bii})$, $K^{i}\cap F^{i}_{l}$ is contained in
$K^{i}_{l}+F^{i}_{l+1}$. For $j=0,\ldots,k-i+\b g{}$ and $\varphi $ in
$D_{k,j,{\goth p}}^{i-\b g{}}$, $\varphi \wedge \varepsilon = 0$  if and only if
$\varphi \wedge \hat{\varepsilon }= 0$ by the proof of Lemma~\ref{lfc6}. Hence the sum
$K^{i}_{l}+F^{i}_{l+1}$ is direct.

(ii) Let $\varphi $ be in $Z^{i}\cap F^{i}_{l}$ and $\overline{\varphi }$ its image in
$\gr _{l}D_{k,{\goth p}}^{i}({\goth g},\widehat{\bii)}$. By Proposition~\ref{pfc4}, 
$\overline{\varphi }$ is a cocycle of degree $i$ of the graded complex $C^{\bullet}_{l}$.
By Corollary~\ref{cfc5}, for $j=0,\ldots,k$, the complex 
$D_{k-j,{\goth p}}^{\bullet}({\goth g},\oi)$ has no cohomolgy of degree different from
$\b g{}$. Then, for some $\overline{\psi }$ in $C^{i-1}_{l}$, 
$$ \overline{\varphi } - \dd \overline{\psi } \in 
\bigoplus _{(l_{1},l_{2})\in {\Bbb N}^{2}_{l-i+\b g{}}}
\tk {\k}{\k[\poi z1{,\ldots,}{2n+\rg}{}{}{}]_{l_{1}}}\tk {\k}{\sy {l_{2}}{{\goth p}_{\u}}}
\tk {\k}{\ex {i-\b g{}}{\bjj}}D_{k-i+\b g{},{\goth p}}^{\b g{}}({\goth g},\oi).$$
Then, by Lemma~\ref{l2fc5}(ii), 
$$ \overline{\varphi } - \dd \overline{\psi } \in 
\bigoplus _{(l_{1},l_{2})\in {\Bbb N}^{2}_{l-i+\b g{}}}
\tk {\k}{\k[\poi z1{,\ldots,}{2n+\rg}{}{}{}]_{l_{1}}}
\tk {\k}{\sy {l_{2}}{{\goth p}_{\u}}}\tk {\k}{\ex {i-\b g{}}{\bjj}} 
\sum_{j=0}^{k-i+\b g{}}D^{0}_{k,j,i-\b g{},{\goth p},*} .$$
So, for a representative $\psi $ of $\overline{\psi }$ in
$F^{i-1}_{l}$, 
$$ \varphi - \dd \psi \in K^{i}_{l}+F^{i}_{l+1} $$
by (i) and Proposition~\ref{pfc4}(iii), whence the assertion.
\end{proof}

Let $\dd _{1}$ and $\dd _{2}$ be the morphisms from $K_{l}^{\bullet}$ to 
$F^{\bullet}_{l}$ such that
$$ \dd _{1} a\tens \omega \nu \wedge \hat{\varepsilon } =
(-1)^{i-\b g{}}a\tens \nu (\dd \omega )\wedge \hat{\varepsilon }, \quad  
\dd _{2} a\tens \omega \nu \wedge \hat{\varepsilon } = 
a\tens \omega \dd \nu \wedge \hat{\varepsilon } \quad  \text{with} $$ $$
a\in \bigoplus _{(l_{1},l_{2})\in {\Bbb N}^{2}_{l-i+\b g{}}}
\tk {\k}{\k[\poi z1{,\ldots,}{2n+\rg}{}{}{}]_{l_{2}}}\sy {l_{2}}{{\goth p}_{\u}}, $$ $$
\omega \in \sy j{\oi_{0,0}}, \quad 
\nu \in D_{k-j}^{i-\b g{}}(\bjj), \quad \omega \nu \in D^{i-\b g{}}_{k,j,{\goth p}} ,
\quad i = \b g{},\ldots,k+\b g{}, \quad j=0,\ldots,k-i+\b g{} .$$

\begin{lemma}\label{l3fc6}
Let $i=\b g{},\ldots,k+\b g{}$ and $l$ a nonnegative integer.

{\rm (i)} The space $\dd K_{l}^{i}$ is contained in $\dd _{2}K_{l}^{i}+F^{i+1}_{l+1}$.

{\rm (ii)} If $i>\b g{}$ then $Z^{i}\cap F^{i}_{l}$ is contained in the sum
$$\dd F^{i-1}_{l} + F^{i}_{l+1}.$$
\end{lemma}

\begin{proof}
(i) By definition the restriction of $\dd$ to $K^{\bullet}_{l}$ is equal to
$\dd_{1}+\dd _{2}$. Let $a$, $\omega $, $\nu $ be as in the above definition. Then 
$$\dd a\tens \omega \nu \wedge \hat{\varepsilon } \in  
\dd_{2} a\tens \omega \nu \wedge \hat{\varepsilon } + 
a\tens D_{k,j-1,{\goth p}}^{i-\b g{}}\wedge \hat{\varepsilon }\wedge \oi_{0,0}
\cap D_{k,{\goth p}}^{i+1}({\goth g},\widehat{\bii}).$$
By Corollary~\ref{cfc2}(i),
$$ a\tens D_{k,j-1,{\goth p}}^{i-\b g{}}
\wedge \hat{\varepsilon }\wedge \oi_{0,0}
\cap D_{k,{\goth p}}^{i+1}({\goth g},\widehat{\bii}) \subset F^{i+1}_{l+1},$$
whence the assertion.

(ii) Let $\varphi $ be in $Z^{i}\cap F^{i}_{l}$. By Lemma~\ref{l2fc6}(ii), for some
$\psi $ in $F^{i-1}_{l}$,
$$ \varphi - \dd \psi \in K^{i}_{l}+F^{i}_{l+1}.$$
By Lemma~\ref{l2fc6}(i), for $i'=i,i+1$, the sum  $K_{l}^{i'} + F^{i'}_{l+1}$ is direct.
Let $\varphi _{1}$ be the component of $\varphi - \dd \psi $ on $K_{l}^{i}$. Then, by
(i), $\dd _{2}\varphi _{1}=0$ since 
$\dd _{2}\varphi _{1}$ is in $K_{l}^{i+1}$ and $\dd F^{i}_{l+1}$ is contained in 
$F^{i+1}_{l+1}$. As a result, by Lemma~\ref{lfc6}, for some $\varphi '_{1}$ in 
$K_{l}^{i-1}$, $\varphi _{1}=\dd_{2}\varphi '_{1}$. Then, by (i),
$$ \psi + \varphi '_{1} \in F^{i-1}_{l} \quad  \text{and} \quad
\varphi - \dd \psi - \dd \varphi '_{1} \in F^{i}_{l+1},$$
whence the assertion.
\end{proof}

\begin{coro}\label{cfc6}
Let $i=\b g{}+1,\ldots,\b g{}+k$.

{\rm (i)} For all nonnegative integer $l$, $Z^{i}$ is contained in $B^{i}+ F^{i}_{l}$.

{\rm (ii)} For some $p$ in $\tk {\k}{{\goth m}}\k[{\goth l}_{*}]$, 
$(1+p)Z^{i}$ is contained in $B^{i}$.
\end{coro}
 
\begin{proof}
(i) By Lemma~\ref{l3fc6}(ii), for $l$ nonnegative integer,
$$ Z^{i}\cap F^{i}_{l} \subset \dd F^{i-1}_{l} + F^{i}_{l+1} .$$ 
Then, by induction on $l$, $Z^{i}$ is contained in $B^{i}+F^{i}_{l}$.

(ii) The natural gradation of $\k[{\goth g}]$ induces a gradation of 
$\tk {\k}{\hat{\an {}{}}}\tk {\k}{\k[\lg l*]}D({\goth g})$. As 
$\widehat{\bii}$ is a graded submodule
of $\tk {\k}{\hat{\an {}{}}}\tk {\k}{\k[\lg l*]}D({\goth g})$ so are
$D_{k}^{\bullet}({\goth g},\widehat{\bii})$,
$D_{k,{\goth p}}^{\bullet}({\goth g},\widehat{\bii})$,
$F^{\bullet}_{l}, \, l=0,1,\ldots$. Then $Z^{i}$ and $B^{i}$ are graded submodules of 
$D_{k,{\goth p}}^{i}({\goth g},\widehat{\bii})$ since the differential of 
$D_{k}^{\bullet}({\goth g},\widehat{\bii})$ is homogeneous of degree $0$ with respect to 
this gradation.  

Let $l$ be a nonnegative integer. Denote by 
$$ \tk {\k}{\hat{\an {}{}}}\k[\lg l*]^{(l)}, \quad 
D^{i,l}_{k}({\goth g},\widehat{\bii}), \quad Z^{i,l}, \quad B^{i,l} $$
the subspaces of degree $l$ of 
$$ \tk {\k}{\hat{\an {}{}}}\k[\lg l*], \quad
D^{i}_{k}({\goth g},\widehat{\bii}), \quad Z^{i}, \quad B^{i} $$
respectively. In particular, these spaces are finitely generated 
$\tk {\k}{\hat{\an {}{}}}\k[{\goth l}_{*}]$-modules. Then, by 
\cite[Ch. 3, Theorem 8.9]{Mat}, for some $p_{l}$ in 
$\tk {\k}{{\goth m}}\k[{\goth l}_{*}]$,
$$ (1+p_{l}) \bigcap _{j\in {\Bbb N}} 
(B^{i,l} + {\goth m}^{j}D^{i,l}_{k}({\goth g},\widehat{\bii})) \subset
B^{i,l} .$$
By (i) and Proposition~\ref{pfc4}(ii),
$$ Z^{i} \subset B^{i} + \hat{J}^{j-n}D_{k}^{i}({\goth g},\widehat{\bii})$$
for $j\geq n$. Then
$$ Z^{i,l} \subset B^{i,l} + 
{\goth m}^{j-l}D^{i,l}_{k}({\goth g},\widehat{\bii})$$
for all integer $j$ bigger than $l+n$ since $J$ is an ideal of $\k[\lg l*]$
generated by elements of positive degree. As a result, 
$$ (1+p_{l}) Z^{i,l} \subset B^{i,l} .$$
Then, for some $p$ in $\tk {\k}{{\goth m}}\k[{\goth l}_{*}]$,
$$ (1+p) Z^{i} \subset B^{i} $$
since $Z^{i}$ is a finitely generated module over 
$\tk {\k}{\hat{\an {}{}}}\k[\lg l*]$.
\end{proof}

\begin{prop}\label{pfc5}
For $k=0,\ldots,n$, $D_{k,{\goth p}}^{\bullet}({\goth g},\widehat{\bii})$ has no
cohomology of degree different from $\b g{}$.
\end{prop}

\begin{proof}
For $i<\b g{}$ or $i>k+\b g{}$, $D_{k,{\goth p}}^{i}({\goth g},\widehat{\bii})=\{0\}$.
By definition, $D_{k,{\goth p}}^{\bullet}({\goth g},\widehat{\bii})$ has no cohomology of
degree $k+\b g{}$. So, it is true for $k=0,1$. Let $k=2,\ldots,n$ and 
$i=\b g{}+1,\ldots,\b g{}+k-1$. For $l$ nonnegative integer, denote by $T_{i,l}$ the
support of $Z^{i,l}/B^{i,l}$ in 
${\mathrm {Spec}}(\tk {\k}{\hat{\an {}{}}}\k[{\goth l}_{*}])$. As 
$Z^{i,l}/B^{i,l}$ is a finitely generated
$\tk {\k}{\hat{\an {}{}}}\k[{\goth l}_{*}]$-module, $T_{i,l}$ is a closed subset of 
${\mathrm {Spec}}(\tk {\k}{\hat{\an {}{}}}\k[{\goth l}_{*}])$. Since ${\goth m}$ is
contained in all maximal ideal of $\tk {\k}{\hat{\an {}{}}}\k[{\goth l}_{*}]$, $T_{i,l}$
does not contain a maximal ideal by Corollary~\ref{cfc6}(ii). Then $T_{i,l}$ is empty and
$Z^{i}=B^{i}$, whence the proposition.
\end{proof}

\subsection{End of the proof of Theorem~\ref{t4int}} \label{fc7}
We can now complete the proof of Theorem~\ref{t4int}. 

We prove the theorem by induction on the dimension of ${\goth g}$. By
Proposition~\ref{ppp}(ii), the theorem is true for $\rg=1$. Suppose $\rg >1$ and
the theorem true for the simple algebras of rank smaller than $\rg$. By
Proposition~\ref{pfc5} and the induction hypothesis, for 
$k=1,\ldots,n$, $D_{k,{\goth p}}^{\bullet}({\goth g},\widehat{\bii})$ has no cohomology
of degree different from $\b g{}$ for all parabolic subalgebra ${\goth p}$ of ${\goth g}$
containing ${\goth b}$. So, by Theorem~\ref{tstp4}, ${\goth g}$ has Property $({\bf P})$.

\begin{appendix}
%\appendix 
\section*{Appendix}

\section{Projective dimension and cohomology} \label{pdc}
Recall in this section classical results. Let $X$ be a Cohen-Macaulay irreducible 
affine algebraic variety and $S$ a closed subset of codimension $p$ of $X$. Let 
$P_{\bullet}$ be a complex of finitely generated projective $\k[X]$-modules whose length 
$l$ is finite and let $\varepsilon $ be an augmentation morphism of $P_{\bullet}$ whose 
image is $R$, whence an augmented complex of $\k[X]$-modules,
$$ \xymatrix{0 \ar[r] & P_{l} \ar[r] &
P_{l-1} \ar[r] & \cdots \ar[r] & P_{0} \ar[r]^{\varepsilon } & R \ar[r] & 0 }.$$
Set:
$$ {\cal P}_{\bullet} := \tk {\k[X]}{\an X{}{}}P_{\bullet}, \quad
{\cal R} := \tk {\k[X]}{\an X{}{}}R, \quad
{\cal K}_{0} := \tk {\k[X]}{\an X{}{}}{{\mathrm {Ker}}\,\varepsilon }$$
and denote by ${\cal K}_{i}$ the kernel of the morphism 
$\xymatrix{{\cal P}_{i}\ar[r] & {\cal P}_{i-1}}$ for $i$ positive integer.

\begin{lemma} \label{lpdc}
Suppose that $S$ contains the support of the homology of the augmented complex  
$P_{\bullet}$. 

{\rm (i)} For all positive integer $i<p-1$ and for all projective 
$\an X{}$-module ${\cal P}$, ${\rm H}^{i}(X\setminus S,{\cal P})$ is equal to $0$.

{\rm (ii)} For all nonnegative integer $j\leq l$ and for all positive integer 
$i<p-j$, the cohomology group 
${\rm H}^{i}(X\setminus S,{\cal K}_{l-j})$ is equal to zero.
\end{lemma}

\begin{proof}
(i) Let $i<p-1$ be a positive integer. Since the functor H$^{i}(X\setminus S,\bullet )$ 
commutes with the direct sum, it suffices to prove
${\mathrm {H}}^{i}(X\setminus S,\an X{})=0$. Since $S$ is a closed subset of $X$, 
we have the relative cohomology long exact sequence
$$ \xymatrix{\cdots \ar[r] & 
{\mathrm H}^{i}_{S}(X,\an X{}) \ar[r] & {\mathrm H}^{i}(X,\an X{})
\ar[r] & {\mathrm H}^{i}(X\setminus S,\an X{}) \ar[r] & 
{\mathrm H}^{i+1}_{S}(X,\an X{}) \ar[r] & \cdots }.$$
Since $X$ is affine, ${\mathrm {H}}^{i}(X,\an X{})$ is equal to zero and 
${\mathrm {H}}^{i}(X\setminus S,\an X{})$ is
isomorphic to ${\mathrm {H}}_{S}^{i+1}(X,\an X{})$. Since $X$ is Cohen-Macaulay, the 
codimension $p$ of $S$ in $X$ is equal to the depth of its ideal of definition in $\k[X]$ 
\cite[Ch. 6, Theorem 17.4]{Mat}. Hence, according to ~\cite[Theorem 3.8]{Gro}, 
${\mathrm {H}}_{S}^{i+1}(X,\an X{})$ and  ${\mathrm {H}}^{i}(X\setminus S,\an X{})$ are
equal to $0$ since $i+1<p$. 

(ii) Let $j$ be a nonnegative integer. Since $S$ contains the support of the homology of 
the complex $P_{\bullet}$, for all nonnegative integer $j$, we have the short exact 
sequence of $\an {X\setminus S}{}$-modules
$$ \xymatrix{0 \ar[r] & 
{\cal K}_{j+1}\left \vert \right. _{X\setminus S} \ar[r] & {\cal P}_{j+1}\left
\vert \right. _{X\setminus S} \ar[r] & 
{\cal K}_{j} \left \vert \right. _{X\setminus S}\ar[r] & 0 }$$
whence the long exact sequence of cohomology
$$ \xymatrix{\cdots \ar[r] & {\mathrm H}^{i}(X\setminus S,{\cal P}_{j+1})
\ar[r] & {\mathrm H}^{i}(X\setminus S,{\cal K}_{j}) \ar[r] & 
{\mathrm H}^{i+1}(X\setminus S,{\cal K}_{j+1}) \ar[r] & 
{\mathrm H}^{i+1}(X\setminus S,{\cal P}_{j+1}) \ar[r] & \cdots }.$$
Then, by (i), for $0<i<p-2$, the cohomology groups 
${\mathrm {H}}^{i}(X\setminus S,{\cal K}_{j})$ and 
${\mathrm {H}}^{i+1}(X\setminus S,{\cal K}_{j+1})$ are isomorphic 
since $P_{j+1}$ is a projective module. Since ${\cal P}_{i}=0$ for $i>l$, 
${\cal K}_{l-1}$ and ${\cal P}_{l}$ have isomorphic restrictions to $X\setminus S$. In 
particular, by (i), for $0<i<p-1$, 
${\mathrm {H}}^{i}(X\setminus S,{\cal K}_{l-1})$ equal zero.
Then, by induction on $j$, for $0<i<p-j$, 
${\mathrm {H}}^{i}(X\setminus S,{\cal K}_{l-j})$ is equal to zero.
\end{proof}

\begin{prop} \label{ppdc}
Let $R'$ be a $\k[X]$-module containing $R$. Suppose that the following conditions 
are satisfied:
\begin{list}{}{}
\item {\rm (1)} $p$ is at least $l+2$,
\item {\rm (2)} $X$ is normal,
\item {\rm (3)} $S$ contains the support of the homology of the augmented complex 
$P_{\bullet}$.
\end{list}

{\rm (i)} The complex $P_{\bullet}$ is a projective resolution of $R$ of length $l$.

{\rm (ii)} Suppose that  $R'$ is torsion free and that $S$ contains the support in $X$ of
$R'/R$. Then $R'=R$.
\end{prop}

\begin{proof}
(i) Let $j$ be a positive integer. We have to prove that 
${\mathrm {H}}^{0}(X,{\cal K}_{j})$ is the image of $P_{j+1}$. By Condition (3), the 
short sequence of $\an {X\setminus S}{}$-modules
$$ \xymatrix{0 \ar[r] & \left. {\cal K}_{j+1} \right \vert _{X\setminus S} \ar[r] & 
{\cal P}_{j+1} \left. \vert \right. _{X\setminus S} \ar[r] & 
\left. {\cal K}_{j} \right \vert _{X\setminus S}\ar[r] & 0 }$$ 
is exact, whence the cohomology long exact sequence
$$ \xymatrix{0 \ar[r] & {\mathrm {H}}^{0}(X\setminus S,{\cal K}_{j+1}) \ar[r] & 
{\mathrm {H}}^{0}(X\setminus S,{\cal P}_{j+1})
\ar[r] & {\mathrm {H}}^{0}(X\setminus S,{\cal K}_{j}) \ar[r] & 
{\mathrm H}^{1}(X\setminus S,{\cal K}_{j+1}) \ar[r] & \cdots }.$$ 
By Lemma~\ref{lpdc}(ii), ${\mathrm {H}}^{1}(X\setminus S,{\cal K}_{j+1})$ equals $0$ since
$1<p-l+j+1$, whence the short exact sequence 
$$ \xymatrix{0 \ar[r] & {\mathrm {H}}^{0}(X\setminus S,{\cal K}_{j+1}) \ar[r] & 
{\mathrm {H}}^{0}(X\setminus S,{\cal P}_{j+1})
\ar[r] & {\mathrm {H}}^{0}(X\setminus S,{\cal K}_{j}) \ar[r] & 0 }.$$ 
As the codimension of $S$ in $X$ is at least $2$ and $X$ is irreducible and
normal, the restriction morphism from $P_{j+1}$ to 
${\mathrm {H}}^{0}(X\setminus S,{\cal P}_{j+1})$ is an isomorphism. Let $\varphi $ be in 
${\mathrm {H}}^{0}(X,{\cal K}_{j})$. Then there exists an element $\psi $ of 
$P_{j+1}$ whose image $\psi '$ in ${\mathrm {H}}^{0}(X,{\cal K}_{j})$ has the same 
restriction to $X\setminus S$ as $\varphi $. Since $P_{j}$ is a projective module and 
$X$ is irreducible, $P_{j}$ is torsion free. Then $\varphi =\psi '$ since 
$\varphi -\psi '$ is a torsion element of $P_{j}$, whence the assertion.

(ii) Set ${\cal R}' := \tk {\k[X]}{\an X{}{}}R'$. Arguing as in (i), since $S$ 
contains the support of $R'/R$ and $1<p-l$, the short sequence
$$ \xymatrix{0 \ar[r] & {\mathrm {H}}^{0}(X\setminus S,{\cal K}_{0}) \ar[r] & 
{\mathrm {H}}^{0}(X\setminus S,{\cal P}_{0})
\ar[r] & {\mathrm {H}}^{0}(X\setminus S,{\cal R}') \ar[r] & 0 }$$ 
is exact. Moreover, the restriction morphism from $P_{0}$ to 
${\mathrm {H}}^{0}(X/S,{\cal P}_{0})$ is an isomorphism since the codimension of $S$ in 
$X$ is at least $2$ and $X$ is irreductible and normal. Let $\varphi $ be in $R'$. 
Then for some $\psi $ in $P_{0}$, $\varphi -\varepsilon (\psi )$ is a torsion element of 
$R'$. So $\varphi =\varepsilon (\psi )$ since $R'$ is torsion free, whence the assertion.
\end{proof}

\begin{coro} \label{cpdc}
Let $C_{\bullet}$ be a homology complex of finitely generated $\k[X]$-modules whose 
length $l$ is finite and positive. For $j=0,\ldots,l$, denote by $Z_{j}$ the space
of cycles of degree $j$ of $C_{\bullet}$. Suppose that the following conditions 
are satisfied:
\begin{list}{}{}
\item {\rm (1)} $S$ contains the support of the homology of the complex
$C_{\bullet}$,
\item {\rm (2)} for all $i$, $C_{i}$ is a submodule of a free module,
\item {\rm (3)} for $i=1,\ldots,l$, $C_{i}$ has projective dimension at most $d$,
\item {\rm (4)} $X$ is normal and $l+d\leq p-1$.
\end{list}
Then $C_{\bullet}$ is acyclic and for $j=0,\ldots,l$, $Z_{j}$ has projective dimension
at most $l+d-j-1$.
\end{coro}

\begin{proof}
Prove by induction on $l-j$ that the complex
$$ \xymatrix{0 \ar[r] & C_{l} \ar[r] & \cdots \ar[r] & C_{j+1}
\ar[r] & Z_{j} \ar[r] & 0 }$$
is acyclic and $Z_{j}$ has projective dimension at most $l+d-j-1$. For 
$j=l$, $Z_{j}$ is equal to zero since $C_{l}$ is torsion free by Condition (2) and 
$Z_{l}$ is a submodule of $C_{l}$, supported by $S$ by Condition (1). Suppose 
$j\leq l-1$ and the statement true for $j+1$. By Condition (3), $C_{j+1}$ has
a projective resolution $P_{\bullet}$ whose length is at most $d$ and whose terms are 
finitely generated. By induction hypothesis, $Z_{j+1}$ has a projective resolution 
$Q_{\bullet}$ whose length is at most $l+d-j-2$ and whose terms are finitely generated, 
whence an augmented complex $R_{\bullet}$ of projective modules whose length is 
$l+d-j-1$,
$$ \xymatrix{0 \ar[r] & Q_{l+d-j-2} \oplus P_{l+d-j-1} \ar[r] & 
\cdots  \ar[r] & Q_{0} \oplus P_{1}
\ar[r] & P_{0} \ar[r] &  Z_{j} \ar[r] & 0 }.$$
Denoting by $\dd $ the differentials of $Q_{\bullet}$ and $P_{\bullet}$, the restriction 
to $Q_{i}\oplus P_{i+1}$ of the differential of $R_{\bullet}$ is the map
$$(x,y) \mapsto (\dd x,\dd y +(-1)^{i} \delta (x)) ,$$ 
with $\delta $ the map which results from the injection of $Z_{j+1}$ into 
$C_{j+1}$. Since $P_{\bullet}$ and $Q_{\bullet}$ are projective resolutions, the complex 
$R_{\bullet}$ is a complex of projective modules having no homology 
of positive degree. Hence the support of the homology of the augmented complex 
$R_{\bullet}$ is contained in $S$ by Condition (1). Then, by Proposition~\ref{ppdc} and 
Condition (4), $R_{\bullet}$ is a projective resolution of $Z_{j}$ of length 
$l+d-j-1$ since $Z_{j}$ is a submodule of a free module by Condition (2), whence the 
corollary since $Z_{0}=C_{0}$ by definition.
\end{proof}

\begin{coro}\label{c2pdc}
Let
$$ \xymatrix{0 \ar[r] & E_{-1}\ar[r] & E_{0} \ar[r] & \cdots 
\ar[r] & E_{l} \ar[r] & 0 }$$
be a complex of finitely generated $\k[X]$-modules. Suppose that the following conditions are 
satisfied:
\begin{itemize}
\item [{\rm (1)}] $E_{-1}$ is projective and for $i=0,\ldots,l-1$, $E_{i}$ has projective
dimension at most $i$,
\item [{\rm (2)}] $S$ contains the support of the cohomology of this complex, 
\item [{\rm (3)}] for $i=0,\ldots,l$, $E_{i}$ is a submodule of a free module,
\item [{\rm (4)}] $X$ is normal and $p\geq l+2$.
\end{itemize}
Then the complex is acyclic and $E_{l}$ has projective dimension at most $l$.
\end{coro}

\begin{proof}
Prove the corollary by induction on $l$. For $l=0$, by Conditions (2), (3), (4), 
the arrow $\xymatrix{E_{-1}\ar[r] & E_{0}}$ is an isomorphism. Suppose the corollary true
for the integers smaller than $l$. Let $Z_{l-1}$ be the kernel of the arrow
$\xymatrix{E_{l-1} \ar[r] & E_{l}}$, whence the two complexes
$$ \xymatrix{0 \ar[r] & E_{-1}\ar[r] & E_{0} \ar[r] & \cdots 
\ar[r] & E_{l-2} \ar[r] & Z_{l-1} \ar[r] & 0 }$$
$$ \xymatrix{ 0 \longrightarrow Z_{l-1} \ar[r] & E_{l-1} \ar[r] & E_{l} \ar[r] & 0}.$$ 
By Condition (2), the support of the cohomology of these two complexes is contained in 
$S$. Then, by induction hypothesis, the first complex is acyclic and $Z_{l-1}$ has 
projective dimension at most $l-1$. As a result, arguing as in the proof of 
Corollary~\ref{cpdc}, we have a complex of $\k[X]$-modules
$$ \xymatrix{0 \ar[r] & P_{l} \ar[r] &
P_{l-1} \ar[r] & \cdots \ar[r] & P_{0} \ar[r] & E_{l} \ar[r] & 0 }$$
such that $\poi P0{,\ldots,}{l}{}{}{}$ are projective, the image of $P_{0}$ in 
$E_{l}$ is the image of the arrow $\xymatrix{E_{l-1} \ar[r] & E_{l}}$ and the support
of its homology is contained in $S$. Then, by Condition (4) and 
Proposition~\ref{ppdc}, it is acyclic so that $E_{l}$ has projective dimension at most $l$
and the complex
$$ \xymatrix{0 \ar[r] & E_{-1}\ar[r] & E_{0} \ar[r] & \cdots \ar[r] & E_{l} \ar[r] & 0 }$$
is acyclic. 
\end{proof}

Let
$$ \xymatrix{0 \ar[r] & M_{0} \ar[r] & M_{1} \ar[r] & M_{2} \ar[r] & 0 }$$ 
be a short exact sequence of $\k[X]$-modules.

\begin{lemma}\label{l2pdc}
Suppose that for $i=0,1,2$, $M_{i}$ has a finite projective dimension $d_{i}$. Then we 
have the inequalities
$$ d_{2} \leq \sup \{d_{0}+1,d_{1}\} \quad  \text{and} \quad
d_{0} \leq \sup \{d_{2}-1,d_{1}\} .$$
\end{lemma}

\begin{proof}
Let $N$ be a ${\k}[X]$-module. We have to prove 
Ext$^{j}(M_{2},N) = 0$ for $j$ bigger than $\sup \{d_{0}+1,d_{1}\}$ and 
Ext$^{j}(M_{1},N) = 0$ for $j$ bigger than $\sup \{d_{2}-1,d_{1}\}$. From the short 
exact sequence, we deduce the long exact sequence
$$ \xymatrix{\cdots \ar[r] & {\rm Ext}^{j}(M_{1},N) \ar[r] & {\rm Ext}^{j}(M_{0},N)
\ar[r] & {\rm Ext}^{j+1}(M_{2},N)\ar[r] & 
{\mathrm {Ext}}^{j+1}(M_{1},N) \ar[r] & \cdots  }.$$  
For $j+1>\sup \{d_{0}+1,d_{1}\}$, ${\rm Ext}^{j+1}(M_{1},N)=0$ and 
${\rm Ext}^{j}(M_{0},N)=0$, whence ${\rm Ext}^{j+1}(M_{2},N)=0$. For 
$j>\sup \{d_{2}-1,d_{1}\}$, ${\rm Ext}^{j}(M_{1},N)=0$ and 
${\rm Ext}^{j+1}(M_{2},N)=0$, whence ${\rm Ext}^{j}(M_{0},N)=0$. 
\end{proof}

\section{Some remarks about representations} \label{rep}
In this section, ${\goth g}$ is a semisimple Lie algebra, ${\goth p}$ is a parabolic
subalgebra of ${\goth g}$, containing ${\goth b}$, ${\goth l}$ is the reductive factor
of ${\goth p}$, containing ${\goth h}$, and ${\goth d}$ is the derived algebra of
${\goth l}$. Let ${\cal R}_{{\goth l}}$ the set of roots $\alpha $ such that
${\goth g}_{\alpha }$ is contained in ${\goth l}$ and ${\cal R}_{{\goth l},+}$ the
intersection of ${\cal R}_{{\goth l}}$ and ${\cal R}_{+}$. Denote by ${\cal P}_{\#}$ the
subset of elements of ${\cal P}({\cal R})$ whose restriction to
${\goth h}\cap {\goth d}$ is a dominant weight of the root system ${\cal R}_{{\goth l}}$
with respect to the positive root system ${\cal R}_{{\goth l},+}$.
 
Let $M$ be a rational ${\goth g}$-module. For $\lambda $ in ${\cal P}_{+}({\cal R})$,
denote by $M_{\lambda }$ the isotypic component of type $V_{\lambda }$ of the
${\goth g}$-module $M$. Let ${\cal P}_{M}$ be the subset of dominant weights $\lambda $
such that $M_{\lambda } \neq 0$.

\begin{lemma}\label{lrep}
The space $M$ is the direct sum of $M_{\lambda }, \; \lambda \in {\cal P}_{M}$.
\end{lemma}

\begin{proof}
As $M$ is a rational ${\goth g}$-module, $M$ is a union of ${\goth g}$-modules of
finite dimension. In particular, all simple ${\goth g}$-module contained in $M$ has
finite dimension. Hence $M_{\lambda }, \; \lambda \in {\cal P}_{M}$ is the set of
isotypic components of $M$. Moreover, $M$ is the direct sum of
$M_{\lambda }, \, \lambda \in {\cal P}_{M}$.
\end{proof}

For $\lambda $ in ${\cal P}_{M}$, denote by $\varpi _{\lambda }$ the canonical projection
$\xymatrix{M \ar[r] & M_{\lambda }}$. Let $N$ be a ${\goth g}$-submodule of $M$. For
$\lambda $ in ${\cal P}_{M}$, denote by $M_{\lambda ,+}$ the subspace of highest weight
vectors of $M_{\lambda }$. For the trivial action of ${\goth g}$ on $M_{\lambda ,+}$,
$\tk {\k}{V_{\lambda }}M_{\lambda ,+}$ is a ${\goth g}$-module. Set
$N_{\lambda ,+} := N\cap M_{\lambda ,+}$ and for $v$ in $M_{\lambda ,+}$, denote by
$M_{v}$ the ${\goth g}$-submodule of $M_{\lambda }$ generated by $v$.  

\begin{lemma}\label{l2rep}
 Let $\lambda $ be in ${\cal P}_{M}$.

{\rm (i)} The ${\goth g}$-module $N$ is the direct sum of
$N\cap M_{\gamma }, \, \gamma  \in {\cal P}_{M}$.  

{\rm (ii)} For $v$ in $M_{\lambda ,+}$, the ${\goth g}$-modules $V_{\lambda }$ and
$M_{v}$ are isomorphic.

{\rm (iii)} There exists a basis $v_{i}, \, i \in I_{\lambda }$ of $M_{\lambda ,+}$
satisfying the following condition: for some subset $I_{N,\lambda }$ of $I_{\lambda }$,
$v_{i}, \, i \in I_{N,\lambda }$ is a basis of $N_{\lambda ,+}$.

{\rm (iv)} For $i$ in $I_{\lambda }$, denote by $\tau _{\lambda ,i}$ an isomorphism of
${\goth g}$-modules $\xymatrix{V_{\lambda } \ar[r] & M_{v_{i}}}$. Then the linear map
$$ \xymatrix{ \tk {\k}{V_{\lambda }}M_{\lambda ,+} \ar[rr]^{\tau _{\lambda }} &&
M_{\lambda }}, \qquad v\tens v_{i} \longmapsto \tau _{\lambda ,i}(v) $$
is an isomorphism of ${\goth g}$-modules such that
$\tau _{\lambda }(\tk {\k}{V_{\lambda }}N_{\lambda ,+}) = N\cap M_{\lambda }$.
\end{lemma}

\begin{proof}
(i) As $N$ is a ${\goth g}$-submodule of $M$, it is rational. So, by
Lemma~\ref{lrep}(i), $N$ is the direct sum of its isotypic components, whence the
assertion since an isotypic component of $N$ is contained in the isotypic component
of $M$ of the same type.

(ii) As $v$ is in $M_{\lambda ,+}$, $M_{v}$ is a module of highest weight $\lambda $
and the space of highest weight vectors in $M_{v}$ is generated by $v$. Hence $M_{v}$ is 
simple and isomorphic to $V_{\lambda }$.

(iii) is straightforward. Moreover, if $N\cap M_{\lambda }=\{0\}$ then
$I_{N,\lambda }$ is empty.

(iv) By (ii), the isomorphisms $\tau _{\lambda ,i}$ does exist. As
$v_{i}, \, i \in I_{\lambda }$ is a basis of $M_{\lambda ,+}$, $M_{\lambda }$
is the direct sum of the subspaces $M_{v_{i}}, \, i \in I_{\lambda }$. Hence 
$\tau _{\lambda }$ is an isomorphism of ${\goth g}$-modules. Moreover, for $i$ in
$I_{\lambda }$, $\tau _{\lambda }(V_{\lambda }\tens v_{i})$ is contained in
$N$ if and only if $i$ is in $I_{N,\lambda }$, whence the assertion since $N$ is
a ${\goth g}$-module.
\end{proof}

Let $M'$ be a ${\goth l}$-submodule of $M$. For $\mu $ in ${\cal P}_{\#}$, denote by
$V'_{\mu }$ a simple ${\goth l}$-module of highest weight $\mu $ and $M'_{\mu }$ the
isotypic component of type $V'_{\mu }$ of $M'$. Denote by ${\cal P}_{M'}$ the subset of
elements $\mu $ of ${\cal P}_{\#}$ such that $M'_{\mu }\neq \{0\}$ and ${\cal P}_{M,M'}$
the subset of elements $(\lambda ,\mu )$ of ${\cal P}_{M}\times {\cal P}_{M'}$ such that 
$\varpi _{\lambda }(M'_{\mu }) \neq \{0\}$.

\begin{lemma}\label{l3rep}
{\rm (i)} The space $M'$ is the direct sum of $M'_{\mu }, \, \mu  \in {\cal P}_{M'}$.

{\rm (ii)} For $(\lambda ,\mu )$ in ${\cal P}_{M,M'}$, $V'_{\mu }$ is isomorphic to a
${\goth l}$-submodule of $V_{\lambda }$.
\end{lemma}

\begin{proof}
(i) As $M'$ is a ${\goth l}$-submodule of the rational ${\goth g}$-module $M$, $M'$ is a
rational ${\goth l}$-module, whence the assertion by Lemma~\ref{lrep}(i).

(ii) Let $(\lambda ,\mu )$ in ${\cal P}_{M,M'}$ and $V_{0}$ a simple ${\goth l}$-module
contained in $\varpi _{\lambda }(M'_{\mu })$. According to Lemma~\ref{l2rep},(ii) and
(iv), 
$$ M_{\lambda } = \bigoplus _{i\in I_{\lambda }} M_{v_{i}} .$$ 
For $i$ in $I_{\lambda }$, denote by $\pi _{i}$ the projection 
$$ \xymatrix{ M_{\lambda } \ar[rr]^-{\pi _{i}} && M_{v_{i}}} $$
corresponding to this decomposition. For some $i$, the restriction of $\pi _{i}$ to
$V_{0}$ is different from $0$. As $V_{0}$ is a simple ${\goth l}$-module, this
restriction is an embedding of $V_{0}$ into $M_{v_{i}}$, whence the assertion since
$M_{v_{i}}$ is isomorphic to $V_{\lambda }$.
\end{proof}

For $\lambda $ in ${\cal P}_{M}$, denote by $V_{\lambda }^{{\goth l}}$ the subspace 
of elements of $V_{\lambda }$, annihilated by ${\goth u}\cap {\goth l}$, and for 
$(\lambda ,\mu )$ in ${\cal P}_{M,M'}$, let $V_{\lambda ,\mu }^{{\goth l}}$ be the 
subspace of weight $\mu $ of $V_{\lambda }^{{\goth l}}$.

\begin{lemma}\label{l4rep}
Let $(\lambda ,\mu )$ be in ${\cal P}_{M,M'}$.

{\rm (i)} For the trivial action of ${\goth l}$ in $V^{{\goth l}}_{\lambda ,\mu }$, there
exists an isomorphism of ${\goth l}$-modules 
$$\xymatrix{ \tk {\k}{V'_{\mu }}V_{\lambda ,\mu }^{{\goth l}} 
\ar[rr]^-{\tau _{\lambda ,\mu }} && \e Ul.V_{\lambda ,\mu }^{{\goth l}}} .$$

{\rm (ii)} For a well defined subspace $E_{\lambda ,\mu }$ of 
$\tk {\k}{V_{\lambda ,\mu }^{{\goth l}}}M_{\lambda ,+}$, 
$$ \varpi _{\lambda }(M'_{\mu }) = 
\tau _{\lambda }\rond (\tau _{\lambda ,\mu }\tens {\mathrm {id}}_{M_{\lambda ,+}})
(\tk {\k}{V'_{\mu }}E_{\lambda ,\mu }) .$$

{\rm (iii)} Let $E_{N,\lambda ,\mu }$ be the intersection of 
$E_{\lambda ,\mu }$ and $\tk {\k}{V^{{\goth l}}_{\lambda ,\mu }}N_{\lambda ,+}$. Then
$$ \varpi _{\lambda }(N\cap M'_{\mu }) = 
\tau _{\lambda }\rond (\tau _{\lambda ,\mu }\tens {\mathrm {id}}_{M_{\lambda ,+}})
(\tk {\k}{V'_{\mu }}E_{N,\lambda ,\mu }) .$$
\end{lemma}

\begin{proof}
(i) Let $\poi w1{,\ldots,}{m}{}{}{}$ be a basis of $V_{\lambda ,\mu }^{{\goth l}}$. 
For $i=1,\ldots,m$, denote by $V'_{i}$ the ${\goth l}$-submodule of $V_{\lambda }$
generated by $w_{i}$. As $w_{i}$ is weight vector of weight $\mu $ of
$V_{\lambda ,\mu }^{{\goth l}}$, $V'_{i}$ is a module of highest weight $\mu $ and the
space of highest weight vectors in $V'_{i}$ is generated by $w_{i}$ so that $V'_{i}$ is a
simple module isomorphic to $V'_{\mu }$. Moreover, $\e Ul.V_{\lambda ,\mu }^{{\goth l}}$
is the direct sum of $V'_{i}, \, i=1,\ldots,m$ since $\poi w1{,\ldots,}{m}{}{}{}$ is a
basis of $V_{\lambda ,\mu }^{{\goth l}}$, whence an isomorphism 
$$\xymatrix{ \tk {\k}{V'_{\mu }}V_{\lambda ,\mu }^{{\goth l}} 
\ar[rr]^-{\tau _{\lambda ,\mu }} && \e Ul.V_{\lambda ,\mu }^{{\goth l}}} .$$

(ii) For $v$ in $\tau _{\lambda }^{-1}(\varpi _{\lambda }(M'_{\mu }))$, $v$ has an
expansion
$$ v = \sum_{i\in I_{\lambda }} v'_{i}\tens v_{i} $$
with $v'_{i}, \, i \in I_{\lambda }$ in $V_{\lambda }$. As $\tau _{\lambda }(v)$ is 
in $\varpi _{\lambda }(M'_{\mu })$, for $i$ in $I_{\lambda }$, $u.v'_{i}$ is in 
$V_{\lambda ,\mu }^{{\goth l}}$ for some $u$ in $\es U{{\goth u}\cap {\goth l}}$. As a
result, $\tau _{\lambda }^{-1}(\varpi _{\lambda }(M'_{\mu }))$ is a subspace of
$\tk {\k}{\e Ul.V_{\lambda ,\mu }^{{\goth l}}}M_{\lambda ,+}$, whence the assertion by
(i).

(iii) Let $v$ be in $E_{\lambda ,\mu }$. By Lemma~\ref{l2rep}(ii),
$\tau _{\lambda }\rond (\tau _{\lambda ,\mu }\tens {\mathrm {id}}_{M_{\lambda ,+}})(v)$ is
in $N$ if and only if $\tau _{\lambda ,\mu }\tens {\mathrm {id}}_{M_{\lambda ,+}}(v)$
is in $\tk {\k}{V_{\lambda }}N_{\lambda ,+}$. Then, by (ii),
$\tau _{\lambda }\rond (\tau _{\lambda ,\mu }\tens {\mathrm {id}}_{M_{\lambda ,+}})(v)$ is
in $\varpi _{\lambda }(N\cap M'_{\mu })$ if and only if $v$ is in
$E_{N,\lambda ,\mu }$.
\end{proof}

For $\lambda $ in ${\cal P}_{M}$, let $\theta _{\lambda }$ be the linear map
$$ \xymatrix{ \tk {\k}{V_{\lambda }^{*}}V_{\lambda } \ar[rr]^-{\theta _{\lambda }}
\ar[rr] && \k}, \quad v'\tens v \longmapsto \dv {v'}v$$
given by the duality. 

\begin{coro}\label{crep}
Suppose that $M'$ generates the ${\goth g}$-module $M$. Let $\lambda $ be in
${\cal P}_{M}$.  Then 
$$ M_{\lambda ,+} = \theta _{\lambda }\tens {\mathrm {id}}_{M_{\lambda ,+}}
(\bigoplus _{\mu \in {\cal P}_{M,M',\lambda }}
\tk {\k}{V_{\lambda }^{*}}E_{\lambda ,\mu })   \quad  \text{with} \quad
{\cal P}_{M,M',\lambda } := \{\mu \in {\cal P}_{M'} \; \vert \;
(\lambda ,\mu ) \in {\cal P}_{M,M'}\} $$
and 
$$ N_{\lambda ,+} = \theta _{\lambda }\tens {\mathrm {id}}_{M_{\lambda ,+}}
(\bigoplus _{\mu \in {\cal P}_{M,M',\lambda }}
\tk {\k}{V_{\lambda }^{*}}E_{N,\lambda ,\mu })  . $$
\end{coro}

\begin{proof}
Since $M'$ generates the ${\goth g}$-module $M$, $\varpi _{\lambda }(M')$ generates the
${\goth g}$-module $M_{\lambda }$ so that $M_{\lambda }$ is the ${\goth g}$-module
generated by
$$ \bigoplus _{\mu \in {\cal P}_{M,M',\lambda }} \varpi _{\lambda }(M'_{\mu }) .$$
As a result, the ${\goth g}$-module $\tk {\k}{V_{\lambda }}M_{\lambda ,+}$ is generated
by 
$$ \tau _{\lambda }^{-1}(\bigoplus _{\mu \in {\cal P}_{M,M',\lambda }}
\varpi _{\lambda }(M'_{\mu }))$$
since for the trivial action of ${\goth g}$ in $M_{\lambda ,+}$, $\tau _{\lambda }$ is
an isomorphism  of ${\goth g}$-modules. Then the ${\goth g}$-module
$\tk {\k}{V_{\lambda }^{*}}\tk {\k}{V_{\lambda }}M_{\lambda ,+}$ is generated by
$$ (\bigoplus _{\mu \in {\cal P}_{M,M',\lambda }}
\tk {\k}{V_{\lambda }^{*}}\tau _{\lambda }^{-1}(\varpi _{\lambda }(M'_{\mu })))$$
and 
$$ \theta _{\lambda }\tens {\mathrm {id}}_{M_{\lambda ,+}}
(\bigoplus _{\mu \in {\cal P}_{M,M',\lambda }}
\tk {\k}{V_{\lambda }^{*}}\tau _{\lambda }^{-1}(\varpi _{\lambda }(M'_{\mu }))) = 
M_{\lambda ,+} $$
since $\theta _{\lambda }\tens {\mathrm {id}}_{M_{\lambda ,+}}$ is 
a morphism of ${\goth g}$-modules. As a result, by Lemma~\ref{l4rep}(ii), 
$$ \theta _{\lambda }\tens {\mathrm {id}}_{M_{\lambda ,+}}
(\bigoplus _{\mu \in {\cal P}_{M,M',\lambda }} 
\tk {\k}{V_{\lambda }^{*}}E_{\lambda ,\mu }) = M_{\lambda ,+} ,$$
and by Lemma~\ref{l4rep}(iii),
$$ \theta _{\lambda }\tens {\mathrm {id}}_{M_{\lambda ,+}}
(\bigoplus _{\mu \in {\cal P}_{M,M',\lambda }} 
\tk {\k}{V_{\lambda }^{*}}E_{N,\lambda ,\mu }) = N_{\lambda ,+} ,$$
whence the assertion.
\end{proof}

\begin{prop}\label{prep}
Suppose that $M'$ generates the ${\goth g}$-module $M$. Then $N\cap M'$ generates the
${\goth g}$-module $N$.
\end{prop}

\begin{proof}
Let $\lambda $ be in ${\cal P}_{M}$. Denote by $\widetilde{N}$ the ${\goth g}$-submodule
of $M$ generated by $N\cap M'$. By Lemma~\ref{l3rep}(i),
$$ \varpi _{\lambda }(\widetilde{N}\cap M'_{\mu }) = \varpi _{\lambda }(N\cap M'_{\mu })$$
for all $\mu $ in ${\cal P}_{M'}$ such that $(\lambda ,\mu )$ is in ${\cal P}_{M,M'}$.
As a result, by Lemma~\ref{l4rep}(iii) and Corollary~\ref{crep}, $N_{\lambda ,+}$ is
contained in $\widetilde{N}$. So, by Lemma~\ref{l2rep}(iv),
$\widetilde{N}\cap M_{\lambda }=N\cap M_{\lambda }$, whence $\widetilde{N}=N$ by
Lemma~\ref{l2rep}(i).
\end{proof}

\end{appendix}

\section*{Tables of Notations}\label{tnt}

\begin{tabular}{ll}
The Lie algebras \\

${\goth g}$ & a reductive Lie algebra \\
${\goth b}$ & a Borel subalgebra \\
${\goth h}$ & a Cartan subalgebra contained in ${\goth b}$ \\
${\goth u}$ & the nilradical of ${\goth b}$ \\
${\goth u}_{-}$ & the nilpotent radical of the Borel subalgebra of ${\goth g}$, opposite
to ${\goth b}$\\
${\goth p}$ & a parabolic subalgebra containing ${\goth b}$ \\
${\goth p}_{-}$ & the parabolic subalgebra opposite to ${\goth p}$ \\
${\goth p}_{\u}$ & the nilradical of ${\goth p}$ \\
${\goth p}_{-,\u}$ & the nilradical of ${\goth p}_{-}$ \\
${\goth l}$ & the reductive factor of ${\goth p}$ containing ${\goth h}$ \\
${\goth z}$ & ${\goth z}$ the center of ${\goth l}$ \\
${\goth d}$ & the derived algebra of ${\goth l}$ \\
${\goth d}_{1},\ldots,{\goth d}_{\n}$ & the simpe factors of ${\goth d}$ \\

\hline

Groups and roots \\
$G$ & the adjoint group of ${\goth g}$ \\
$L$ & the centralizer of ${\goth z}$ in $G$ \\
$H$ & the centralizer of ${\goth h}$ in $G$ \\
${\cal R}$ & the root system of ${\goth h}$ in ${\goth g}$ \\
$W_{{\cal R}}$ & the Weyl group of ${\cal R}$ \\
${\cal R}_{+}$ & the positive root system of ${\cal R}$ defined by ${\goth b}$ \\
$\poi {\beta }1{,\ldots,}{\ell}{}{}{}$ & the basis of ${\cal R}_{+}$ \\
${\goth g}_{\alpha }$ & the weight subspace of weight $\alpha $ for $\alpha $ in
${\cal R}$\\
$x_{\alpha }$ & a generator of ${\goth g}_{\alpha }$ \\
$H_{\alpha }$ & the coroot of $\alpha $ \\
${\cal R}_{{\goth l}}$ & the root system of ${\goth h}$ in ${\goth l}$ \\

\end{tabular}

\begin{tabular}{ll}

The numbers \\
$\ell$ & the rank of ${\goth g}$ \\
${\mathrm {b}}_{{\goth g}}$ & the dimension of ${\goth b}$ \\
$n$ & the dimension of ${\goth u}$ \\
$d$ & the dimension of ${\goth p}_{\u}$ \\
$d_{0}$ & the dimension of ${\goth z}$ \\
$\ell_{{\goth l}}$ & the rank of ${\goth l}$ \\
$\n$  & the number of simple factors of ${\goth d}$ \\
$\ell _{i}$ & the rank of ${\goth d}_{i}$ \\
$\b di$   & the dimension of ${\goth d}_{i}\cap {\goth b}$ \\
$n_{i}$ & the dimension of ${\goth d}_{i}\cap {\goth u}$ \\
$\poi m1{,\ldots,}{\ell}{}{}{}$ & the increasing sequence of the exponents of ${\cal R}$\\
$d_{i}$ & $d_{i} := m_{i}+1$ \\
$\poie m1{,\ldots,}{\ell}{}{}{}{\prime}{\prime}$ & the increasing sequence of the
exponents of ${\cal R}_{{\goth l}}$\\
$d'_{i}$ & $d'_{i} := m'_{i}+1$ \\  

\hline

The sets \\
${\Bbb N}$ & the set of nonnegative integers \\
$\prec$ & the lexicographic order on ${\Bbb N}^{n}$ induced by the usual order of
${\Bbb N}$ \\
$\vert n \vert $ & $\poi n1{+\cdots +}{j}{}{}{}$ for
$n=(\poi n1{,\ldots,}{j}{}{}{}) \in {\Bbb N}^{j}$ \\
${\Bbb N}^{k}_{j}$ & $\{n\in {\Bbb N}^{k} \, \vert \, \vert n \vert = j\}$ \\
$I_{0}$ & $\{(i,m)\in \{1,\ldots,\ell\}\times {\Bbb N} \, \vert \, 0\leq m \leq m_{i} \}$
\\
$I_{*,0}$ & $I_{0}\cap ({\Bbb N}\times ({\Bbb N}\setminus \{0\})$ \\
$I'_{0}$ & $\{(i,m)\in \{1,\ldots,\ell\}\times {\Bbb N} \, \vert \,
0\leq m \leq m'_{i} \}$ \\
${\Bbb I}$ & \{$(\poi i{-1}{,\ldots,}{\n}{}{}{}) \in {\Bbb N}^{\n+2} \, \vert \,
i_{1}\leq n_{1},\ldots,i_{\n}\leq n_{\n}\}$ \\
${\Bbb I}'$ & ${\Bbb I}\cap \{0\}\times {\Bbb N}^{\n+1}$ \\
${\Bbb I}''$ & ${\Bbb I}\cap \{0,0\}\times {\Bbb N}^{\n}$ \\
${\Bbb I}_{k}$ & ${\Bbb I}\cap {\Bbb N}^{\n+2}_{k}$ \\
${\Bbb I}'_{k}$ & ${\Bbb I}'\cap {\Bbb N}^{\n+1}_{k}$ \\
${\Bbb I}''_{k}$ & ${\Bbb I}''\cap {\Bbb N}^{\n}_{k}$ \\
${\goth I}_{j}$ & $\{ 1\leq \poi i1{<\cdots <}{j}{}{}{}\leq n \}$ \, $(j=1,\ldots,n)$\\
${\goth I}$ & $\{0\} \cup {\goth I}_{1} \cup \cdots \cup {\goth I}_{n}$ \\
${\goth I}_{+}$ & $\{\iota \in {\goth I} \, \vert \{\iota \} \subset \{1,\ldots,d\}\}$ \\
${\goth J}_{*}$ & $\{(i,m) \in {\Bbb N}^{2} \, \vert \, 
1\leq i \leq \n, 1\leq m \leq n_{i}\}\cup \{1,\ldots,d\}$ \\
${\goth J}_{k}$ & $\{\poi {j}1{<\cdots <}{k}{}{}{} \, \vert \,
\{\poi {j}1{,\ldots ,}{k}{}{}{}\} \subset {\goth J}_{*}\}$ \\
${\goth J}$ & $\{0\} \cup {\goth J}_{1} \cup \cdots \cup {\goth J}_{n}$ \\
${\Bbb N}_{{\goth p}}$ & ${\Bbb N}^{r}\times {\goth J}$\\
${\Bbb N}_{{\goth p},+}$ & 
${\Bbb N}^{2n+\rg}\times {\Bbb N}^{r}\times {\goth J}$\\

\end{tabular}

\begin{tabular}{ll}

Algebras, ideals and varieties\\
$\es SV$ & the symmetric algebra of the vector space $V$ \\
$\sy iV$ & the component of degree $i$ of $\es SV$ \\
$\ex {}V$ & the exterior algebra of the vector space $V$  \\
$\ex iV$  & the component of degree $i$ of $\ex {}V$ \\
Gr$_{k}(V)$ & the grassmannian of all $k$-dimensional subspaces of $V$ \\
$\e Sg^{G}$ & the algebra of elements of $\e Sg$ invariant under $G$ \\
$\e Sl^{L}$ & the algebra of elements of $\e Sl$ invariant under $L$ \\
$\e Ul$ & the enveloping algebra of ${\goth l}$ \\
$\an {}{}$ & the local ring of $G$ at the identity \\
${\goth m}$ & the maximal ideal of $\an {}{}$ \\
$\hat{\an {}{}}$ & the completion of $\an {}{}$ for the ${\goth m}$-adic toplogy \\
${\goth l}_{*}$ & $\{x \in {\goth l} \, \vert \,
\tr \ad_{{\goth g}/{\goth l}} x \neq 0\}$ \\
$J$ & the ideal of $\k[{\goth l}_{*}\times {\goth g}]$ generated by
$1\tens {\goth p}_{\u}$ \\
$\hat{J}$ & the ideal of
$\tk {\k}{\hat{\an {}{}}}{\k[\lg l*]}$ generated by ${\goth m}\tens 1$ and
$1\tens J$ \\

\hline

Maps, Modules\\
$\poi p1{,\ldots,}{\ell}{}{}{}$  &  $\k[\poi p1{,\ldots,}{\rg}{}{}{}] =\e Sg^{G}$ and
$p_{i}$ is a homogeneous element of degree $d_{i}$ \\
$p_{i}^{(m)}$ & the $2$-polarisation of $p_{i}$ of bidegree $(d_{i}-m,m)$ for
$(i,m) \in I_{0}$ \\
$\varepsilon _{i}$ & the element of $\tk {\k}{\e Sg}{\goth g}$ identified with
$p_{i}^{(1)}$ for $i=1,\ldots,\ell$ \\
$\varepsilon _{i}^{(m)}$ & the $2$-polarization of $\varepsilon _{i}$ of bidegree
$(m_{i}-m,m)$ for $(i,m)$ in $I_{0}$\\
$\bi g{}$ & $\bigoplus _{(i,m) \in I_{0}} \sgg g{} \varepsilon _{i}^{(m)}$ \\
$\bi l{}$ & the analogous of $\bi g{}$ for ${\goth l}$ \\
$\bk g{}$ & $\bigoplus _{(i,m) \in I_{0}} \sgg g{}
[\varepsilon _{1},\varepsilon _{i}^{(m)}]$ for ${\goth g}$ simple \\
$\varepsilon $ & the product of $\varepsilon _{i}^{(m)}, \, (i,m) \in I_{0}$ in
$\ex {}{\bi g{}}$ \\
$\varepsilon _{0}$ & the product of $\varepsilon _{i}^{(m)}, \, (i,m) \in I'_{0}$ in
$\ex {}{\bi g{}}$ \\
$\varepsilon _{0,0}$ & the product of $\poi {\varepsilon }1{,\ldots,}{\rg}{}{}{}$ in
$\tk {\k}{\k[{\goth l}_{*}]}\ex {}{{\goth l}}$\\
$\hat{\varepsilon }$ & the map $(g,x,y)\mapsto \varepsilon (g(x),g(y))$ from
$G\times \lg l*$ to $\ex {\b g{}}{{\goth g}}$ \\
$\bii {}$ & the restriction of $\bi g{}$ to ${\goth l}_{*}\times {\goth g}$ \\
$\oi$ & the restriction of $\bi g{}$ to ${\goth l}_{*}\times {\goth p}$ \\
$\oi _{0}$ & the submodule of $\oi$ generated
by $\varepsilon _{i}^{(m)}, \, (i,m) \in I'_{0}$\\
$\oi _{00}$ & the submodule of $\oi$ generated
by $\poi {\varepsilon }1{,\ldots,}{\rg}{}{}{}$\\
$\oi _{{\goth l}}$ & the restriction of $\bi l{}$ to ${\goth l}_{*}\times {\goth l}$ \\
$\widetilde{\bii}$ & the $\tk {\k}{\an {}{}}{\k[\lg l*]}$-submodule of
$\tk {\k}{\an {}{}}\tk {\k}{\k[\lg l*]}{\goth g}$ generated by $G.\bii$ \\
$\widehat{\bii}$ & the $\tk {\k}{\han {}{}}{\k[\lg l*]}$-module
$\tk {\an {}{}}{\han {}{}}\widetilde{\bii}$\\

\end{tabular}


\begin{thebibliography}{DK00}

\bibitem[Bol91]{Bol}
A.V.~Bolsinov, {\it Commutative families of functions related to consistent
Poisson brackets}, Acta Applicandae Mathematicae, {\bf 24} (1991),
{\bf n$^{\circ}$1}, p. 253--274.

\bibitem[Bou02]{Bou}
N.~Bourbaki, {\it Lie groups and {L}ie algebras. {C}hapters 4--6. Translated
from the 1968 French original by Andrew Pressley}, Springer-Verlag, Berlin (2002).

\bibitem[Bou98]{Bou1}
N.~Bourbaki, {\it Alg\`ebre commutative, Chapitre 10, \'El\'ements de math\'ematiques}, 
Masson (1998), Paris.

\bibitem[Bru]{Br} W.~Bruns and J.~Herzog, {\it Cohen-Macaulay rings}, 
Cambridge studies in advanced mathematics {\bf n$^{\circ}$39}, Cambridge
University Press, Cambridge (1996).

\bibitem[CMo08]{CMo}
J.-Y.~Charbonnel and A.~Moreau, {\it Nilpotent bicone and characteristic
submodule of a reductive Lie algebra}, Tranformation Groups, \textbf{14}, (2008),
p. 319--360.

\bibitem[C20]{Ch}
J.-Y.~Charbonnel, {\it On some subspaces of the exterior algebra of a simple algebra},
Algebras and Representation Theory, \textbf{25(3)}, (2022), pp. 725--746.

\bibitem[Di74]{Di1} J.~Dixmier,
{\em Alg\`ebres enveloppantes}, Gauthier-Villars (1974).

\bibitem[Di79]{Di}
J.~Dixmier, {\it Champs de vecteurs adjoints sur les
groupes et alg\`ebres de Lie semi-simples}, Journal f\"ur die reine
und angewandte Mathematik, Band.\,{\bf 309} (1979), p. 183--190.

\bibitem[Ga-Gi06]{Ga}
W. L.~Gan, V.~Ginzburg, {\it Almost-commuting variety, ${\cal D}$-modules, and Cherednik 
algebras.}, International Mathematics Research Papers, {\bf 2}, (2006), p. 1--54. 

\bibitem[Gi12]{Gi}
V.~Ginzburg, {\it Isospectral commuting variety, the Harish-Chandra ${\cal D}$-module,
and principal nilpotent pairs}, Duke Mathematical Journal, {\bf 161}, (2012), 
p. 2023--2111.

\bibitem[Gro67]{Gro}
A.~Grothendieck, {\it Local cohomology}, Lecture Notes in Mathematics
{\bf n$^{\circ}$41} (1967), Springer-Verlag, Berlin, Heidelberg, New York.

\bibitem[H77]{Ha}
R.~Hartshorne, {\it Algebraic Geometry}, Graduate
Texts in Mathematics {\bf n$^{\circ}$52} (1977), Springer-Verlag, Berlin
Heidelberg New York.

\bibitem[HuWi97]{Hun}
G.~Huneke and R.~Wiegand, {\it Tensor products of modules, Rigidity and Local cohomology},
Mathematica Scandinavica, {\bf 81}, (1997), p. 161--183.

\bibitem[Ko63]{Ko}
B.~Kostant, {\it Lie group representations on polynomial rings}, American
Journal of Mathematics {\bf 85} (1963), p. 327--404.

\bibitem[MA86]{Mat}
H.~Matsumura, {\it Commutative ring theory} 
Cambridge studies in advanced mathematics {\bf n$^{\circ}$8} (1986), Cambridge
University Press, Cambridge, London, New York, New Rochelle, Melbourne,
Sydney.

\bibitem[MF78]{MF}
A.S.~Mishchenko and A.T.~Fomenko,
{\it Euler equations on Lie groups}, Math. USSR-Izv. {\bf 12} (1978), p. 371--389.

\bibitem[Mu88]{Mu}
D.~Mumford, {\it The Red Book of Varieties and Schemes}, Lecture Notes
in Mathematics	{\bf n$^{\circ}$1358} (1988), Springer-Verlag, Berlin,
Heidelberg, New York, London, Paris, Tokyo.

\bibitem[Po08]{Po}
V.L.~Popov,
{\it Irregular and singular loci of commuting varieties}, Transformation Groups
{\bf 13} (2008), p. 819--837.

\bibitem[Po08]{PV}
V.L.~Popov and E. B.~Vinberg, {\it Invariant Theory, in: Algebraic Geometry IV}, 
Encyclopaedia of MathematicalSciences {\bf n$^{\circ}$55} (1994), Springer-Verlag, Berlin,
p.123--284.

\bibitem[Ri79]{Ric}
R.~W.~Richardson, {\it Commuting varieties of semisimple Lie algebras and 
algebraic groups}, Compositio Mathematica {\bf 38} (1979), p. 311--322.

\bibitem[Sh94]{Sh}
I.R.~Shafarevich, {\it Basic algebraic geometry 2},
Springer-Verlag (1994), Berlin, Heidelberg, New York, London, Paris,
Tokyo, Hong-Kong, Barcelona, Budapest.

\bibitem[V72]{Ve}
F.D.~Veldkamp, {\it The center of the universal
enveloping algebra of a Lie algebra in characteristic $p$}, Annales
Scientifiques de L'\'Ecole Normale Sup\'erieure, {\bf 5}, (1972) , p. 217--240.

\end{thebibliography}
\end{document}